\documentclass[a4paper,12pt]{article}

\usepackage[english]{babel} 
\usepackage{amsmath,amsfonts,amssymb,amsthm,bbm,graphics,epsfig,psfrag,epstopdf}
\usepackage[active]{srcltx}
\usepackage{paralist,subfigure,float}
\usepackage{color,soul}
\usepackage[dvipsnames]{xcolor}
\usepackage{hyperref} 
\setstcolor{red}

\newtheorem{theorem}{Theorem}[section]
\newtheorem{corollary}[theorem]{Corollary}
\newtheorem{lemma}[theorem]{Lemma}
\newtheorem{proposition}[theorem]{Proposition}

\newtheorem{remark}[theorem]{Remark}

\newtheorem{conjecture}[theorem]{Conjecture}
\newtheorem{hypothesis}[theorem]{Hypothesis}

\def\N{\mathbb{N}}
\def\Z{\mathbb{Z}}

\def\R{\mathbb{R}}
\def\epsilon{\varepsilon}
\let\e=\varepsilon
\let\vp=\varphi
\let\vt=\vartheta
\let\t=\widetilde
\let\ol=\overline
\let\ul=\underline
\let\.=\cdot
\let\0=\emptyset
\let\mc=\mathcal

\def\O{\Omega}

\def\W{\mc{W}}
\def\hat{\widehat}
\def\tilde{\widetilde}

\DeclareMathOperator{\dist}{dist}

\DeclareMathOperator\supp{supp}

\def\thm#1{Theorem~\ref{thm:#1}}
\def\seq#1{(#1_n)_{n\in\N}}

\def\as{\quad\text{as }\;}

\def\tilde{\widetilde}
\def\1{\mathbbm{1}}
\def\Sph{\mathbb{S}^{N-1}}

\newenvironment{formula}[1]{\begin{equation}\label{#1}}{\end{equation}\noindent}
\def\Fi#1{\begin{formula}{#1}}
\def\Ff{\end{formula}\noindent}

\newcommand{\be}{\begin{equation}}
\newcommand{\ee}{\end{equation}}
\newcommand{\baa}{\begin{array}}
\newcommand{\eaa}{\end{array}}
\newcommand{\ba}{\begin{eqnarray}}
\newcommand{\ea}{\end{eqnarray}}

\numberwithin{equation}{section}

\begin{document}
\date{}
\title{\bf{Spreading, flattening and logarithmic lag\\ for reaction-diffusion equations in $\R^N$:\\
old and new results}}
\author{Fran\c cois Hamel$^{\hbox{\small{ a}}}$ and Luca Rossi$^{\hbox{\small{ b,c }}}$\thanks{This work has received funding from Excellence Initiative of Aix-Marseille Universit\'e~-~A*MIDEX, a French ``Investissements d'Avenir'' programme, and from the French ANR RESISTE (ANR-18-CE45-0019) and ReaCh (ANR-23-CE40-0023-02) projects. The first author is grateful to the hospitality of Universit\`a degli Studi di Roma La Sapienza, where part of this work was done. The authors are grateful to the anonymous reviewers for their valuable comments and suggestions.}\\
\\
\footnotesize{$^{\hbox{a }}$Aix Marseille Univ, CNRS, I2M, Marseille, France}\\
\footnotesize{$^{\hbox{b }}$SAPIENZA Univ Roma, Istituto ``G.~Castelnuovo'', Roma, Italy}\\
\footnotesize{$^{\hbox{c }}$CNRS, EHESS, CAMS, Paris, France}\\
}
\maketitle
	
\begin{abstract}
\noindent{}This paper is concerned with the large-time dynamics of bounded solutions of reaction-diffusion equations with bounded or unbounded initial support in $\R^N$. We~start with a survey of some old and recent results on the spreading speeds of the solutions and their asymptotic local one-dimensional symmetry. We then derive some flattening properties of the level sets of the solutions if initially supported on subgraphs. We also investigate the special case of asymptotically conical-shaped initial conditions. Lastly, we reclaim some known results about the logarithmic lag between the position of the solutions and that of planar or spherical fronts expan\-ding with minimal speed, for almost-planar or compactly supported initial conditions. We~then prove some new logarithmic-in-time estimates of the lag of the position of the solutions with respect to that of a planar front, for initial conditions which are supported on subgraphs with logarithmic growth at infinity. These estimates entail in particular  that the same lag as for compactly supported initial data holds true for a class of unbounded initial supports. The paper also contains some related conjectures and open problems.
\end{abstract}
	
\tableofcontents
	
		
\section{Introduction}\label{sec1}


\subsection{General framework}
	
In this paper, we are interested in the large-time dynamics of solutions of the Cauchy problem for the reaction-diffusion equation
\Fi{homo}
\partial_t u=\Delta u+f(u),\quad t>0,\ x\in\R^N,
\Ff
in any dimension $N\ge1$. The reaction term~$f:[0,1]\to\R$ is given, of class $C^1([0,1])$, with
$$f(0)=f(1)=0.$$
For mathematical convenience, we extend~$f$ by~$0$ in~$\R\setminus[0,1]$, and the extended function, still denoted $f$, is then Lipschitz continuous in $\R$.

The initial conditions $u_0$ are mostly assumed to be indicator functions of sets~$U$:
\be\label{defu0}
u_0(x)=\1_U(x):=\left\{\baa{ll}1&\hbox{if }x\in U,\vspace{3pt}\\ 0&\hbox{if }x\in\R^N\!\setminus\!U,\eaa\right.
\ee
where the ``initial support $U$"\footnote{\ We use the term ``initial support $U$", with an abuse of notation, to refer to the set~$U$ in the definition~\eqref{defu0} of the initial condition $u_0$. This set $U$ differs in general from the usual support $\supp(u_0)$ of $u_0$, which is defined as the complement of the largest open set of $\R^N$ where $u_0$ is equal to $0$ almost everywhere with respect to the Lebesgue measure. However, $U=\supp u_0$ if and only if $U$ is closed and the intersection of $U$ with any non-trivial ball centered at any point of $U$ has a positive Lebesgue measure.} is typically an unbounded measurable subset of $\R^N$ (although some results also deal with bounded sets $U$). Given $u_0$, there is then a unique bounded classical solution $u:(0,+\infty)\times\R^N\to[0,1]$ of~\eqref{homo} such that~$u(t,\cdot)\to u_0$ as~$t\to0^+$ in~$L^1_{loc}(\R^N)$.  Instead of initial conditions $u_0=\1_U$, more general initial conditions $0\leq u_0\leq 1$ could have been considered, whose upper level set $\{x\in\R^N:u_0(x)\geq\theta\}$ lies at bounded Hausdorff distance from $\supp(u_0)$, where $\theta\in(0,1)$ is a suitable value depending on $f$. For the sake of simplicity of the presentation and readability of the paper, we kept the assumption~$u_0=\1_U$ in the main new results, all the more as this case already gives rise to many interesting and non-trivial results.

From the strong parabolic maximum principle, the solution $u$ of~\eqref{homo}-\eqref{defu0} satisfies
$$0<u<1\ \hbox{ in }(0,+\infty)\times\R^N,$$
provided the Lebesgue measures of $U$ and~$\R^N\setminus U$ are positive. However, from parabolic estimates, at each time $t>0$, $u$ stays close to $1$ or $0$ in subregions of $U$ or $\R^N\setminus U$ that are far away from $\partial U$.


\subsection{Main questions and outline of the paper}

The main goal of the paper is to discuss the large-time properties of the level sets of the solutions $u$. For $\lambda\in(0,1)$, the level set of $u$ at time $t>0$ with level $\lambda$ is defined by
\be\label{defElambda0}
E_\lambda(t):=\big\{x\in\R^N:u(t,x)=\lambda\},
\ee
and its upper level set is given by
\be\label{defFlambda}
F_\lambda(t):=\big\{x\in\R^N:u(t,x)>\lambda\}.
\ee
We are then interested in the properties of $E_\lambda(t)$ and $F_\lambda(t)$ as $t\to+\infty$. More precisely, we want to know where these level sets or upper level sets are located and how they look~like.

A typical question is the existence of a {\em spreading speed}, which, for a given vector $e\in\R^N$ with unit Euclidean norm (that is, $|e|=1$), is a quantity~$w(e)>0$ such that
\be\label{spreading0}\left\{\begin{array}{lll}
u(t,cte)\to1 & \hbox{as $t\to+\infty$} & \hbox{for every $0\le c<w(e)$},\vspace{3pt}\\
u(t,cte)\to0 & \hbox{as $t\to+\infty$} & \hbox{for every $c>w(e)$.}\end{array}\right.
\ee
Provided that $w(e)$ is finite, that would imply that
$$\lim_{t\to+\infty}\Big(\min_{x\in E_\lambda(t)\cap\R^+e}\frac{x\cdot e}{t}\Big)=\lim_{t\to+\infty}\Big(\max_{x\in E_\lambda(t)\cap\R^+e}\frac{x\cdot e}{t}\Big)=w(e)$$
for every $\lambda\in(0,1)$, meaning that all level sets with values in $(0,1)$ move with asymptotic speed $w(e)$ in the direction $e$, where we set $\R^+e:=\{\tau e:\tau>0\}$ with $\R^+:=(0,+\infty)$. It follows that if $w(e)$ exists and is finite in every direction $e$, then, for any $\lambda\in(0,1)$, the rescaled upper level sets $t^{-1}F_\lambda(t)$ approach as $t\to+\infty$, in a suitable sense, the envelop set of the spreading speeds $w(e)$, i.e., the set
\be\label{W}
\W:=\big\{re\,:\,e\in\R^N,\ |e|=1,\ \ 0\leq r<w(e)\big\}.
\ee
This issue and related problems are discussed in Sections~\ref{sec23} and~\ref{sec24}. In particular, in these sections, we recall some old and recent results, including those of~\cite{HR1}, on the existence and characterization, by Freidlin-G\"artner type formulas, of the spreading speeds $w(e)$ and the spreading sets $\W$ for some classes of reactions $f$ and initial conditions, either with bounded or unbounded supports.

Next, once the spreading speed is shown to exist, what can it be said about the possible {\em gaps} between the actual position of the level set $E_\lambda(t)$ and the dilated envelop set $t\mc{W}$ along the ray $\R^+e$, that is, how to estimate
\be\label{gaps}
\min_{x\in E_\lambda(t)\cap\R^+e}\big(x\cdot e-w(e)t\big)\ \hbox{ and }\ \max_{x\in E_\lambda(t)\cap\R^+e}\big(x\cdot e-w(e)t\big)
\ee
as $t\to+\infty$? They are $o(t)$ as $t\to+\infty$ if~\eqref{spreading0} holds and $w(e)$ is finite, but can we say more? Known results related to this issue for solutions with compact or almost-planar initial support are recalled in Section~\ref{seclag1}. New results on the lag~\eqref{gaps} for solutions initially supported on subgraphs having logarithmic growth at infinity, with reactions $f$ of the KPP type~\eqref{kpp0} or~\eqref{fkpp} below, are presented in Section~\ref{seclag2}, with proofs of the new statements given in Section~\ref{sec6}.

Another typical question is to know whether, for $\lambda\in(0,1)$ and for a sequence $(t_n,x_n)_{n\in\N}$ in $(0,+\infty)\times\R^N$ such that $x_n\in E_\lambda(t_n)$ and $t_n\to+\infty$ as $n\to+\infty$, the level sets $E_\lambda(t_n)$ become asymptotically locally flat around $x_n$ as $n\to+\infty$. By that we mean that there is a family of affine hyperplanes  $(H_n)_{n\in\N}$ in $\R^N$ such that, for any $R>0$, the Hausdorff distance between $E_\lambda(t_n)\cap B_R(x_n)$ and $H_n\cap B_R(x_n)$ tends to $0$ as $n\to+\infty$ (see the notations and the definition below for the open Euclidean balls $B_r(x)$ and the Hausdorff distance). If yes, we then speak of {\em local flattening} of the level sets. Known results of~\cite{J} on this topic and the related notion of asymptotic local one-dimensional symmetry for solutions with general reactions $f$ and bounded initial supports, and of~\cite{HR2} for solutions with KPP type reactions~\eqref{fkpp} and unbounded convex-like initial supports, are recalled in Section~\ref{sec25}. New results for more general reactions $f$ and solutions initially supported on subgraphs are stated in Section~\ref{sec3}. The proofs of the new results are carried out in Section~\ref{sec5}.

The answer to the above questions shall strongly depend on the given function $f$, on the initial support $U$ of $u$, as well as on the dimension $N$. We first presenting and discuss some standard hypotheses in Sections~\ref{sec21} and~\ref{sec22}. 
 

\subsection{Some notations}

Throughout the paper, ``$|\ |$'' and ``$\ \cdot\ $'' denote respectively the Euclidean norm and inner product in $\R^N$,
$$B_r(x):=\{y\in\R^N:|y-x|<r\}$$
is the open Euclidean ball of center $x\in\R^N$ and radius $r>0$, $B_r:=B_r(0)$, and $\Sph:=\{e\in\R^N:|e|=1\}$ is the unit Euclidean sphere of~$\R^N$. The distance of a point $x\in\R^N$ to a set $A\subset\R^N$ is given by
$$\dist(x,A):=\inf\big\{|y-x|:y\in A\big\},$$
with the convention $\dist(x,\emptyset)=+\infty$. The Hausdorff distance between two subsets $A,B\subset\R^N$ is given by
$$d_{\mc{H}}(A,B):=\max\Big(\sup_{x\in A}\dist(x,B),\,\sup_{y\in B}\dist(y,A)\Big),$$
with the conventions that $d_{\mc{H}}(A,\emptyset)=d_{\mc{H}}(\emptyset,A)=+\infty$ if $A\neq\emptyset$ and $d_{\mc{H}}(\emptyset,\emptyset)=0$. 

For $x\in\R^n\setminus\{0\}$, we set
$$\hat{x}:=\frac{x}{|x|}.$$
Lastly, we call $(\mathrm{e}_1,\cdots,\mathrm{e}_N)$ the canonical basis of $\R^N$, that is,
$$\mathrm{e}_i=(0,\cdots,0,1,0,\cdots,0)$$
for $1\le i\le N$, where $1$ is the $i$-th coordinate of $\mathrm{e}_i$.


\section{Standard hypotheses and known results}\label{sec2}

Before stating the main new results in Sections~\ref{sec3} and~\ref{sec4}, we introduce some important hypotheses which are satisfied in standard situations. The hypotheses are expressed in terms of the solutions of~\eqref{homo} with more general initial conditions than indicator functions, or are expressed in terms of the function $f$ solely. We then discuss the logical link between these hypotheses and we review some known spreading and flattening results for compactly supported or more general initial data.


\subsection{Invasion property}\label{sec21}

Both $0$ and $1$ are steady states of~\eqref{homo}, since $f(0)=f(1)=0$, and we consider a non-symmetric situation in which, say, the state $1$ is more attractive than $0$, in the sense that it attracts the solutions of~\eqref{homo} --~not necessarily satisfying~\eqref{defu0}~-- that are ``large enough" in large balls at initial~time. 

\begin{hypothesis}\label{hyp:invasion}
The {\rm invasion property} occurs for any solution $u$ of~\eqref{homo} with a ``large enough'' initial datum $u_0$, that~is, there exist $\theta\in(0,1)$ and $\rho>0$ such that if
\be\label{hyptheta}
\theta\,\1_{B_\rho(x_0)}\le u_0\le1\ \hbox{ in }\R^N,
\ee
for some $x_0\in\R^N$, then $u(t,x)\to1$ as $t\to+\infty$, locally uniformly with respect to $x\in\R^N$. We then say that $u$ is an {\rm invading solution}.
\end{hypothesis}
	
If $f$ is such that
\Fi{HTconditions}
f>0 \text{ \ in $(0,1)$\quad and \quad }\liminf_{s\to0^+}\frac{f(s)}{s^{1+2/N}}>0,
\Ff 
then Hypothesis~\ref{hyp:invasion} is satisfied with any $\theta\in(0,1)$ and $\rho>0$, and this property is known as the ``hair trigger effect'', see~\cite{AW}. If $f>0$ in~$(0,1)$ (without any further assumption on the behavior of $f$ at $0^+$), then Hypothesis~\ref{hyp:invasion} still holds with any $\theta\in(0,1)$, and with $\rho>0$ large enough (depending on $\theta$). Actually, Hypothesis~\ref{hyp:invasion} holds as well if~$f$ is of the ignition type, that~is,
\be\label{ignition}
\exists\,\alpha\in(0,1),\ \ f=0\hbox{ in $[0,\alpha]$ and $f>0$ in $(\alpha,1)$},
\ee
and $\theta$ in Hypothesis~\ref{hyp:invasion} can be any real number in the interval $(\alpha,1)$, provided $\rho>0$ is large enough (depending on $\theta$). For a bistable function $f$, i.e.
\be\label{bistable}
\exists\,\alpha\in(0,1),\ \ f<0\hbox{ in }(0,\alpha)\hbox{ and }f>0\hbox{ in }(\alpha,1),
\ee
Hypothesis~\ref{hyp:invasion} is fulfilled if and only if $\int_0^1f(s)\,ds>0$, see~\cite{AW,FM}, and in that case $\theta$ can be any real number in $(\alpha,1)$, provided $\rho>0$ is large enough (depending on $\theta$). Notice that the validity of Hypothesis~\ref{hyp:invasion} and the choice of $\theta$ when $f$ is of the ignition type or is just positive in $(0,1)$ can be viewed as a consequence of the aforementioned property in the bistable case and of the comparison principle, by putting below $f$ a suitable bistable function with positive integral over $[0,1]$. For a tristable function $f$, namely
\be\label{tristable}
\exists\,0<\alpha<\beta<\gamma<1,\quad f<0\hbox{ in $(0,\alpha)\cup(\beta,\gamma)\ $ and $\ f>0$ in $(\alpha,\beta)\cup(\gamma,1)$},
\ee
then it follows from~\cite{FM} that Hypothesis~\ref{hyp:invasion} is fulfilled if and only if both integrals $\int_{\beta}^1f$ and $\int_0^1f$ are positive, and, for such a function, these positivity conditions are in turn equivalent to the positivity of $\int_t^1f$ for every $t\in[0,1)$.

More generally speaking, it actually turns out from~\cite{DP,P3} that Hypothesis~\ref{hyp:invasion} is equivalent to the following two simple simultaneous conditions on the function~$f$:
\be\label{hyp:theta}
\exists\,\theta\in(0,1),\quad f>0\hbox{ in }[\theta,1),
\ee
and
\be\label{intf}
\forall\,t\in[0,1),\quad\int_t^1f(s)\,ds>0.
\ee
Furthermore, $\theta$ can be chosen as the same real number in Hypothesis~$\ref{hyp:invasion}$ and in~\eqref{hyp:theta}. More precisely, the fact that Hypothesis~\ref{hyp:invasion} implies~\eqref{hyp:theta}-\eqref{intf} follows from~\cite[Proposition~2.12]{P3}, while the converse implication follows from~\cite[Lemma~2.4]{DP}. In particular, Hypothesis~\ref{hyp:invasion} is satisfied if $f\ge0$ in $[0,1]$ and if condition~\eqref{hyp:theta} holds. Condition~\eqref{hyp:theta} alone is not enough to guarantee Hypothesis~\ref{hyp:invasion}, as shown by bistable functions $f$ of the type~\eqref{bistable} with~$\int_0^1f\le0$. Similarly, condition~\eqref{intf} alone is not enough to guarantee Hypothesis~\ref{hyp:invasion}, since there are $C^1([0,1])$ functions $f$ which vanish at $0$ and $1$ and satisfy~\eqref{intf} but not~\eqref{hyp:theta}: consider for instance $f$ defined by $f(1)=0$ and $f(s)=s(1-s)^3\sin^2(1/(1-s))$ for $s\in[0,1)$.

It also follows from the equivalence between conditions~\eqref{hyp:theta}-\eqref{intf} and Hypothesis~\ref{hyp:invasion} that the latter is independent of the dimension $N$. On the other hand, for a function $f$ which is positive in $(0,1)$, the validity of the hair trigger effect (that is, the arbitrariness of $\theta\in(0,1)$ and $\rho>0$ in Hypothesis~\ref{hyp:invasion}) does depend on~$N$. For instance, for the function $f(s)=s^p(1-s)$ with $p\ge1$, Hypothesis~\ref{hyp:invasion} holds in any dimension $N\ge1$, but the hair trigger effect holds if and only if $p\le1+2/N$, see~\cite{AW}.


\subsection{Planar traveling fronts}\label{sec22}

In the large-time dynamics of bounded solutions of the reaction-diffusion equation~\eqref{homo}, a crucial role is played by particular solutions, called {\em planar traveling fronts}, which connect the steady states $1$ and $0$. These are solutions of the form
$$u(t,x)=\vp(x\cdot e-ct)$$
with $c\in\R$, $e\in\Sph$, and
\be\label{limitsvp}
0=\vp(+\infty)<\vp(z)<\vp(-\infty)=1\quad \text{for all $z\in\R$}.
\ee
The level sets of these solutions are parallel hyperplanes orthogonal to $e$ traveling with the constant speed $c$ in the direction~$e$. If a planar traveling front solution exists, its profile $\vp$ solves the ODE
$$\vp''+c\vp'+f(\vp)=0\ \hbox{ in $\R$},$$
and it is necessarily decreasing and unique up to shifts, for a given speed~$c$, see e.g.~\cite{HR1}.

The second main hypothesis used in the present paper is concerned with the existence of planar traveling front solutions connecting $1$ to $0$.

\begin{hypothesis}\label{hyp:minimalspeed}
For any direction $e\in\Sph$, equation~\eqref{homo} admits a planar traveling front solution $u(t,x)=\vp(x\cdot e-c_0t)$ connecting $1$ to~$0$ in the sense of~\eqref{limitsvp}, with positive speed $c_0>0$.
\end{hypothesis}

Notice that Hypothesis~\ref{hyp:minimalspeed} is equivalent to the existence of a traveling front solution $\vp(x-c_0t)$ connecting $1$ to $0$ with $c_0>0$ for the one-dimensional version of~\eqref{homo}. Hypothesis~\ref{hyp:minimalspeed} thus depends on the function $f$ only, and not on the dimension $N$, as it is the case for Hypothesis~\ref{hyp:invasion}. Hypothesis~\ref{hyp:minimalspeed} is fulfilled for instance if $f>0$ in $(0,1)$, or if $f$ is of the ignition type~\eqref{ignition}, or if $f$ is of the bistable type~\eqref{bistable} with $\int_0^1f(s)\,ds>0$ (in the last two cases, the speed $c_0$ is unique), see~\cite{AW,FM,F,KPP}. Hypothesis~\ref{hyp:minimalspeed} is also satisfied for some functions~$f$ having multiple oscillations in the interval $[0,1]$. For instance, for a tristable function $f$ satisfying~\eqref{tristable}, there exist unique speeds $c_1$ and $c_2$ of one-dimensional fronts $\varphi_1(x-c_1t)$ and $\varphi_2(x-c_2t)$ such that
$$0=\varphi_1(+\infty)<\varphi_1(z)<\varphi_1(-\infty)=\beta$$
and
$$\beta=\varphi_2(+\infty)<\varphi_2(z)<\varphi_2(-\infty)=1$$
for all $z\in\R$, and Hypothesis~\ref{hyp:minimalspeed} is fulfilled if and only if $c_1<c_2$ and $\int_0^1f(s)\,ds>0$, see~\cite{FM} (furthermore, in that case, $c_0$ is unique and $c_1<c_0<c_2$). It also follows from~\cite[Proposition~1.1]{DLL} that Hypothesis~\ref{hyp:minimalspeed} is satisfied if the following condition holds: $\int_0^1f(s)ds>0$, and $f(s)>0$ for all $s\in(0,1)$ such that $\int_0^sf(\tau)d\tau>0$. In particular, for a tristable function $f$ of the type~\eqref{tristable}, the latter condition means that $\int_0^1f(s)ds>0$ and $\int_0^\beta f(s)ds\le0$ (hence, $\int_\beta^1f(s)ds>0$), which in turn yields $c_1\le 0<c_2$, and Hypothesis~\ref{hyp:minimalspeed} is thus well fulfilled.

As a matter of fact, it turns out that Hypothesis~\ref{hyp:minimalspeed} is equivalent to the existence of a positive minimal speed $c^*$ of traveling fronts connecting $1$ to $0$, that is, the existence of $c^*>0$ such that~\eqref{homo} in $\R$ admits a solution of the form $\vp^*(x-c^*t)$ satisfying~\eqref{limitsvp} with~$\vp^*$ instead of $\vp$, and it does not admit any solution of the same type with $c\in(-\infty,c^*)$ instead of $c^*$ (thus, necessarily, $c^*\le c_0$), see~\cite[Lemma~3.5]{HR1}. If $f>0$ in $(0,1)$, then the set of admissible speeds of planar traveling fronts connecting $1$ to $0$ is  equal to the whole interval $[c^*,+\infty)$, and moreover $c^*\ge2\sqrt{f'(0)}$ with equality if (but not only if) $f$ further satisfies the so-called KPP condition: $f(s)\le f'(0)s$ for all $s\in[0,1]$, see~\cite{AW,FM,F,KPP}. On the other hand, if $f$ is of the ignition type~\eqref{ignition}, or if $f$ is of the bistable type~\eqref{bistable} with $\int_0^1f(s)\,ds>0$, then $c^*=c_0$, since $c_0$ is unique in these two cases.


\subsection{Known spreading results for localized or general initial data}\label{sec23}

In this section, we first establish the link between Hypotheses~\ref{hyp:invasion} and~\ref{hyp:minimalspeed} and we recall some classical results on the spreading of solutions with initial bounded supports. We then present some recent results of~\cite{HR1} on the spreading speeds and spreading sets of solutions with general initial supports, and we finish with some counter-examples to the estimates.

\subsubsection{Relationship between Hypotheses~\ref{hyp:invasion} and~\ref{hyp:minimalspeed}, and spreading speeds for solutions with bounded initial supports}

Hypothesis~\ref{hyp:invasion} is concerned with a property satisfied by the solutions of the Cauchy pro\-blem~\eqref{homo} with large enough initial conditions, whereas Hypothesis~\ref{hyp:minimalspeed} is related to the existence of some special entire (defined for all times $t\in\R$) solutions of~\eqref{homo} having flat level sets moving with constant speed. It is therefore not clear to see how these two pro\-perties could be related. However, it turns out that Hypothesis~\ref{hyp:minimalspeed} implies Hypothesis~\ref{hyp:invasion}, as follows from~\cite[Lemma~3.4]{HR1} together with~\cite[Lemma~2.4]{DP}. This implication can also be derived from \cite[Theorem~1.5]{DR} under the additional condition that there is~$\delta>0$ such that $f$ is nonincreasing in~$[0,\delta]$ and in~$[1-\delta,1]$. Furthermore, under Hypothesis~\ref{hyp:minimalspeed}, the following properties hold:
any solution $u$ as in Hypothesis~\ref{hyp:invasion} satisfies
\be\label{c<c*}
\forall\,c\in[0,c^*),\quad\min_{|x|\le ct}u(t,x)\to1\as t\to+\infty,
\ee
while, if the initial datum $u_0$ is compactly supported, then
\Fi{c>c*}\forall\,c>c^*,\quad
\sup_{|x|\geq ct}u(t,x)\to0\as t\to+\infty,
\Ff
where $c^*>0$ is the minimal speed of planar traveling fronts connecting $1$ to $0$ (given in the last paragraph of Subsection~\ref{sec22}), see~\cite[Proposition~1.3]{HR1}. These properties imply that, under Hypothesis~\ref{hyp:minimalspeed}, solutions with compactly supported, but large enough, initial data have a spreading speed $w(e)$ in any direction $e\in\Sph$, in the sense of~\eqref{spreading0}, and moreover $w(e)=c^*$ for all $e\in\Sph$. This answers the first question mentioned in Section~\ref{sec1}. But we point out that the properties~\eqref{c<c*}-\eqref{c>c*} are stronger than~\eqref{spreading0}, in that they include a uniformity with respect to the directions and with respect to the speeds smaller or larger than $c^*\pm\varepsilon$ for any $\varepsilon>0$ small enough. The properties~\eqref{c<c*}-\eqref{c>c*}, which hold under Hypothesis~\ref{hyp:minimalspeed}, can be viewed as a natural extension of some results of the seminal paper~\cite{AW}, which were originally obtained under more specific assumptions on $f$, especially of the type~\eqref{HTconditions},~\eqref{ignition}, or~\eqref{bistable} with $\int_0^1f(s)ds>0$.

Whereas Hypothesis~\ref{hyp:minimalspeed} implies Hypothesis~\ref{hyp:invasion}, the converse implication is false in general. For instance, consider equation~\eqref{homo} in dimension $N=1$ with a tristable function~$f$ satisfying~\eqref{tristable} and such that $\int_0^\beta f>0$ and $\int_\beta^1f>0$, and let $c_1$ and $c_2$ be the unique (positive) speeds of the traveling fronts $\vp_1(x-c_1t)$ and $\vp_2(x-c_2t)$ connecting $\beta$~to~$0$, and~$1$ to~$\beta$, respectively. It follows from~\cite{FM} that, if~$c_1\ge c_2$, then Hypothesis~\ref{hyp:minimalspeed} is not satisfied, while Hypothesis~\ref{hyp:invasion} is, from~\cite{DP,FM}. Furthermore, if $c_1>c_2$, then the invading solutions $u$ emanating from compactly supported initial conditions $u_0$ as in Hypothesis~\ref{hyp:invasion} develop into a terrace of two expanding fronts with speeds~$c_1$ and~$c_2$, in the sense that
\be\label{terrasse}\left\{\begin{array}{lll}
\displaystyle\inf_{B_{ct}}u(t,\cdot)\to1 & \hbox{as $t\to+\infty$} & \hbox{if $0<c<c_2$},\vspace{3pt}\\
\displaystyle\sup_{B_{c''t}\setminus B_{c't}}|u(t,\cdot)-\beta|\to0 & \hbox{as $t\to+\infty$} & \hbox{if $c_2<c'<c''<c_1$},\vspace{3pt}\\
\displaystyle\sup_{\R^N\setminus B_{ct}}u(t,\cdot)\to0 & \hbox{as $t\to+\infty$} & \hbox{if $c>c_1$},\end{array}\right.
\ee
see~\cite{DM2,FM}. In particular, the existence of $w(e)$ satisfying~\eqref{spreading0} fails in that case. We refer to~\cite{DM2,DGM,GR,P2,P3} for more results on propagating terraces in more general frameworks. We also point out that, under Hypothesis~\ref{hyp:minimalspeed}, properties~\eqref{c<c*}-\eqref{c>c*} give the exact spreading speed of the solutions $u$ with compactly supported and large enough initial conditions. However, under the sole Hypothesis~\ref{hyp:invasion}, property~\eqref{c<c*} is still fulfilled, for a certain positive speed $c^*$ (which nevertheless may not be any speed of a traveling front solution connecting~$1$ to~$0$): indeed, if $v$ denotes the solution to~\eqref{homo} with initial condition $v_0:=\theta\,\1_{B_\rho(x_0)}\le u_0$, then $v(t,\cdot)\to1$ as $t\to+\infty$ locally uniformly in $\R^N$ by Hypothesis~\ref{hyp:invasion} and there exists $T>0$ such that
$$1\ge u(T,\cdot+y)\ge v(T,\cdot+y)\ge v_0\hbox{ in $\R^N$ for every~$|y|\leq1$},$$
whence $1\ge u(kT+t,\cdot+ky)\ge v(kT+t,\cdot+ky)\ge v(t,\cdot)$ in $\R^N$ for all $k\in\N$, $t\geq0$, and $|y|\le1$ by immediate induction. This entails~\eqref{c<c*} with $c^*:=1/T$.

\subsubsection{Spreading speeds and spreading sets for solutions with general unbounded initial supports}

While~\eqref{c<c*}-\eqref{c>c*} give, under Hypothesis~\ref{hyp:minimalspeed}, the existence and characterization of the spreading speed $w(e)=c^*$ for any $e\in\Sph$ in the sense of~\eqref{spreading0} for the solutions  of~\eqref{homo} with compactly supported and large enough initial data, the situation is much more intricate when the initial condition $u_0=\1_U$ in~\eqref{defu0} has an unbounded initial support $U$. We already know that, if $U$ contains a ball of radius $\rho$, with $\rho>0$ given by Hypothesis~\ref{hyp:invasion} (following from Hypothesis~\ref{hyp:minimalspeed}), then the spreading speed in a direction $e\in\Sph$, if any, in the sense of~\eqref{spreading0}, necessarily satisfies $w(e)\ge c^*$, thanks to~\eqref{c<c*} and the comparison principle. To show the existence and provide formulas of the spreading speeds for an arbitrary set~$U$, we introduced in~\cite{HR1} some notions of sets of directions ``around which~$U$ is bounded'' and ``around which~$U$ is unbounded'', respectively defined as follows:
$$\left\{\begin{array}{l}
\displaystyle\mc{B}(U):=\Big\{\xi\in\Sph:\liminf_{\tau\to+\infty}\frac{\dist(\tau\xi,U)}{\tau}>0\Big\},\vspace{3pt}\\
\displaystyle\mc{U}(U):=\Big\{\xi\in\Sph:\lim_{\tau\to+\infty}\frac{\dist(\tau\xi,U)}{\tau}=0\Big\}.\end{array}\right.$$
These sets are respectively open and closed relatively to $\Sph$. A direction $\xi$ belongs to~$\mc{B}(U)$ if and only if there is an open cone $\mc{C}$ containing the ray~$\R^+\xi$ such that $U\cap\mc{C}$ is bounded. On the other hand, $\xi\in\mc{U}(U)$ if $\R^+\xi\setminus U$ is bounded and only if for any open cone~$\mc{C}$ containing the ray~$\R^+\xi$, the set $U\cap\mc{C}$ is unbounded. For any $\delta>0$, we denote
$$U_\delta:=\big\{x\in U:\dist(x,\partial U)\ge\delta\big\}.$$
One of the main results of~\cite{HR1} is that, under Hypothesis~\ref{hyp:minimalspeed} (which implies that Hypo\-thesis~\ref{hyp:invasion} holds for some $\rho>0$), if $U_\rho\neq\emptyset$ and
\Fi{hyp:U}
\mc{B}(U)\cup\,\mc{U}(U_\rho)=\Sph,
\Ff
then the solution $u$ of~\eqref{homo}-\eqref{defu0} admits a continuous (from $\Sph$ to $[c^*,+\infty]$) family of spreading speeds $e\mapsto w(e)$ in the sense of~\eqref{spreading0}, and even in the uniform sense
\Fi{ass-cpt}\left\{\baa{lll}
\displaystyle\lim_{t\to+\infty}\,\Big(\min_{x\in C}u(t, tx)\Big)\!&\! =1 & \text{if }C\subset\mc{W},\vspace{3pt}\\
\displaystyle\lim_{t\to+\infty}\,\Big(\max_{x\in C}u(t,tx)\Big)\!&\!=0 & \text{if }C\subset\R^N\setminus\ol{\mc{W}}\eaa\right.
\Ff
for any compact set $C\subset\R^N$, with $\W$ being the envelop set of $w(e)$ as defined in~\eqref{W}. In~addition, the set $\W$ is explicitly given by
\be\label{asspre}
\W=\R^+\, \mc{U}(U)\,+\,B_{c^*}
\ee
(with the convention $\emptyset+B_{c^*}:=B_{c^*}$), which implies that the spreading speed $w(e)$ in any direction $e$ is given by the equivalent formulas
\Fi{FGgeneral}
w(e)=\sup_{\xi\in\mc{U}(U),\ \xi\.e\ge0}\frac{c^*}{\sqrt{1-(\xi\.e)^2}}=\frac{c^*}{\dist(e,\R^+\, \mc{U}(U))}\ \in[c^*,+\infty],
\Ff
with the conventions $w(e)=c^*$ if there is no $\xi\in\mc{U}(U)$ such that $\xi\cdot e\ge0$, and $c^*/0=+\infty$ (in particular, $w(e)=c^*$ if $\mc{U}(U)=\emptyset$). The above results are contained in~\cite[{Theorems~2.1-2.2}]{HR1}.

\subsubsection{Estimates on the upper level sets $F_\lambda(t)$ defined in~\eqref{defFlambda}}

Formula~\eqref{FGgeneral} can be viewed as a Freidlin-G\"artner type formula, as these authors were the first ones to derive in~\cite{FG} a variational formula for the spreading speeds of solutions with compact initial supports, in the context of spatially periodic reaction-diffusion equations of the Fisher-KPP type (the formula derived from~\cite{FG} and~\cite{BHN,BHN1,R1,W} actually involves further notions of minimal speeds of pulsating fronts in suitable directions). However, while the anisotropy of the original Freidlin-G\"artner formula is a result of the spatial heterogeneity of the equation, the one in formula~\eqref{FGgeneral} above reflects the shape of the initial support of the solution. Formula~\eqref{ass-cpt} means that the envelop set $\mc{W}$ of the spreading speeds $w(e)$ is a  {\em spreading set} for the solution~$u$ of~\eqref{homo}-\eqref{defu0}. Furthermore,~\eqref{asspre} implies that $\mc{W}$ is the open $c^*$-neighborhood of the positive cone generated by the directions $\mc{U}(U)$ (it is therefore either unbounded, when $\mc{U}(U)\neq\emptyset$, or it coincides with $B_{c^*}$). Notice that $\mc{W}$ is not convex in general (for instance, if~$U\neq\emptyset$ is a non-convex closed cone, say with vertex~$0$, then~$\R^+\mc{U}(U)=U\setminus\{0\}$ and thus  $\W$ is not convex either). Nevertheless, if $U$ is convex or if there is a convex set $U'$ such that $d_{\mc{H}}(U,U')<+\infty$, then
$$\R^+\mc{U}(U)\cup\{0\}=\R^+\mc{U}(U')\cup\{0\}$$
is convex and $\W$ is convex too. It also follows from the formulas~\eqref{ass-cpt}-\eqref{asspre} that the rescaled upper level sets $t^{-1}F_\lambda(t)$, as defined in~\eqref{defFlambda}, converge locally to the spreading set $\mc{W}$, in the sense that, for any $R>0$ and any $\lambda\in(0,1)$,
\Fi{Hloc}
d_{\mc{H}}\big(\ol{B_R}\cap t^{-1}F_\lambda(t)\,,\,\ol{B_R}\cap \mc{W}\big)\to0\ \text{ as }t\to+\infty,
\Ff
see~\cite[Theorem~2.3]{HR1}. The above convergence is not global in general, that is, it does not hold in general without the intersection with the balls $\ol{B_R}$, see~\cite[Proposition~6.5]{HR1}. However, still under Hypothesis~\ref{hyp:minimalspeed}, if instead of~\eqref{hyp:U} one assumes that 
\Fi{dUrho}
d_{\mc{H}}(U,U_\rho)<+\infty,
\Ff
then the rescaled upper level sets $t^{-1}F_\lambda(t)$ globally approach the $c^*$-neighborhood of the rescaled initial supports $t^{-1}U$, in the sense that
\Fi{dH}
d_{\mc{H}}\big(t^{-1}F_\lambda(t) \,,\, t^{-1}U+B_{c^*}\big)\to0\ \text{ as }t\to+\infty,
\Ff
see~\cite[Theorem~2.4]{HR1}.
The role of condition~\eqref{dH} is cutting off regions of $U$ which play a negligible role in the large-time behavior of the solution. 

To illustrate the definitions and properties of the sets $\mc{B}(U)$, $\mc{U}(U)$ and $\mc{W}$ defined above, and other applications of the previous results, consider the case when $U$ is a subgraph
\be\label{Ugamma}
U=\big\{x=(x',x_N)\in \R^{N-1}\times\R:x_N\le\gamma(x')\big\},
\ee
with $\gamma\in L^\infty_{loc}(\R^{N-1})$. Thus, $U_\rho\neq\emptyset$ for any $\rho>0$. If $\gamma(x')/|x'|\to\alpha\in\R$ as $|x'|\to+\infty$, then
$$\mc{B}(U)=\big\{e=(e',e_N)\in\Sph:e_N>\alpha|e'|\big\}$$
and
$$\mc{U}(U)=\mc{U}(U_\rho)=\big\{e\in\Sph:e_N\leq\alpha|e'|\big\}.$$
Hence~\eqref{hyp:U} is fulfilled and then, under Hypothesis~\ref{hyp:minimalspeed}, the previous results hold. Nevertheless, the shape of the spreading set $\mc{W}$ given by~\eqref{asspre} changes according to $\alpha$: if $\alpha>0$, then
\be\label{Wshift}
\mc{W}=\big\{x=(x',x_N)\in\R^N:x_N< \alpha\,|x'|+c^*\sqrt{1+\alpha^2}\big\}
\ee
(a shift of the interior of the cone $\R^+\mc{U}(U)$); if~$\alpha<0$, then~$\mc{W}$ is still the $c^*$-neighborhood of the cone $\R^+\mc{U}(U)$, but $\mc{W}$ is now $C^1$ and convex, and~$w(e)=c^*$ if~$e_N\ge|e'|/|\alpha|$; if~$\alpha=0$, then
$$\mc{W}=\{x\in\R^N:x_N<c^*\},$$
$w(e)=+\infty$ if~$e_N\le0$, and~$w(e)=c^*/e_N$ if~$e_N>0$. Now, if $\gamma(x')/|x'|\to-\infty$ as $|x'|\to+\infty$, then
$$\mc{B}(U)\!=\!\Sph\setminus\{-\mathrm{e}_N\},\ \ \mc{U}(U)\!=\!\mc{U}(U_\rho)\!=\!\{-\mathrm{e}_N\},$$
and
$$\mc{W}=-\R^+\mathrm{e}_N+B_{c^*}=\big\{x\in\R^N:|x'|<c^*,\ x_N\leq0\big\}\cup B_{c^*}.$$
On the other hand, if $\gamma(x')/|x'|\to+\infty$ as $|x'|\to+\infty$, then $\mc{B}(U)=\emptyset$, $\mc{U}(U)\!=\!\mc{U}(U_\rho)\!=\!\Sph$, and $\mc{W}=\R^N$.

\subsubsection{Examples and counter-examples}

Several additional comments on these spreading properties are in order. First of all, the existence of spreading speeds satisfying~\eqref{spreading0}, as well as the formulas~\eqref{ass-cpt},~\eqref{Hloc} or~\eqref{dH}, do not hold in general without Hypothesis~\ref{hyp:minimalspeed}, as follows for instance from~\cite{FM} and~\eqref{terrasse} for some tristable functions of the type~\eqref{tristable} with $c_1>c_2$. We also point out that, on the one hand, the geometric assumption~\eqref{hyp:U} is invariant under rigid transformations of~$U$ and is fulfilled in particular when $d_{\mc{H}}(U,U_\rho)<+\infty$ and in addition $U$ is star-shaped or $\mc{B}(U')\cup\mc{U}(U')=\Sph$ for some $U'$ such that $d_{\mc{H}}(U,U')<+\infty$, see~\cite[Proposition~5.1]{HR1}. On the other hand, a sufficient condition for~\eqref{dUrho} to hold is that the set $U$ fulfills the uniform interior sphere condition of radius~$\rho$ (then, $d_{\mc{H}}(U,U_\rho)\leq2\rho$). Another sufficient condition for~\eqref{dUrho} is the case of a subgraph~\eqref{Ugamma} with $\gamma$ having uniformly bounded local oscillations, that is,
\be\label{unifosc}
\sup_{x',y'\in\R^{N-1},\,|x'-y'|\le1}|\gamma(x')-\gamma(y')|<+\infty,
\ee
Independently of~\eqref{Ugamma} or~\eqref{unifosc}, notice also that, in order to have the conclusion~\eqref{dH}, condition~\eqref{dUrho} needs to be fulfilled with the quantity~$\rho$ provided by Hypothesis~\ref{hyp:invasion}. Thus,~\eqref{dH} holds when~$f$ satisfies the condition~\eqref{HTconditions} (ensuring the hair trigger effect), as soon as $U\neq\emptyset$ is uniformly $C^{1,1}$. 

We finally mention that the conditions~\eqref{hyp:U} and~\eqref{dUrho} can not be compared and the spreading properties do not hold in general without them. For instance, on the one hand, the set
$$U:=\bigcup_{n\in\N}\overline{B_{2^n+1}}\setminus B_{2^n-1}$$
satisfies $U_\rho\neq\emptyset$ and~\eqref{dUrho} for any $\rho\in(0,1]$, but it does not satisfy~\eqref{hyp:U} with any $\rho>0$, and the solution $u$ of~\eqref{homo}-\eqref{defu0} with, say, $f(s)=s(1-s)$ then satisfies~\eqref{dH}, but it does not satisfy~\eqref{spreading0},~\eqref{ass-cpt} or~\eqref{Hloc}, for any function $w:\Sph\to[0,+\infty]$ and any open set~$\mc{W}\subset\R^N$ which is star-shaped with respect to the origin, see~\cite[Proposition~6.1]{HR1}. On the other hand, the set
$$\baa{rcll}
U& := & & \big\{x\in\R^N:x_1\ge0,\ x_2^2+\cdots+x_N^2\le1\big\}\vspace{3pt}\\
& & \cup & \big\{x\in\R^N:x_1\ge0,\ (x_2-\sqrt{x_1})^2+x_3^2+\cdots+x_N^2\le e^{-x_1^2}\big\}\eaa$$
satisfies $U_\rho\neq\emptyset$ and~\eqref{hyp:U} for any $\rho\in(0,1]$, but it does not satisfy~\eqref{dUrho} with any $\rho>0$, and the solution $u$ of~\eqref{homo}-\eqref{defu0} with, say, $f(s)=s(1-s)$ then satisfies~\eqref{spreading0},~\eqref{ass-cpt} and~\eqref{Hloc} with $\mc{W}=\R^+\mathrm{e}_1+B_{c^*}$, but it does not satisfy~\eqref{dH}, see~\cite[Proposition~6.2]{HR1}. 


\subsection{Further convergence results for general reactions and loca\-lized initial data}\label{sec24}

Many papers have been devoted to the study of large-time dynamics of solutions of equations of the type~\eqref{homo}, with or without Hypotheses~\ref{hyp:invasion} or~\ref{hyp:minimalspeed}, when the initial conditions~$u_0:\R^N\to[0,1]$ are compactly supported or are somehow localized. For instance, with $N=1$ and $f$ of the type~\eqref{ignition}, with in addition $f$ non-decreasing in a neighborhood of $\alpha$, or of the type~\eqref{bistable} with $\int_0^1f(s)ds>0$, it was proved in~\cite{DM1} that, for any family $[0,+\infty)\ni\lambda\mapsto u_{0,\lambda}$ of compactly supported initial conditions, which is continuous and increasing in the $L^1(\R)$ sense and which is such that $u_{0,0}=0$, there is a unique threshold $\lambda^*\in(0,+\infty]$ such that the solutions $u_\lambda$ of~\eqref{homo} with initial conditions $u_{0,\lambda}$ satisfy:
\begin{itemize}
\item $u_\lambda(t,\cdot)\to0$ as $t\to+\infty$ uniformly in $\R$ if $0\le\lambda<\lambda^*$ (the so-called extinction case),
\item $u_\lambda(t,\cdot)\to1$ as $t\to+\infty$ locally uniformly in $\R$ if $\lambda^*<+\infty$ and $\lambda>\lambda^*$ (the invasion case),
\item $u_{\lambda^*}(t,\cdot)\to\alpha$ as $t\to+\infty$ locally uniformly in $\R$ if $\lambda^*<+\infty$ in the case~\eqref{ignition}, while $u_{\lambda^*}(t,\cdot)\to\Phi$ as $t\to+\infty$ uniformly in $\R$ in the case~\eqref{bistable}, where $\Phi:\R\to(0,1)$ is a stationary solution of~\eqref{homo} such that $\Phi(\pm\infty)=0$ (such a function $\Phi$ is actually unique and radially decreasing, up to spatial shifts).
\end{itemize}
The first results of that type were obtained in~\cite{Z} when the initial conditions are indicator functions of bounded intervals. We also refer to~\cite{AHR,AW,GRH,K,LK} for other extinction/invasion results with respect to the size, amplitude or fragmentation of the initial condition~$u_0$ in $\R^N$ for various reaction terms $f$,~\cite{MP1,MP2,P0} for the existence of unique thresholds with bistable-type autonomous or non-autonomous equations and compactly supported initial conditions in~$\R$ or~$\R^N$, and to~\cite{MZ1,MZ2} for similar conclusions with various functions $f$ in the case of radially non-increasing symmetric and possibly not-compactly-supported initial conditions in $L^2(\R)$ and $L^2(\R^N)$. 

Other results, holding for more general reaction terms $f$, deal with the question of the local or global large-time convergence of the solutions of~\eqref{homo} to a stationary solution (convergence results), or to the set of stationary solutions (quasiconvergence). For instance, it was proved in~\cite{DM1} that, in dimension $N=1$, under the sole assumption $f(0)=0$, any bounded nonnegative solution of~\eqref{homo} with a compactly supported initial condition converges locally in $\R$ to a stationary solution, which is either constant or even and decreasing with respect to a point. Further positive or negative convergence or quasiconvergence results for various equations in $\R$ or $\R^N$ have been obtained in~\cite{DP,MP1,MP2,P1}. 

Lastly, equations of the type~\eqref{homo} set in unbounded domains~$\Omega$ instead of $\R^N$ and notions of spreading speeds and persistence/invasion for solutions that are initially compactly supported in such domains have been investigated in~\cite{BHN2,R3}.


\subsection{Known flattening results for bounded or convex-like initial supports}\label{sec25}

In this section, we first recall the classical result of~\cite{J} leading to the local flattening of invading solutions with bounded initial supports. We then introduce the notions of $\Omega$-limit set and local one-dimensional symmetry for solutions with general initial supports, after~\cite{HR2}, and we present known recent results in the case of almost-planar, $V$-shaped, or convex-like initial supports.

\subsubsection{Local flattening for solutions with bounded initial supports}\label{sec251}

For the invading solutions~$u$ (that is, those converging to $1$ locally uniformly in $\R^N$ as $t\to+\infty$) with compactly supported initial conditions $u_0$ such that $0\le u_0\le 1$ in $\R^N$, it was proved in~\cite{J} from the parabolic strong maximum principle and the Hopf lemma that, for any $t>0$ and any $x\in\R^N$, there holds that
\be\label{jones1}
(\{x\}+\R^+\nabla u(t,x))\,\cap\,C\neq\emptyset,
\ee
where $C$ is the convex hull of the support of $u_0$. Since $C$ is bounded, it follows that
\be\label{jones2}
\sup_{|x|\ge A,\,t>0}\big|\hat{\nabla u(t,x)}+\hat{x}\big|\to0\ \hbox{ as }A\to+\infty.
\ee
Since $\lim_{t\to+\infty}u(t,\cdot)=1$ (locally uniformly in $\R^N$) and $\lim_{|x|\to+\infty}u(t,x)=0$ for every $t>0$ (because $u_0$ is compactly supported), one gets that, for any $\lambda\in(0,1)$, the level set $E_\lambda(t)$ is a non-empty compact set for every $t>0$ large enough, and $\min_{\lambda'\in(0,\lambda],\,x\in E_{\lambda'}(t)}|x|\to+\infty$ as $t\to+\infty$. Hence, for any sequence $(t_n)_{n\in\N}$ in $(0,+\infty)$ diverging to $+\infty$ and any sequence $(x_n)_{n\in\N}$ in $\R^N$ such that $\lambda_n:=u(t_n,x_n)\to\lambda\in(0,1)$ as $n\to+\infty$, there holds that, for every $\varepsilon>0$ and $R>0$,
\be\label{jones3}
E_{\lambda_n}(t_n)\cap B_R(x_n)\subset\Big\{x\in B_R(x_n):\big|(x-x_n)\cdot\hat{x_n}\big|<\varepsilon\Big\}
\ee
for all $n$ large enough.\footnote{\ Indeed, otherwise, there would be a sequence $(y_n)_{n\in\N}$ in $\R^N$ such that $u(t_n,y_n)=\lambda_n=u(t_n,x_n)$, $\sup_{n\in\N}|y_n-x_n|<+\infty$ and $\limsup_{n\to+\infty}|(y_n-x_n)\cdot\hat{x_n}|>0$. Rolle's theorem would then yield the existence of a sequence $(z_n)_{n\in\N}$ such that $z_n\in[x_n,y_n]$ and $\nabla u(t_n,z_n)\cdot(y_n-x_n)=0$ for all $n\in\N$. Since $|x_n|\to+\infty$ and then $|z_n|\to+\infty$ as $n\to+\infty$, and since the sequence $(y_n-x_n)_{n\in\N}$ is bounded, it would then follow from~\eqref{jones2} that $(y_n-x_n)\cdot\hat{z_n}\to0$ as $n\to+\infty$, and then $(y_n-x_n)\cdot\hat{x_n}\to0$ as $n\to+\infty$, leading to a contradiction.} Furthermore, assuming without loss of generality that $|x_n|>0$ for all $n$, and letting $R_0>0$ be such that $C\subset B_{R_0}$, property~\eqref{jones1} implies that each function $u(t_n,\cdot)$ is then decreasing with respect to the direction $\hat{x_n}$ in the half-space $\{x\in\R^N:x\cdot\hat{x_n}\ge R_0\}$. Since $\lim_{t\to+\infty}u(t,\cdot)=1$ locally uniformly in $\R^N$ and since $\lim_{|x|\to+\infty}u(t,x)=0$ for every $t>0$, it then follows from~\eqref{jones3} and the previous observations that, for every $\varepsilon>0$ and $R>R'>0$, there is $n_0\in\N$ such that, for every $n\ge n_0$, $E_{\lambda_n}(t_n)\cap B_R(x_n)$ is the graph of a function of the $N-1$ variables orthogonal to $\hat{x_n}$, which is defined in at least an open Euclidean ball $B'$ of radius $R'$ in $\R^{N-1}$, and whose oscillation in $B'$ is less than $\varepsilon$. 
In particular, the level sets $E_\lambda(t)$ become locally flat as $t\to+\infty$ 
in the sense described at the end of Section~\ref{sec1}.

Moreover, with the same assumptions and notations as in the previous paragraph, up to extraction of a subsequence, one can assume that $\hat{x_n}\to e\in\Sph$ and that, from standard parabolic estimates, the functions $u(t_n+\cdot,x_n+\cdot)$ converge in $C^{1;2}_{loc}(\R_t\times\R_x^N)$ to a solution $v:\R\times\R^N\to[0,1]$ solving~\eqref{homo} for all $t\in\R$. Furthermore, from the strong parabolic maximum principle, either $v\equiv0$ in $\R\times\R^N$, or $v\equiv1$ in $\R\times\R^N$, or $0<v<1$ in $\R\times\R^N$. Here, since $v(0,0)=\lambda\in(0,1)$, one has $0<v<1$ in $\R\times\R^N$. It then follows from~\eqref{jones2} that $\psi:=v(0,\cdot)\in C^2(\R^N)$ satisfies
$$e'\cdot\nabla\psi\equiv0\hbox{ in $\R^N$}$$
for every $e'\in\Sph$ 
orthogonal to $e$. In other words, $\psi$ is thus one-dimensional or planar, that is, it can be written as
$$\psi(x)=\Psi(x\cdot e)$$
for some function $\Psi:\R\to\R$.

Notice also that, together with $\lim_{t\to+\infty}u(t,\cdot)=1$ locally uniformly in $\R^N$ and $\lim_{|x|\to+\infty}u(t,x)=0$ for every $t>0$, the fundamental property~\eqref{jones1} also implies that, for any $\lambda\in(0,1)$ and $x_0\in\R^N$, there are $T>0$, $A>0$ and a map $[T,+\infty)\ni t\mapsto R(t)\in(0,+\infty)$ such that $\lim_{t\to+\infty}R(t)=+\infty$ and
$$E_\lambda(t)\ \subset\ B_{R(t)+A}(x_0)\setminus B_{R(t)}(x_0)\quad \text{for all }t\ge T,$$
see~\cite{R2}. This nevertheless does not mean that the level sets become asymptotically spherical in the sense that $\inf_{x_0\in\R^N,\,R>0}d_{\mc{H}}(E_\lambda(t),\partial B_R(x_0))\to0$ as $t\to+\infty$. Actually, the latter property fails in general, as proved in~\cite{R2,R} for various functions $f$.

\subsubsection{Notions of $\Omega$-limit set and asymptotic local one-dimensional symmetry for solutions with general initial supports}

Following~\cite{HR2,P2,P3}, the observations of Section~\ref{sec251} led to introduce the notion of {\em asymptotic local one-dimensional symmetry} for a solution $u:(0,+\infty)\times\R^N\to[0,1]$ of~\eqref{homo}. More precisely, we define the $\O$-limit set $\Omega(u)$ of $u$ by
\be\label{def:Omega}\begin{array}{ll}
\hspace{-8pt}\Omega(u)\!:=\!\Big\{&\!\!\!\!\!\!\psi\in L^\infty(\R^N):u(t_n,x_n+\.)\to\psi\text{ in $L^\infty_{loc}(\R^N)$ as $n\to+\infty$,}\\
& \!\!\!\!\!\text{for some sequences $(t_n)_{n\in\N}$ in $\R^+\!$ diverging to $\!+\infty$, and $\!(x_n)_{n\in\N}$ in $\R^N$}\!\Big\}.
\end{array}
\ee
Notice that $\Omega(u)\neq\emptyset$ and $\Omega(u)\subset C^2(\R^N,[0,1])$, from standard parabolic estimates. The solution $u$ is called {\em asymptotically locally planar} if every $\psi\in\Omega(u)$ is one-dimensional, that is, if
$$\psi(x)\equiv\Psi(x\.e)\ \hbox{ for all $x\in\R^N$},$$
for some $\Psi\in C^2(\R)$ and~$e\in\Sph$, and $\psi$ is called one-dimensional and (strictly) monotone if $\psi(x)\equiv\Psi(x\.e)$ with $\Psi$ (strictly) monotone.\footnote{\ This property reclaims the De Giorgi conjecture about the one-dimensional property of bounded solutions of the elliptic Allen-Cahn equation $\Delta u+u(1-u)(u-1/2)=0$ in $\R^N$ (obtained after a change of unknown from the original Allen-Cahn equation), which are assumed to be monotone in one direction, see~\cite{AC,BCN,dGconj,DKW,GG,S}.} Furthermore, from the parabolic strong maximum principle, any $\psi\in\Omega(u)$ satisfies either $\psi\equiv0$ in $\R^N$, or $\psi\equiv1$ in $\R^N$, or $0<\psi<1$ in $\R^N$. We also point out that if the level sets of $u$ become locally flat as $t\to+\infty$, in the sense described in Section~\ref{sec1}, then~$u$ is necessarily asymptotically locally planar.\footnote{\ Indeed, otherwise, there would exist $\psi\in\Omega(u)$ and two points $y\neq z\in\R^N$ such that $\nabla\psi(y)$ and $\nabla\psi(z)$ are not parallel. Thus, $0<\psi<1$ in $\R^N$. Since, say, $\nabla\psi(y)$ is not zero, one can assume without loss of generality, even if it means slightly moving $y$, that $\psi(y)\neq\psi(z)$. By definition of $\Omega(u)$, there are a sequence $(t_n)_{n\in\N}$ diverging to $+\infty$ and a sequence $(x_n)_{n\in\N}$ in $\R^N$ such that
$$\lim_{n\to+\infty}u(t_n,x_n+\cdot)=\psi$$
in $C^2_{loc}(\R^N)$. Since $\nabla\psi(y)$ and $\nabla\psi(z)$ are not zero, there are two sequences $(y_n)_{n\in\N}$ and $(z_n)_{n\in\N}$ in $\R^N$, converging respectively to $y$ and $z$, and such that $u(t_n,x_n+y_n)=\psi(y)$ and $u(t_n,x_n+z_n)=\psi(z)$ for all $n\in\N$. From the assumed asymptotic local flatness of the level sets of $u$, applied around the points $x_n+y_n$ and $x_n+z_n$ and the values $\psi(y)$ and $\psi(z)$, it follows that there are two affine hyperplanes $H_y$ and $H_z$, containing the points $y$ and $z$ respectively, such that $\psi$ is constant on $H_y$ and on $H_z$. Thus, $\psi(x)=\psi(y)$ for all $x\in H_y$ and $\psi(x)=\psi(z)$ for all $x\in H_z$. Since $\nabla\psi(y)$ and $\nabla\psi(z)$ are not zero, it follows that the hyperplanes $H_y$ and $H_z$ are orthogonal to $\nabla\psi(y)$ and $\nabla\psi(z)$, respectively. Therefore, these two hyperplanes are not parallel and then have a non-empty intersection $E$. Finally, for each $x\in E$, one then has $\psi(x)=\psi(y)$ and $\psi(x)=\psi(z)$, yielding a contradiction (since $\psi(y)\neq\psi(z)$).} But the converse property is trivially false in general (for instance, if $u_0\equiv\lambda\in(0,1)$ with $f(\lambda)=0$, then $u$ is asymptotically locally planar since $\Omega(u)=\{\lambda\}$, whereas the level set $E_\lambda(t)=\R^N$ is not locally flat as $t\to+\infty$!).

The results described in the first two paragraphs of Section~\ref{sec251}, derived from~\cite{J}, imply that any invading solution $u$ of~\eqref{homo} (this is the case where Hypothesis~\ref{hyp:invasion} is assumed and when $u_0=\1_U$ with $U$ containing a ball of radius $\rho$) with compactly supported initial condition is asymptotically locally planar. This property nevertheless is not true in general for non-invading solutions. For instance, with a bistable function $f$ of the type~\eqref{bistable} with $\int_0^1f(s)ds>0$, it follows from~\cite{MZ2,P0} that there is a unique $R>0$ such that the solution $u$ of~\eqref{homo}-\eqref{defu0} with $U=B_R$ converges as $t\to+\infty$ in $C^2(\R^N)$ to a radially decreasing stationary solution $\Phi:\R^N\to(0,1)$ such that $\Phi(x)\to0$ as $|x|\to+\infty$. Then
$$\Omega(u)=\{0\}\cup\{\Phi(a+\cdot):a\in\R^N\}$$
and $u$ is not asymptotically locally planar.

\subsubsection{Almost-planar and $V$-shaped initial supports}

For general initial conditions $u_0$ which are not compactly supported, the questions of the flattening of the level sets and the asymptotic local one-dimensional symmetry of the solutions of~\eqref{homo} can still be addressed. But the situation is more intricate, as for the spreading properties discussed in Section~\ref{sec23}. Actually, even for initial conditions $u_0$ of the type $u_0=\1_U$, the case of unbounded sets $U$ has been much less studied in the literature. However, on the one hand, if $U$ is the subgraph of a bounded function (or more generally when there are two parallel half-spaces $H$ and $H'$ with outward normal $e$ such that $H\subset U\subset H'$), then the asymptotic local one-dimensional symmetry holds for functions $f$ of the bistable type~\eqref{bistable}, as follows from~\cite{BH1,FM,MN,MNT,RR1}. The same conclusion is valid when $f$ satisfies the Fisher-KPP condition
\be\label{kpp0}
f(0)=f(1)=0,\ 0<f(s)\le f'(0)s\hbox{ for all $s\in(0,1)$},
\ee
from~\cite{BH1,B,HNRR,L,U1}. Furthermore, in these two cases, the $\Omega$-limit set $\Omega(u)$ is made up of the constants $0$ and $1$ and the shifts of planar traveling front profiles $x\mapsto\varphi^*(x\cdot e)$, with $\varphi^*$ solving~\eqref{limitsvp}, $(\varphi^*)''+c^*(\varphi^*)'+f(\varphi^*)=0$ in $\R$, and $c^*$ being the minimal speed (the unique speed in the bistable case) of planar traveling fronts connecting $1$ to $0$. Since $(\varphi^*)'<0$ in~$\R$, the level sets of these solutions $u$ necessarily locally flatten as $t\to+\infty$. Notice that this local flatness is however not global in general, see~\cite{MNT,RR1,RR2}. On the other hand, the asymptotic local flatness and one-dimensional symmetry are known to fail in general, for instance when $U$ is ``V-shaped'', i.e.~when $U$ is the union of two half-spaces with non-parallel boundaries, as follows from~\cite{HMR1,HMR2,HN,HR2,NT,RR1} for bistable or Fisher-KPP functions $f$.

\subsubsection{Convex-like initial supports}

Recently, we considered in~\cite{HR2} other classes of initial supports $U$, when the function $f$ satisfies a stronger Fisher-KPP condition, that is,
\be\label{fkpp}\left\{\baa{l}
f(0)=f(1)=0,\ f(s)>0\hbox{ for all }s\in(0,1),\vspace{3pt}\\
\displaystyle s\mapsto\frac{f(s)}{s}\hbox{ is nonincreasing in $(0,1]$}.\eaa\right.
\ee
In this case the hair trigger effect holds~\cite{AW}, i.e., Hypothesis~\ref{hyp:invasion} is fulfilled for any~$\theta,\rho\!>\!0$, moreover Hypothesis~\ref{hyp:minimalspeed} also holds and the minimal speed $c^*$ of planar traveling fronts connecting $1$ to $0$ is equal to $c^*=2\sqrt{f'(0)}$, see~\cite{AW,KPP}. We showed in~\cite[Theorem~2.1]{HR2} that, if $U$ satisfies~\eqref{dUrho} for some $\rho>0$ and if $U$ is convex, or at a finite Hausdorff distance from a convex set, then the solution $u$ of~\eqref{homo}-\eqref{defu0} is asymptotically locally planar: for any $\psi\in\Omega(u)$, there are $e\in\Sph$ and $\Psi:\R\to[0,1]$ such that $\psi(x)\equiv\Psi(x\cdot e)$ in~$\R^N$. Furthermore, either $\Psi$ is constant or $\Psi$ is strictly monotone. In the case where $\Psi$ is strictly monotone in $\R$ and if $(t_n)_{n\in\N}$ and $(x_n)_{n\in\N}$ are sequences such that $t_n\to+\infty$ and
$$u(t_n,x_n+\cdot)\to\psi\hbox{ in $C^2_{loc}(\R^N)$},$$
then the level sets $E_\lambda(t_n)$ locally flatten around the points~$x_n$ as $n\to+\infty$. The case of sets $U$ that are convex or at finite distance from a convex set is actually a particular case of a more general situation. Namely, for a given non-empty set $U\subset\R^N$ and a given point $x\in\R^N$, by letting
$$\pi_x:=\big\{\xi\in\ol U: |x-\xi|=\dist(x,U)\big\}$$
be the set of orthogonal projections of~$x$ onto $\ol U$ and, for $x\notin\ol U$, by defining
$$\mc{O}(x):=\sup_{\xi\in\pi_x,\,y\in U\setminus\{\xi\}}\,\hat{x-\xi}\.\hat{y-\xi},$$
with the convention that $\mc{O}(x)=-\infty$ if $U=\emptyset$ or $U$ is a singleton (otherwise $\mc{O}(x)=\cos\vartheta$, where $\vartheta$ is the infimum among all~$\xi\in\pi_x$ of half the opening of the largest exterior cone to $U$ at~$\xi$ having axis~$x-\xi$), we showed in~\cite[Theorem~2.2]{HR2} that, if  
\Fi{ballcone}
\lim_{R\to+\infty}\bigg(\,\sup_{x\in\R^N,\,\dist(x,U)=R}\mc{O}(x)\bigg)\leq 0,
\Ff	
then the solution $u$ of~\eqref{homo}-\eqref{defu0} is asymptotically locally planar, and the elements of~$\Omega(u)$ are either constant or planar and strictly decreasing with respect to a direction~$e$. In~\eqref{ballcone}, the left-hand side is equal to $-\infty$ if $\sup_{x\in\R^N}\dist(x,U)<+\infty$ (in this case, since $d_{\mathcal{H}}(U,U_\rho)<+\infty$, one gets that $u(t,x)\to1$ uniformly in $x\in\R^N$ as~$t\to+\infty$ from the results of Section~\ref{sec23}, hence $\Omega(u)=\{1\}$). The limit in~\eqref{ballcone} always exists, since the map
$$R\mapsto\sup_{x\in\R^N,\,\dist(x,U)=R}\mc{O}(x)$$
is nonincreasing, see~\cite[Lemma~4.3]{HR2}. Observe also that, if~$U$ is convex, then $\mc{O}(x)\leq0$ for every $x\notin\ol U$, hence~\eqref{ballcone} is fulfilled (it turns out to be satisfied as well if $U$ is at a finite distance from a convex set). Condition~\eqref{ballcone} is fulfilled as well by subgraphs~\eqref{Ugamma} of functions $\gamma\in L^\infty_{loc}(\R^{N-1})$ with vanishing global mean, that is, $|\gamma(x')-\gamma(y')|=o(|x'-y'|)$ as $|x'-y'|\to+\infty$ (notice that such sets $U$ are in general not convex nor at finite distance from a convex set). On the other hand, without the geometric condition~\eqref{ballcone}, the asymptotic local one-dimensional symmetry of the solutions of~\eqref{homo}-\eqref{defu0} does not hold in general (consider for instance a $V$-shaped $U$ that is the union of two non-parallel half-spaces), nor does it in general without the condition~\eqref{dUrho}, see~\cite[Propositions~4.4 and~4.6]{HR2}.

Under the assumptions~\eqref{dUrho} and~\eqref{fkpp}-\eqref{ballcone}, we also characterized in~\cite[Theorem~2.4]{HR2} the set of directions of strict decreasing monotonicity of the elements of $\Omega(u)$, defined by
$$\mathcal{E}=\big\{e\in\Sph\ :\ \exists\,\Psi:\R\to\R\hbox{ decreasing},\ \big(x\mapsto\Psi(x\cdot e)\big)\in\Omega(u)\big\}.$$
Namely, we showed that $\mc{E}$ coincides with the set of all limits of sequences $(\hat{x_n-\xi_n})_{n\in\N}$ with $\dist(x_n,U)\to+\infty$ and $\xi_n\in\pi_{x_n}$. In particular, if $U$ is bounded, then $\mc{E}=\Sph$, and this property in this case can also be viewed as consequences of results of~\cite{D,RRR}. If $U$ is convex, then $\mc{E}$ is the closure of the set of outward unit normal vectors to all half-spaces containing $U$. If $U$ satisfies~\eqref{Ugamma} with $\gamma$ having vanishing global mean, then $\mc{E}=\{\mathrm{e}_N\}$, that is, the elements of $\Omega(u)$ depend on the variable $x_N$ only, hence $\nabla_{\!x'}u(t,x',x_N)\to0$ as $t\to+\infty$ uniformly with respect to $(x',x_N)\in\R^{N-1}\times\R$.\footnote{\ However, as $t\to+\infty$, $u(t,\cdot)$ is in general not arbitrarily close uniformly in $\R^N$ to functions depending on $x_N$ only: this is the case for instance in dimension $N=2$ if $\gamma:\R\to\R$ has two different limits at $\pm\infty$.} Bounded functions $\gamma$ are particular cases of functions with vanishing global mean and the previous general results applied to this very specific case (and to the case when $U$ is trapped between two parallel half-spaces, even if not a subgraph itself) can also be derived from~\cite{BH1,B,HNRR,L,U1}.

For bistable functions $f$ of the type~\eqref{bistable} with $\int_0^1f(s)ds>0$, it is known in the literature that the asymptotic local one-dimensional symmetry holds when $U$ is bounded and contains a large enough ball (see~\cite{AW,J}), or when $U$ is trapped between two parallel half-spaces (see~\cite{BH1,MN,MNT,HR2}), and it is known to fail when $U$ is $V$-shaped (see~\cite{HMR1,HMR2,NT,RR1}). These facts lead us to conjecture that the result of \cite{HR2} should remain valid beyond the KPP case~\eqref{fkpp}. Namely, for general functions $f$ satisfying Hypothesis~\ref{hyp:invasion} for some $\rho>0$, we conjecture that the solution $u$ of~\eqref{homo}-\eqref{defu0} is asymptotically locally planar provided that $U$ satisfies~\eqref{dUrho} and~\eqref{ballcone}.

Some results in the literature lead to a complete characterization of the $\O$-limit set in specific situations. Namely, when $U$ is bounded with non-empty interior, or when $U$ is trapped between two parallel half-spaces, it holds that
\be\label{conjOmega}
\O(u)=\{0,1\}\,\cup\,\big\{x\mapsto\vp^*(x\.e +a)\,:\,e\in\mc{E},\,a\in\R\,\big\}.
\ee
where $\varphi^*$ is the profile of the planar traveling front connecting $1$ to $0$ with minimal speed $c^*=2\sqrt{f'(0)}$ (see~\cite{D,RRR} for the first case and~\cite{BH1,B,HNRR,L,U1} for the second one). Let us point out that~\eqref{conjOmega} is also coherent with the result~\eqref{dH}, which loosely says that the upper level sets spread with normal velocity~$c^*$, which is precisely the speed at which the fronts $\vp^*(x\.e-c^*t+a)$ travel. Therefore, it is reasonable to conjecture that~\eqref{conjOmega} holds true for a general set $U$ satisfying~\eqref{dUrho} and~\eqref{ballcone}.

	
\section{The subgraph case: new flattening properties and related conjectures}\label{sec3}

In this section, we focus on the important class of initial conditions which are indicator functions of subgraphs in $\R^N$. Up to rotation, let us consider graphs in the direction~$x_N$, and initial conditions $u_0$ given by
\be\label{defu0bis}
u_0(x',x_N)=\begin{cases} 0 & \text{if }x_N>\gamma(x'),\\
1 & \text{otherwise},
\end{cases}
\ee
that is, $u_0=\1_U$ with $U$ given by~\eqref{Ugamma}, where the function $\gamma:\R^{N-1}\to\R$ is always assumed to be in $L^\infty_{loc}(\R^{N-1})$. First of all, from parabolic estimates, one has~$u(t,x',x_N)\to0$ as~$x_N\to+\infty$ and~$u(t,x',x_N)\to1$ as $x_N\to-\infty$, locally uniformly in~$(t,x')\in[0,+\infty)\times\R^{N-1}$. Furthermore,~$u(t,x',x_N)$ is non-increasing with respect to~$x_N$ by the parabolic maximum principle, because the initial datum $u_0$ is, and one actually sees that $\partial_{x_N}u<0$ in $(0,+\infty)\times\R^N$ by differentiating~\eqref{homo} with respect to~$x_N$ and applying the strong maximum principle to~$\partial_{x_N}u$. As~a consequence, one infers that, for every $t>0$, $x'\in\R^{N-1}$ and~$\lambda\in(0,1)$, there exists a unique value $x_N$ such that $u(t,x',x_N)=\lambda$, which will be denoted $X_\lambda(t,x')$ in the sequel,\footnote{\  The above arguments also easily imply that the function $(\lambda,t,x')\mapsto X_\lambda(t,x')$ is continuous in $(0,1)\times(0,+\infty)\times\R^{N-1}$.} that is,
\be\label{defX}
u(t,x',X_\lambda(t,x'))=\lambda.
\ee
In other words, the sets $E_\lambda(t)$ and $F_\lambda(t)$ given in~\eqref{defElambda0} and~\eqref{defFlambda} are respectively the graphs and open subgraphs of the functions $x'\mapsto X_\lambda(t,x')$.

Under Hypothesis~\ref{hyp:minimalspeed}, the spreading results of Section~\ref{sec23} applied to this case give some information on the shape of the graphs of $X_\lambda(t,\cdot)$ at large time and large space in terms of the function~$\gamma$, provided the conditions~\eqref{hyp:U} or~\eqref{dUrho} are fulfilled (the latter holds as soon as $\gamma$ has uniformly bounded local oscillations in the sense of~\eqref{unifosc}). We are now interested in the local-in-space behavior of the graphs of $X_\lambda(t,\cdot)$ at large time. Let us first point out that, because of the asymmetry of the roles of the rest states $0$ and $1$, the behavior of the graphs of $X_\lambda(t,\cdot)$ will be radically different depending on the profile of the function~$\gamma$ at infinity, and especially on whether the function $\gamma$ be large enough or not at infinity. This difference is already inherent in the results of Section~\ref{sec23}. Indeed, for instance, in the particular case~$\gamma(x')=\alpha\,|x'|$, whatever $\alpha\in\R$ may be, the graphs of the functions $X_\lambda(t,\cdot)$ look like the sets $\{x\in\R^N:\dist(x,U)=c^*t\}$ at large time~$t$, in the sense of~\eqref{dH}. If $\alpha>0$, then for each~$t>0$ the set $\{x\in\R^N:\dist(x,U)=c^*t\}$ is a shift of the graph of $\gamma$ in the direction~$x_N$ and therefore it has a vertex, whereas it is $C^1$ if $\alpha\le0$. Of course, for each~$t>0$, in both cases~$\alpha>0$ and~$\alpha\le0$, each level set of $u$ (that is, each graph of $X_\lambda(t,\cdot)$) is at least of class~$C^2$ from the implicit function theorem and the fact that $\partial_{x_N}u<0$ in $(0,+\infty)\times\R^N$. Nevertheless, the previous observations imply that there should be a difference between the flattening properties of the level sets of $u$ according to the coercivity of the function $\gamma$ at infinity.

The following result deals with the non-coercive case, i.e., $\limsup_{|x'|\to+\infty}\gamma(x')/|x'|\le0$.

\begin{theorem}\label{thm:subgraph}
Assume that Hypothesis~$\ref{hyp:minimalspeed}$ holds $($hence Hypothesis~$\ref{hyp:invasion}$ as well$)$. Let~$u$ be the solution of~\eqref{homo} with an initial datum~$u_0$ given by~\eqref{defu0bis}. If
\Fi{gamma<0}
\limsup_{|x'|\to+\infty}\frac{\gamma(x')}{|x'|}\leq0,
\Ff
then, for every $\lambda\in[\theta,1)$, with $\theta\in(0,1)$ given by Hypothesis~$\ref{hyp:invasion}$, and every basis $(\mathrm{e}'_1,\cdots,\mathrm{e}'_{N-1})$ of $\R^{N-1}$, there holds
\be\label{liminf}
\liminf_{t\to+\infty}\Big(\min_{|x'|\le R,\,1\le i\le N-1}|\nabla_{\!x'}X_\lambda(t,x')\cdot\mathrm{e}'_i|\Big)\ \longrightarrow\ 0\ \hbox{ as }R\to+\infty,
\ee
and even
\be\label{liminf2}
\sup_{x'_0\in\R^{N-1}}\Big[\liminf_{t\to+\infty}\Big(\min_{x'\in\overline{B'_R(x'_0)},\,1\le i\le N-1}|\nabla_{\!x'}X_\lambda(t,x')\cdot\mathrm{e}'_i|\Big)\Big]\ \longrightarrow\ 0\ \hbox{ as }R\to+\infty,
\ee
where $B'_R(x'_0)$ denotes the open Euclidean ball of center $x'_0$ and radius $R$ in $\R^{N-1}$.
\end{theorem}

Notice that, in dimension~$N=2$, property~\eqref{liminf} means that
$$\liminf_{t\to+\infty}\Big(\min_{[-R,R]}|\partial_{x_1}X_\lambda(t,\cdot)|\Big)\ \longrightarrow\ 0\ \hbox{ as }R\to+\infty,$$
for every $\lambda\in[\theta,1)$, and that an analogous consideration holds for~\eqref{liminf2}.

Theorem~\ref{thm:subgraph} will be proved in Section~\ref{sec51}. Roughly speaking, the conclusion says that the level set of any value $\lambda\in[\theta,1)$ becomes almost flat in some directions along some sequences of points and some sequences of times converging to $+\infty$. We point out that the estimates on $\nabla_{\!x'}X_\lambda(t,x')$ immediately imply analogous estimates on $\nabla_{\!x'}u(t,x',X_\lambda(t,x'))$, because
\Fi{DchiDu}
\nabla_{\!x'}u(t,x',X_\lambda(t,x'))=-\partial_{x_N}u(t,x',X_\lambda(t,x'))\nabla_{\!x'}X_\lambda(t,x'),
\Ff
and $\partial_{x_N}u$ is bounded in $[1,+\infty)\times\R^N$ from standard parabolic estimates.\footnote{\ Since $u$ is of class $C^1$ in $(0,+\infty)\times\R^N$ and the function $(\lambda,t,x')\mapsto X_\lambda(t,x')$ is continuous in~$(0,1)\times(0,+\infty)\times\R^{N-1}$, it follows that the function $(\lambda,t,x')\mapsto\nabla_{\!x'}X_\lambda(t,x')$ is also continuous in~$(0,1)\times(0,+\infty)\times\R^{N-1}$.} Hence, the conclusions~\eqref{liminf}-\eqref{liminf2} imply that
$$\liminf_{t\to+\infty}\Big(\min_{|x'|\le R,\,1\le i\le N-1}|\nabla_{\!x'}u(t,x',X_\lambda(t,x'))\cdot\mathrm{e}'_i|\Big)\ \longrightarrow\ 0\ \hbox{ as }R\to+\infty$$
and
$$\sup_{x'_0\in\R^{N-1}}\Big[\liminf_{t\to+\infty}\Big(\min_{x'\in\overline{B'_R(x'_0)},\,1\le i\le N-1}|\nabla_{\!x'}u(t,x',X_\lambda(t,x'))\cdot\mathrm{e}'_i|\Big)\Big]\ \longrightarrow\ 0\ \hbox{ as }R\to+\infty,$$
for every $\lambda\in[\theta,1)$ and every basis $(\mathrm{e}'_1,\cdots,\mathrm{e}'_{N-1})$ of $\R^{N-1}$. The proof of~\eqref{liminf}-\eqref{liminf2} is done by way of contradiction and uses the fact that the level value $\lambda$ belongs to the interval~$[\theta,1)$, where $\theta\in(0,1)$ is given by Hypothesis~\ref{hyp:invasion}. Since the interface between the values~$0$ and~$1$ is initially sharp, we expect, as in the one-dimensional case handled in~\cite{N}, that the transition between~$0$ and~$1$ has a uniformly bounded width (in the sense of~\cite{BH2}) in the direction $x_N$ if, say, $\gamma$ is Lipschitz continuous (although the proof of this property does not extend easily in dimensions $N\ge 2$). If so, it would follow from the proof of Theorem~\ref{thm:subgraph} that~\eqref{liminf}-\eqref{liminf2} would then hold for any $\lambda\in(0,1)$. Actually, if $f$ is positive in~$(0,1)$, Hypothesis~\ref{hyp:invasion} is satisfied for any $\theta\in(0,1)$  and it follows from Theorem~\ref{thm:subgraph} that~\eqref{liminf}-\eqref{liminf2} then hold for all $\lambda\in(0,1)$.

We stress that, without the assumption~\eqref{gamma<0}, the conclusions~\eqref{liminf}-\eqref{liminf2} is imme\-diately seen to fail in general (a trivial counterexample is given by the solution with initial condition~\eqref{defu0bis} with $\gamma(x')=x'\cdot e'$ for some nonzero vector~$e'\in\R^{N-1}$, whose level sets are hyper\-planes in $\R^N$ orthogonal to $(e',-1)$ for all $t>0$, hence $\nabla_{\!x'}X_\lambda(t,x')=e'$ for all $\lambda\in(0,1)$, $x'\in\R^{N-1}$ and $t>0$ by~\eqref{DchiDu}). Moreover, if one assumes that $\liminf_{|x'|\to+\infty}\gamma(x')/|x'|\ge0$ instead of~\eqref{gamma<0}, counterexamples to~\eqref{liminf}-\eqref{liminf2} are provided by rotated $V$-shaped fronts, see~Proposition~\ref{pro46}~(i) below. However, with the assumption~\eqref{gamma<0}, we expect that the liminf of the minimum can be replaced by a limit in~\eqref{liminf}, without any reference to the size $R$, leading to the following conjecture.

\begin{conjecture}\label{conj1}
Under the assumptions of Theorem~$\ref{thm:subgraph}$, the conclusion~\eqref{liminf} can be strengthened by the limit
\be\label{limit2}
\nabla_{\!x'}X_\lambda(t,x')\rightarrow 0\ \hbox{ as $t\to+\infty$, \ locally uniformly in $x'\!\in\!\R^{N-1}$},
\ee
for every $\lambda\in[\theta,1)$.
\end{conjecture}

Even for $x'$-symmetric solutions $u$, property~\eqref{limit2} does not hold in general without the assumption~\eqref{gamma<0} of Theorem~\ref{thm:subgraph} (as for~\eqref{liminf}-\eqref{liminf2}, counterexamples are given by $V$-shaped fronts, see Proposition~\ref{pro46}~(ii) for further details). We also point out that, even with the assumption~\eqref{gamma<0}, property~\eqref{limit2} does not hold in general {\em uniformly} with respect to $x'\in\R^{N-1}$ (for instance, in dimension $N=2$, easy counter\-examples are given by nonpositive functions $\gamma$ with a negative slope as $x_1\to+\infty$, see Proposition~\ref{pro46}~(iii) for further details). On the other hand, a strong support to the validity of Conjecture~\ref{conj1} is provided by the results cited in Section~\ref{sec23}. Indeed,~\cite[Theorem~2.4]{HR1} asserts that, under assumption~\eqref{dUrho} (for instance if~\eqref{unifosc} holds), then $F_\lambda(t)\sim U+B_{c^*t}$ for $t$ large for any $\lambda\in(0,1)$, in the sense of~\eqref{dH}, and one can check that condition~\eqref{gamma<0} entails that the exterior unit normals to the set $U+B_{c^*t}$ at the points $(x',x_N)\in\partial(U+B_{c^*t})$ (whenever they exist) approach the vertical direction $\mathrm{e}_N=(0,\cdots,0,1)$ as $t\to+\infty$, locally uniformly with respect to $x'\in\R^{N-1}$. Hence the same is expected to hold for the sets $F_\lambda(t)$, which is what~\eqref{limit2} asserts. This kind of argument can be made rigorous, building on the results of Section~\ref{sec23}, and lead to a weaker version of Conjecture~\ref{conj1}, see Proposition~\ref{pro:trapped} below. We derive two other weaker forms of Conjecture~\ref{conj1}: the first one in the case where~$f$ fulfills~\eqref{HTconditions}, see Proposition~\ref{flatteningHT}, the second one when $f$ is of the strong Fisher-KPP type~\eqref{fkpp}, see Proposition~\ref{cor:flatteningKPP}. This latter result asserts that the conclusion~\eqref{limit2} holds up to subsequences, and relies on the results about the asymptotic local one-dimensional symmetry of~\cite{HR2}. As~for the full Conjecture~\ref{conj1}, we will prove it  in the strong Fisher-KPP case, provided  that condition~\eqref{gamma<0} is strengthened by a suitable divergence to $-\infty$ of $\gamma$ at infinity, see Corollary~\ref{cor:DGlag} below; this result is proved combining the asymptotic local one-dimensional symmetry of~\cite{HR2} with a new result about the location of the level sets of solutions, that we derive under those assumptions in Section~\ref{seclag2}.

Another situation in which we are able to obtain the conclusion of Conjecture~\ref{conj1} is when the initial condition~$u_0$ has an asymptotically $x'$-symmetric conical support or is the subgraph of an axi\-symmetric nonincreasing function. Here is the precise result, which actually holds under the weaker Hypothesis~\ref{hyp:invasion} instead of Hypothesis~\ref{hyp:minimalspeed}.

\begin{theorem}\label{thm:conical}
Assume that Hypothesis~$\ref{hyp:invasion}$ holds. Let $u$ be the solution of~\eqref{homo} with an initial datum $u_0$ given by~\eqref{defu0bis}, where the function $\gamma\in L^\infty_{loc}(\R^{N-1})$ satisfies one of the following assumptions:
\begin{itemize}
\item[{\rm{(i)}}] either $\gamma$ is of class~$C^1$ outside a compact set and there is $\ell\ge0$ such that
\be\label{hypconical}
\hspace{-8pt}\left\{\baa{ll}
\!\!\!\gamma'(x_1)\to\mp\ell\ \hbox{ as }x_1\to\pm\infty & \hbox{if $\!N\!=\!2$},\vspace{3pt}\\
\displaystyle\!\!\!\nabla\gamma(x')=-\ell\,\frac{x'}{|x'|}+O(|x'|^{-1-\eta})\ \hbox{ as }|x'|\to+\infty,\hbox{ for some $\eta>0$,} & \hbox{if }\!N\!\ge\!3\eaa\right.
\ee
$($in dimension $N=3$, by writing $\gamma(x')=\tilde{\gamma}(r,\vartheta)$ in the standard polar coordinates, that means that $\partial_r\tilde{\gamma}(r,\vartheta)=-\ell+O(r^{-1-\eta})$ and $\partial_\vartheta\tilde{\gamma}(r,\vartheta)=O(r^{-\eta})$ as $r\to+\infty$$)$;
\item[{\rm{(ii)}}] or $\gamma$ is continuous outside a compact set and $\gamma(x')/|x'|\to-\infty$ as~$|x'|\to+\infty$;
\item[{\rm{(iii)}}] or $\gamma(x')=\Gamma(|x'-x'_0|)$ outside a compact set,  for some $x'_0\in\R^{N-1}$ and some con\-tinuous nonincreasing function $\Gamma:\R^+\to\R$;
\item[{\rm{(iv)}}]  or $\gamma(x')=\Gamma(|x'-x'_0|)$ outside a compact set,  for some $x'_0\in\R^{N-1}$ and some $C^1$ function $\Gamma:\R^+\to\R$ such that  $\Gamma'(r)\to0$ as $r\to+\infty$.
\end{itemize}
Then, for every $\lambda_0\in(0,1)$, there holds that
\be\label{limit3chi}
\nabla_{\!x'}X_\lambda(t,x')\rightarrow 0\hbox{ \ as $t\to+\infty$, locally in $x'\!\in\!\R^{N-1}$ and uniformly in $\lambda\!\in\!(0,\lambda_0]$}
\ee
and moreover
\be\label{limit3u}
\nabla_{\!x'}u(t,x',x_N)\rightarrow 0\hbox{ \ as $t\to+\infty$, locally in $x'\!\in\!\R^{N-1}$ and uniformly in $x_N\!\in\!\R$}.
\ee
\end{theorem}

Theorem~\ref{thm:conical} is proved in Section~\ref{sec53}. We observe at once that it entails that the level sets locally flatten at large time around points with bounded $x'$-coordinates. Furthermore, remembering the definition~\eqref{def:Omega} of $\Omega(u)$, it yields that the elements $\psi$ of $\Omega(u)$ corres\-ponding to sequences $(x_n)_{n\in\N}$ with $\sup_{n\in\N}|x'_n|<+\infty$ are then necessarily functions of~$x_N$ only. In particular, the solution $u$ becomes asymptotically locally one-dimensional in any region of $\R^N$ with bounded $x'$-coordinates.

On the other hand, it is easy to see that, even under Hypothesis~\ref{hyp:minimalspeed} (which is stronger than Hypothesis~\ref{hyp:invasion}), if~\eqref{hypconical} holds with $\ell>0$, then the convergence in~\eqref{limit3chi} cannot be uniform with respect to~$x'\in\R^{N-1}$, see Proposition~\ref{pro46}~(iv) for further details. In other words, if the initial interface between the states $0$ and $1$ has a non-zero slope at infinity, then the level sets cannot become uniformly flat at large time. This observation naturally leads to the following conjecture.
	
\begin{conjecture}\label{conj:order1}
Assume that Hypothesis $\ref{hyp:minimalspeed}$ holds $($hence Hypothesis~$\ref{hyp:invasion}$ as well$)$. Let $u$ be the solution of~\eqref{homo} with an initial datum $u_0$ given by~\eqref{defu0bis}. If $\gamma$ is of class $C^1$ outside a compact set and if
\Fi{gamma'to0}
\lim_{|x'|\to+\infty}\nabla\gamma(x')=0,
\Ff
then, for every $\lambda_0\in(0,1)$,
\be\label{Dchi0}
\nabla_{\!x'}X_\lambda(t,x')\rightarrow 0\hbox{ \ as $t\to+\infty$, uniformly in $x'\!\in\!\R^{N-1}$ and in $\lambda\!\in\!(0,\lambda_0]$}
\ee
and moreover
\be\label{Du0}
\nabla_{\!x'}u(t,x)\rightarrow 0\hbox{ \ as $t\to+\infty$, uniformly in $x\!\in\!\R^N$}.
\ee
\end{conjecture}
	
Properties~\eqref{Dchi0}-\eqref{Du0} are known to hold if condition~\eqref{gamma'to0} is replaced by the boundedness of $\gamma$, at least for some classes of functions $f$ and with $\lambda\in(0,\lambda_0]$ in~\eqref{Dchi0} replaced by $\lambda\in[a,b]$, for any fixed $0<a\le b<1$. More precisely, if the function $f$ is of the bistable type~\eqref{bistable} these properties follow from some results in~\cite{BH1,FM,MN,MNT,RR1}, and the same conclusions hold for more general functions $f$ of the multistable type, see~\cite{P2}, or for KPP type functions~$f$ satisfying~\eqref{kpp0}, see~\cite{BH1,B,HNRR,L,U1}.  

However, by considering some functions $\gamma$ with large local oscillations at infinity, it turns out that both conclusions of Conjecture~\ref{conj:order1} cannot hold if~\eqref{gamma'to0} is replaced by the weaker condition $\lim_{|x'|\to+\infty}\gamma(x')/|x'|=0$, as asserted by the following result, whose proof is done in Section~\ref{sec54}.

\begin{proposition}\label{pro:asymp}
Conjecture~$\ref{conj:order1}$ fails in general if assumption~\eqref{gamma'to0} is replaced by  
\be\label{gamma0}
\nabla\gamma\in L^{\infty}(\R^{N-1})\ \hbox{ and }\ \lim_{|x'|\to+\infty}\frac{\gamma(x')}{|x'|}=0.\ee
\end{proposition}

To complete this section, we state a conjecture on the limiting profile of the solution of~\eqref{homo} and~\eqref{defu0bis} when $\gamma$ satisfies the non-coercivity condition~\eqref{gamma<0}. Let us preliminarily observe that the results of Section~\ref{sec23} allow one to derive the spreading speed in the vertical direction in such a case, also if the initial support does not fulfill condition~\eqref{hyp:U}. Indeed, on the one hand, we know that, under Hypothesis~\ref{hyp:minimalspeed}, the solution $u$ of~\eqref{homo} with~\eqref{defu0bis} and $\gamma\in L^\infty_{loc}(\R^{N-1})$ satisfies~\eqref{c<c*}, from~\cite[Proposition~3.1]{HR1}. On the other hand, under assumption~\eqref{gamma<0}, for any $\varepsilon>0$, there is $M_\varepsilon\in\R$ such that $U\subset U^\varepsilon:=\{(x',x_N)\in\R^N:x_N\le M_\varepsilon+\varepsilon|x'|\}$. Notice that each set~$U^\varepsilon$ does satisfy~\eqref{hyp:U} and therefore~\eqref{ass-cpt},~\eqref{FGgeneral} and~\eqref{Wshift} hold for the solution $u_\varepsilon$ of~\eqref{homo} with initial condition $\1_{U^\varepsilon}$, hence in particular $u_\varepsilon$ has a spreading speed equal to $w_\varepsilon(\mathrm{e}_N)=c^*\sqrt{1+\varepsilon^2}$ in the direction $\mathrm{e}_N$, in the sense of~\eqref{spreading0}. Since $u\le u_\varepsilon$ in $[0,+\infty)\times\R^N$ by the comparison principle, one infers from the arbitrariness of~$\varepsilon>0$ and from~\eqref{c<c*} that $u$ has a spreading speed equal to $w(\mathrm{e}_N)=c^*$ in the direction $\mathrm{e}_N$. This result about the spreading speed in the vertical direction, together with the flattening properties of the level sets discussed above, lead us to conjecture that, for initial data of the type~\eqref{defu0bis} and~\eqref{gamma<0}, the solutions locally converge along their level sets to the front profile $\varphi^*$ connecting $1$ to $0$ with minimal speed~$c^*$.

\begin{conjecture}\label{conj3}
Under the assumptions of Theorem~$\ref{thm:subgraph}$, there holds, for every $\lambda\in(0,1)$, for every sequence $(t_n)_{n\in\N}$ diverging to $+\infty$, and for every bounded sequence $(x'_n)_{n\in\N}$ in~$\R^{N-1}$,
\be\label{convphi}
u(t_n+t,x'_n+x',X_\lambda(t_n,x'_n)+x_N)\longrightarrow\vp^*(x_N-c^*t+(\vp^*)^{-1}(\lambda))\ \text{ as }n\to+\infty,
\ee
in $C^{1;2}_{loc}(\R_t\times\R_{x'}^{N-1})$ and uniformly with respect to $x_N\in\R$. If one further assumes~\eqref{gamma'to0}, then the above limit holds for every sequence $(x'_n)_{n\in\N}$ in $\R^{N-1}$, bounded or not.
\end{conjecture}

As before, by Proposition~\ref{pro:asymp}, the second conclusion does not hold in general if assumption~\eqref{gamma'to0} is replaced by $\lim_{|x'|\to+\infty}\gamma(x')/|x'|=0$. On the other hand, Conjecture~\ref{conj3}, and especially its second part, holds if $\gamma$ is bounded, for some classes of functions $f$, see~\cite{BH1,MN,MNT,P2,RR1,RR2}.

\begin{remark}\label{rem37} 
The conditions used in the main new results of this section, namely Theorems~$\ref{thm:subgraph}$ and~$\ref{thm:conical}$ are in general not related to the conditions~\eqref{hyp:U} and~\eqref{dUrho} mentioned in Section~$\ref{sec2}$ and used in~{\rm{\cite{HR1}}} to estimate the spreading speeds and upper level sets of solutions with general initial supports which may not be subgraphs. Conditions~\eqref{hyp:U} and~\eqref{dUrho} are needed to get a global picture of the spreading properties in all directions, whereas conditions~\eqref{defu0bis} and~\eqref{gamma<0}, or those of Theorem~$\ref{thm:conical}$, are used to derive local flattening properties in the direction $\mathrm{e}_N$. More precisely, in Theorem~$\ref{thm:subgraph}$, the conditions~\eqref{defu0bis} and~\eqref{gamma<0} do not imply~\eqref{hyp:U} and~\eqref{dUrho} in general, and conversely. For instance, on the one hand, if $\alpha<0$, if $(\varepsilon_k)_{k\in\N}$ is any sequence of positive real numbers converging to $0$, and~if
$$\gamma(x'):=\alpha\,|x'|+|\alpha|\sum_{k=1}^\infty k\times\1_{(k,k+\varepsilon_k)}(|x'|),$$
then $\gamma\in L^\infty_{loc}(\R^{N-1})$ and satisfies~\eqref{gamma<0}, while $U$ defined by~\eqref{defu0bis} does not satisfy~\eqref{hyp:U} or~\eqref{dUrho}, for any $\rho>0$, since $\mc{B}(U)=\{e=(e',e_N)\in\Sph:e_N>0\}$, $\mc{U}(U_\rho)=\{e=(e',e_N)\in\Sph:e_N\le\alpha|e'|\}$, and $\varepsilon_k\to0$ as $k\to+\infty$. On the other hand, if $\gamma(x')=\alpha\,|x'|$ with $\alpha>0$, then~\eqref{hyp:U} and~\eqref{dUrho} hold for any $\rho>0$, while~\eqref{gamma<0} is not satisfied, nor is actually any of the conditions {\rm{(i)-(iv)}} of Theorem~$\ref{thm:conical}$. Nevertheless, any of the conditions~{\rm{(i)}},~{\rm{(ii)}} or~{\rm{(iv)}} of Theorem~$\ref{thm:conical}$ implies~\eqref{hyp:U} and~\eqref{dUrho}, for any $\rho>0$. Condition~{\rm{(iii)}} of Theorem~$\ref{thm:conical}$ implies~\eqref{dUrho}, but not~\eqref{hyp:U} in general: for instance, if $\alpha<0$ and if
$$\gamma(x'):=\alpha\int_0^{|x'|}\sum_{k=1}^\infty k!\times\1_{(k!,k!+1)}(r)\,dr,$$
then $\gamma\in L^\infty_{loc}(\R^{N-1})$ and satisfies condition~{\rm{(iii)}} of Theorem~$\ref{thm:conical}$, but the set $U$ defined in~\eqref{defu0bis} is such that $\mc{B}(U)=\{e=(e',e_N)\in\Sph:e_N>0\}$ and $\mc{U}(U_\rho)=\{e=(e',e_N)\in\Sph:e_N\le\alpha|e'|\}$ for any $\rho>0$, whence~\eqref{hyp:U} does not hold.
\end{remark}


\section{Logarithmic lag: old and new results}\label{sec4}	

In this section, we discuss the large-time gaps, as defined in~\eqref{gaps}, between the position of the level sets of the solutions and the points moving with a constant speed equal to the spreading speed, in a given direction $e$. We first review in Section~\ref{seclag1} some standard results for compactly supported or almost-planar initial conditions. We then state in Section~\ref{seclag2} some new results on the logarithmic lag in the Fisher-KPP case.


\subsection{Known results for bounded or almost-planar initial supports}\label{seclag1}

Let us start with invading solutions $u$ of~\eqref{homo} with compactly supported initial conditions (we recall that the invading solutions are those which converge to $1$ as $t\to+\infty$ locally uniformly in $\R^N$). The estimates of the position at large time of the level sets of $u$ strongly depend on the reaction function $f$ and on the dimension $N$. On the one hand, when $f$ is of the bistable type~\eqref{bistable} with $\int_0^1f(s)ds>0$ (in which case Hypotheses~\ref{hyp:invasion} and~\ref{hyp:minimalspeed} are fulfilled), then, in dimension $N=1$, the solution $u$ develops into a pair of diverging fronts, that is, there are $x_\pm\in\R$ such that $u(t,x)-\varphi(\pm x-c_0t+x_{\pm})\to0$ as $t\to+\infty$ uniformly in $x\in I_\pm$, with $I_-:=(-\infty,0]$ and $I_+:=[0,+\infty)$, where $\varphi$ is the profile of a (unique up to shifts) traveling front connecting $1$ to $0$, with speed $c_0>0$, see~\cite{FM}. In particular, for every $\lambda\in(0,1)$,
$$\limsup_{t\to+\infty}\ d_{\mc{H}}\big(E_\lambda(t),\{-c_0t,c_0t\}\big)<+\infty.$$
This last property also holds when $f$ is of the ignition type~\eqref{ignition} by~\cite{U2}. In dimensions $N\ge2$, then a logarithmic lag occurs, due to curvature effects, namely it is shown in~\cite{U2} that, for both ignition~\eqref{ignition} or bistable~\eqref{bistable} reactions $f$ with $\int_0^1f(s)ds>0$,
$$\limsup_{t\to+\infty}\ d_{\mc{H}}\big(E_\lambda(t),\partial B_{c_0t-((N-1)/c_0)\log t}\big)<+\infty.$$

On the other hand, when $f$ is of the Fisher-KPP type~\eqref{kpp0}, then, in dimension $N=1$, there are $x_\pm\in\R$ such that $u(t,x)-\varphi^*(\pm x-c^*t+(3/c^*)\log t+x_\pm)\to0$ as $t\to+\infty$ uniformly in $x\in I_\pm$, where $\varphi^*$ is a profile of a traveling front connecting $1$ to $0$, with minimal speed $c^*=2\sqrt{f'(0)}$, see~\cite{B,HNRR,L,NRR,U1}. Therefore, for every $\lambda\in(0,1)$,
$$\limsup_{t\to+\infty}\ d_{\mc{H}}\Big(E_\lambda(t),\Big\{-c^*t+\frac{3}{c^*}\log t,c^*t-\frac{3}{c^*}\log t\Big\}\Big)<+\infty.$$ 
Notice the presence of a logarithmic lag, even in dimension $N=1$, which is due to the sublinearity of $f$ and its saturation at the value $1$. In dimensions $N\ge2$, the same curvature effects as for the bistable case are responsible of an additional logarithmic lag~$((N-1)/c^*)\log t$, that is,
$$\limsup_{t\to+\infty}\ d_{\mc{H}}\big(E_\lambda(t),\partial B_{c^*t-((N+2)/c^*)\log t}\big)<+\infty,$$
see~\cite{G}. More precisely, it is known that
$$\sup_{x\in\R^N\setminus\{0\}}\Big|u(t,x)-\varphi^*\Big(|x|-c^*t+\frac{N+2}{c^*}\,\log t+a(\hat{x})\Big)\Big|\to0\ \hbox{ as }t\to+\infty,$$
for some Lipschitz continuous function $a$ defined in~$\Sph$, see~\cite{D,RRR}. 

Consider now the case of almost-planar initial conditions of the type~\eqref{defu0}. Namely, when $u_0=\1_U$ and $U$ is trapped between two parallel half-spaces, say
$$\big\{x\in\R^N:x_N\le 0\big\}\subset U\subset\big\{x\in\R^N:x_N\le A\big\}$$
for some $A>0$, then it follows from the one-dimensional case and the comparison principle that $\limsup_{t\to+\infty}d_{\mc{H}}(E_\lambda(t),\{x_N=c_0t\})<+\infty$ if $f$ is of the ignition or bistable types~\eqref{ignition} or~\eqref{bistable}, and one even has
$$\sup_{x\in\R^N}\big|u(t,x)-\varphi(x_N-c_0t+a(t,x'))\big|\to0\hbox{ as $t\to+\infty$},$$
for some bounded function $a:(0,+\infty)\times\R^{N-1}\to\R$, see~\cite{MN,P2,RR1,U2}. In the Fisher-KPP case~\eqref{kpp0}, then $\limsup_{t\to+\infty}d_{\mc{H}}(E_\lambda(t),\{x_N=c^*t-(3/c^*)\log t\})<+\infty$ in any dimension $N\ge2$, and
$$\sup_{x\in\R^2}\Big|u(t,x)-\varphi^*\Big(x_2-c^*t+\frac{3}{c^*}\log t+a(t,x_1)\Big)\Big|\to0\hbox{ as $t\to+\infty$},$$
in dimension $N=2$ with $f(s)=s(1-s)$, for some bounded function $a:(0,+\infty)\times\R\to\R$, see~\cite{RR2}.


\subsection{New results}\label{seclag2}

We start with general logarithmic-in-time lower and upper bounds of the distance between the upper level sets $F_\lambda(t)$ of the solutions of~\eqref{homo}-\eqref{defu0} and the $c^*t$-neighborhoods of $U$ for general sets $U$, in the Fisher-KPP case~\eqref{kpp0}. We recall that Hypotheses~\ref{hyp:invasion} and~\ref{hyp:minimalspeed} hold in this case, as well as the hair trigger effect, and that $c^*=2\sqrt{f'(0)}$.

\begin{theorem}\label{thlag}
Assume that $f$ is of the Fisher-KPP type~\eqref{kpp0}, and let~$u$ be the solution of~\eqref{homo}-\eqref{defu0} for some $U\subset\R^N$ with $N\ge2$. Then, for any $\lambda\in(0,1)$, there exists $R>0$ such that
\begin{equation}\label{E<U}
F_\lambda(t) \subset U+B_{c^*t+\frac{N-2}{c^*}\log t+R}\quad\text{for all }t\geq1.
\end{equation}
If in addition there exists $\rho>0$ such that $U_\rho\neq\emptyset$~and~\eqref{dUrho} is fulfilled, then it also holds that
\begin{equation}\label{E>U}
F_\lambda(t)+B_R\supset U+B_{c^*t-\frac{N+2}{c^*}\log t}\quad\text{for all $t\ge1$},
\end{equation}
hence, in particular, 
\begin{equation}\label{dH2}
d_{\mc{H}}\big(F_\lambda(t) \,,\, U+B_{c^*t}\big)=O(\log t)\ \text{ as }t\to+\infty.
\end{equation}
\end{theorem}

The proof of Theorem~\ref{thlag} is given in Section~\ref{sec61}. We point out that the estimate~\eqref{E<U} is not sharp in some specific cases, for instance when $U$ is bounded or when it is a half space. However, together with \eqref{E>U}, it suffices to derive the logarithmic-in-time bound~\eqref{dH2}. It is obtained by estimating the position of the level sets of some solutions of the linearized equation, using the heat kernel. One may wonder whether a direct analysis on the nonlinear equation without passing to the linearized problem --~in the spirit for instance of~\cite{RRR}~-- would provide a sharper estimate for a general initial support $U$. We leave this question to future investigations.

Next, we turn to the question of the position of the level sets of solutions in a specific case. Namely, we consider a solution to~\eqref{homo} with an initial condition given by~\eqref{defu0bis} with~$\gamma$ bounded from above, and investigate the lag between the position of the level sets of~$u$ behind $c^*t$ in the direction~$x_N$. By comparison, we know from the results of Section~\ref{seclag1} that, up to an additive constant, the lag is between $(3/c^*)\log t$, which is the lag in the $1$-dimensional case, and $((N+2)/c^*)\log t$, which is the lag in the case of compactly supported initial conditions.
Under the notation~\eqref{defX}, this means that the lag $c^*t-X_\lambda(t,x')$ satisfies, for every~$\lambda\in(0,1)$ and $x'\in\R^{N-1}$,
\be\label{lagbetween}
\frac{3}{c^*}\,\log t+O(1)\leq c^*t-X_\lambda(t,x')\leq \frac{N+2}{c^*}\,\log t+O(1)\ \ \hbox{ as }t\to+\infty.
\ee
A natural question is whether or not this lag is equal to one of these bounds or whether it may take some intermediate values. Our next result states that, for an initial condition~$u_0$ satisfying~\eqref{defu0bis} with a function $x'\mapsto\gamma(x')$ tending to $-\infty$ as $|x'|\to+\infty$ faster than a suitable multiple of the logarithm, the lag coincides with the above upper bound, that is, the position of the level sets of $u$ in the direction $x_N$ is the same as when the initial condition is compactly supported. 

\begin{theorem}\label{thm:normdelay}
Assume that $f$ is of the strong Fisher-KPP type~\eqref{fkpp}, and let $u$ be the solution of~\eqref{homo} with an initial condition $u_0$ satisfying~\eqref{defu0bis}. If
\Fi{gamma<}
\limsup_{|x'|\to+\infty}\,\frac{\gamma(x')}{\log(|x'|)}<-\frac{2(N-1)}{c^*},
\Ff
then
\be\label{lag1}
X_\lambda(t,x')=c^*t-\frac{N+2}{c^*}\,\log t+O(1)\ \hbox{ as }t\to+\infty,
\ee
locally uniformly with respect to $\lambda\in(0,1)$ and $x'\in\R^{N-1}$, and the inequality ``$\leq$'' holds true in the above formula locally uniformly in $\lambda\in(0,1)$ and uniformly in $x'\in\R^{N-1}$.
\end{theorem}

If the upper bound for $\gamma$ in~\eqref{gamma<} is relaxed, we expect the lag of the solution with respect to the critical front to differ from the one associated with compactly supported initial data, that we recall is $((N+2)/c^*)\log t$. We derive the following lower bound for the lag.

\begin{proposition}\label{pro:lag}
Assume that $f$ is of the strong Fisher-KPP type~\eqref{fkpp} and let $u$ be the solution of~\eqref{homo} with an initial condition $u_0$ satisfying~\eqref{defu0bis}. If there is $\sigma\ge-(N-1)$ such that
\be\label{hyplag}
\limsup_{|x'|\to+\infty}\frac{\gamma(x')}{\log|x'|}\le\frac{2\sigma}{c^*},
\ee
then, for any $\lambda\in(0,1)$,
\be\label{ineqlag}
X_\lambda(t,x')\leq c^*t-\frac{3-\sigma}{c^*}\,\log t+o(\log t)\ \hbox{ as }t\to+\infty,
\ee
locally uniformly with respect to $x'\in\R^{N-1}$.
\end{proposition}

Theorem~\ref{thm:normdelay} and Proposition~\ref{pro:lag} are shown in Section~\ref{sec62}. Property~\eqref{ineqlag} means that the lag $c^*t-X_\lambda(t,x')$ is at least~$((3-\sigma)/c^*)\log t+o(\log t)$ as $t\to+\infty$. We point out that this holds even for positive $\sigma$. We conjecture that the estimate~\eqref{ineqlag} is actually sharp, in the sense that if the $\limsup$ is replaced by a limit in~\eqref{hyplag} and the inequality by an equality, then the lag should precisely be
$$c^*t-X_\lambda(t,x')=\frac{3-\sigma}{c^*}\log t+o(\log t)\quad\text{as $t\to+\infty$},$$
for every $\lambda\in(0,1)$ and~$x'\in\R^{N-1}$. When $\sigma=0$, this formula gives the $1$-dimensional lag, which is indeed the correct one for planar initial data. The above formula would also mean that the constant $-2(N-1)/c^*$ in~\eqref{gamma<} is optimal for the lag to be equivalent to that of solutions with compactly supported initial conditions. Lastly, it would provide a continuum of logarithmic lags with factors ranging in the whole half-line $(-\infty,(N+2)/c^*]$. In particular, solutions with initial conditions of the type~\eqref{defu0bis} with~$\gamma(x')\sim(6/c^*)\log|x'|$ as~$|x'|\to+\infty$ would have no logarithmic lag, i.e., the same position~$c^*t$ along the $x_N$-axis as that of the planar front moving in the direction $\mathrm{e}_N$, up to a $o(\log t)$ term as $t\to+\infty$. On the other hand, a subgraph $U$ satisfying~$\gamma(x')\sim \kappa\log|x'|$ as~$|x'|\to+\infty$ for some $\kappa>6/c^*$ would lead to a negative logarithmic lag, i.e., the position of the solution would be ahead of that of the front by a logarithmic-in-time term (observe that the term is linear in time when $\gamma(x')\sim \alpha |x'|$ as~$|x'|\to+\infty$ with $\alpha>0$, according to formulas~\eqref{FGgeneral} and~\eqref{Wshift}).

Using Theorem~\ref{thm:normdelay} we are able to prove Conjecture~\ref{conj1} about the flattening of the level sets under the hypotheses of that theorem.

\begin{corollary}\label{cor:DGlag}
Assume that $f$ is of the strong Fisher-KPP type~\eqref{fkpp} and let $u$ be the solution of~\eqref{homo} with an initial condition $u_0$ satisfying~\eqref{defu0bis} and~\eqref{gamma<}. Then the following hold:
\begin{itemize}
\item[{\rm{(i)}}] the conclusion~\eqref{limit2} of Conjecture~$\ref{conj1}$ holds, and even locally uniformly with respect to $\lambda\in(0,1)$, that is,
$$\nabla_{\!x'}X_\lambda(t,x')\to0\ \hbox{ as }t\to+\infty,\hbox{ locally uniformly in $x'\in\R^{N-1}$ and in $\lambda\in(0,1)$};$$
\item[{\rm{(ii)}}] for any $\lambda\in(0,1)$ and~$x'_0\in\R^{N-1}$, the limit function
$$\t u(t,x_N):=\lim_{s\to+\infty}u(s+t,x',X_\lambda(s,x'_0)+x_N),$$
which exists $($up to subsequences$)$ locally uniformly with respect to $(t,x',x_N)\in\R\times\R^N$, is independent of $x'$ and satisfies
$$\lim_{x_N\to-\infty}\t u(t,x_N+c^*t)=1\ \hbox{ and }\ \lim_{x_N\to+\infty}\t u(t,x_N+c^*t)=0,$$ 
uniformly with respect to $t\in\R$.
\end{itemize}
\end{corollary}

Corollary~\ref{cor:DGlag}, which is proved in Section~\ref{sec62}, shows that, in the large-time limit, the solution approaches a one-dimensional entire solution whose level sets move in the direction~$\mathrm{e}_N$ with an average velo\-city equal to the minimal speed~$c^*$. It is then natural to expect that $\t u(t,x_N)=\varphi^*(x_N-c^*t+(\varphi^*)^{-1}(\lambda))$ for all $(t,x_N)\in\R^2$, where $\varphi^*$ is the front profile connecting~$1$ and~$0$ with minimal speed $c^*$. That would correspond to property~\eqref{convphi} in Conjecture~\ref{conj3}.

\begin{remark}
Condition~\eqref{hyp:U}, used in~{\rm{\cite{HR1}}} to estimate the spreading speeds of solutions with more general reactions $f$, is not assumed in Theorem~$\ref{thlag}$. Regarding condition~\eqref{dUrho} used in~{\rm{\cite{HR1}}} too, it is assumed here only for property~\eqref{E>U} and the $O(\log t)$ estimate~\eqref{dH2} holding under the KPP assumption~\eqref{kpp0}. Notice that the estimate~\eqref{dH2} is more precise than the $o(t)$ estimate derived in~{\rm{\cite{HR1}}}, for more general reactions $f$. Furthermore, as in Remark~$\ref{rem37}$, conditions~\eqref{hyp:U} and~\eqref{dUrho} used in~{\rm{\cite{HR1}}} for the spreading speeds in all directions are in general not related to the conditions~\eqref{gamma<} and~\eqref{hyplag} used in Theorem~$\ref{thm:normdelay}$, Proposition~$\ref{pro:lag}$ and Corollary~$\ref{cor:DGlag}$, to estimate the position of the level sets in the direction $\mathrm{e}_N$ locally with respect to the orthogonal variables. For instance, if $\alpha<\beta<0$, if $(\varepsilon_k)_{k\in\N}$ is any sequence of positive real numbers converging to $0$, and if
$$\gamma(x')=\alpha\,|x'|+|\beta|\sum_{k=1}^\infty k\,\1_{(k,k+\varepsilon_k)}(|x'|),$$
then conditions~\eqref{gamma<} and~\eqref{hyplag} are satisfied, but the set $U$ defined by~\eqref{defu0bis} does not satisfy~\eqref{hyp:U} or~\eqref{dUrho} $($here, $\mc{B}(U)=\{e=(e',e_n)\in\Sph:e_n>(\alpha+|\beta|)|e'|\}$ and $\mc{U}(U_\rho)=\{e=(e',e_n)\in\Sph:e_n\le\alpha|e'|\}$ for any $\rho>0$$)$. Conversely, if $\gamma(x')=\alpha|x'|$ with $\alpha>0$, then~\eqref{hyp:U} and~\eqref{dUrho} are satisfied, but~\eqref{gamma<} or~\eqref{hyplag} are not.
\end{remark}

	
\section{The subgraph case: proofs of Theorems~\ref{thm:subgraph} and~\ref{thm:conical}, and Proposition~\ref{pro:asymp}}\label{sec5}

Section~\ref{sec51} is devoted to the proof of Theorem~\ref{thm:subgraph} about the flatness property of the level sets of solutions at large time if the initial support is below a graph which is not coercive at infinity. Section~\ref{sec52} contains the proofs of other flatness results and weaker versions of Conjecture~\ref{conj1}. In Sections~\ref{sec53} and~\ref{sec54}, we respectively prove Theorem~\ref{thm:conical} on the case of asymptotically conical or more general initial support, and Proposition~\ref{pro:asymp} on the counterexample to the global flatness of the level sets even if the initial support is asymptotically flat.


\subsection{Proof of Theorem~\ref{thm:subgraph}}\label{sec51}

We start with two auxiliary lemmas that will be used in the proof of Theorem~\ref{thm:subgraph} as well as in Section~\ref{sec52}.

\begin{lemma}\label{lem:X/t}
Assume that Hypothesis~$\ref{hyp:minimalspeed}$ holds $($hence Hypothesis~$\ref{hyp:invasion}$ as well$)$. Let $c^*>0$ be the minimal speed given in Section~$\ref{sec22}$, and let $u$ be a solution of~\eqref{homo} with an initial datum~$u_0$ given by~\eqref{defu0bis} with $\gamma$ satisfying~\eqref{gamma<0}. Then, for every $\lambda\in(0,1)$ and every~$x'_0\in\R^{N-1}$, there holds that	
\be\label{X=ct+o}
X_\lambda(t,x'_0)=c^*t+o(t)\ \hbox{ as }t\to+\infty,
\ee
and moreover
\be\label{X<ct}
\forall\,\alpha>0,\quad\max_{x'\in\overline{B'_{\alpha t}(x'_0)}}X_\lambda(t,x')\leq c^*t+o(t)\ \hbox{ as }t\to+\infty.
\ee	
\end{lemma}

\begin{proof}
Since hypothesis~\eqref{gamma<0} on the initial datum $u_0$ is invariant by translation of the coordinate system of $\R^{N-1}$, we can restrict without loss of generality to the case
$$x'_0=0.$$
Fix $\lambda\in(0,1)$. Because~$u_0$ is given by~\eqref{defu0bis} with $\gamma\in L^\infty_{loc}(\R^{N-1})$, there is $x_0\in\R^N$ such that $1\ge u_0\ge\1_{B_\rho(x_0)}$ in~$\R^N$, with $\rho>0$ given by Hypothesis~\ref{hyp:invasion}. Property~\eqref{c<c*} and the monotonicity of $u(t,x)$ with respect to $x_N$ then imply that
$$\liminf_{t\to+\infty}\frac{X_\lambda(t,0)}{t}\ge c^*.$$

It remains to show~\eqref{X<ct}. Together with the previous formula,~\eqref{X<ct} will then yield~\eqref{X=ct+o}. To show~\eqref{X<ct}, we compare $u$ with the same functions $(u_\e)_{\e>0}$ as in the paragraph following Proposition~\ref{pro:asymp}, i.e., with the solutions $u_\e$ of~\eqref{homo} with initial conditions $\1_{U^\varepsilon}$, where $U^\varepsilon:=\{(x',x_N)\in\R^N:x_N\le M_\varepsilon+\varepsilon|x'|\}$ and $M_\varepsilon\in\R$ is chosen so that $\gamma(x')\leq M_\varepsilon+\varepsilon|x'|$ for $x'\in\R^{N-1}$, which is possible thanks to assumption~\eqref{gamma<0}. It follows that $u\leq u_\e$ for any $\e>0$. The set~$U^\varepsilon$ satisfies~\eqref{dUrho}, for any given $\e>0$, therefore \cite[Theorem~2.4]{HR1} applies and yields the asymptotic formula~\eqref{dH} for the upper level sets $F_\lambda^\e(t)$ of $u_\e(t,\.)$ in terms of $t^{-1}U^\e+B_{c^*}$. Since the sets $t^{-1}U^\e$ approach $\t U^\varepsilon:=\{(x',x_N)\in\R^N:x_N\le \varepsilon|x'|\}$  in Hausdorff distance as $t\to+\infty$, the formula rewrites~as
\Fi{dHepsilon}
d_{\mc{H}}\big(t^{-1}F_\lambda^\e(t) \,,\, \t U^\e+B_{c^*}\big)\to0\ \text{ as }t\to+\infty.
\Ff
Take now $\alpha>0$ and $c>c_*$. Consider $x'\in \ol{B'_{\alpha}}$. For $0<\e<c/\alpha$ one has that $(x',c)\notin \t U^\e$, hence $\dist((x',c),\t U^\e)$ is attained at some point $z\in\partial \t U^\e$. It is then straightforward to check that 
$$\dist((x',c),\t U^\e)=\frac{c-\e|x'|}{\sqrt{1+\e^2}}\geq\frac{c-\e\alpha}{\sqrt{1+\e^2}}.$$
We can then choose $\e>0$ small enough in such a way that the latter term is larger than~$c^*$, and therefore $\min_{x'\in\ol{B'_{\alpha}}}\dist((x',c),\t U^\e)>c^*$. By virtue of~\eqref{dHepsilon}, we deduce that $\max_{x'\in\ol{B'_{\alpha}}} u_\e(t,tx',ct)< \lambda$ for $t$ sufficiently large, and thus the same is true for $u$. This means that $\max_{x'\in\overline{B'_{\alpha t}(0)}}X_\lambda(t,x')<ct$ for $t$ large. Property~\eqref{X<ct} then follows by the arbitrariness of $c>c^*$.
\end{proof}

\begin{lemma}\label{lem:tech}
Assume that Hypothesis~$\ref{hyp:minimalspeed}$ holds, hence Hypothesis~$\ref{hyp:invasion}$ as well, for some~$\theta\in(0,1)$ and $\rho>0$. Let $u$ be a solution of~\eqref{homo} with an initial datum~$u_0$ given by~\eqref{defu0bis} with $\gamma$ satisfying~\eqref{gamma<0}, and let $F_\theta(t)$ be the upper level set~$\{x\in\R^N:u(t,x)>\theta\}$ and~$(X_\lambda)_{\lambda\in(0,1)}$ be the functions given by~\eqref{defX}. Then, for every $\lambda\in(0,1)$ and $\omega>0$, there exists $\bar R>0$ such that 
\be\label{tech}
\forall\,x'_0\in\R^{N-1},\quad\liminf_{t\to+\infty}\left(\sup_{x'\in\partial B'_{\bar R}(x'_0)}\dist\big((x',X_\lambda(t,x'_0)+\omega\bar R)\,,\,\R^N\setminus F_\theta(t)\big)\right)\leq\rho.
\ee
\end{lemma}

\begin{proof}
Fix a real number $c$ such that
\be\label{defc0}
\frac{c^*}{\sqrt{1+\omega^2}}<c<c^*.
\ee
Let $v$ be the solution of~\eqref{homo} with initial condition $v_0:=\theta\,\1_{B_\rho}$. By~\eqref{c<c*}-\eqref{c>c*}, the function~$v$ spreads with the speed $c^*$. In particular, we can find $T>0$ such that
\be\label{vlambda0}
\min_{|x|\le cT}v(T,x)\ge\lambda.
\ee
Call
$$\bar R:=\frac\omega{\sqrt{1+\omega^2}}\,cT.$$
For all $y'\in\R^{N-1}$ such that $|y'|=\bar R$, we compute
\begin{align*}
\big|(0,cT\sqrt{1+\omega^2})-(y',\omega\bar R)\big|&=cT\sqrt{\frac{\omega^2}{1+\omega^2}+\Big(\sqrt{1+\omega^2}-\frac{\omega^2}{\sqrt{1+\omega^2}}\Big)^2}=cT.
\end{align*}
It follows that
\Fi{v>l}
v\big(T,(0,cT\sqrt{1+\omega^2})-(y',\omega\bar R)\big)\geq\lambda.
\Ff

We now assume by way of contradiction that~\eqref{tech} does not hold. Namely, there exist  $x'_0\in\R^{N-1}$ and $\tau>0$ such that
$$\forall\,t\geq \tau,\quad \sup_{x'\in\partial B'_{\bar R}(x'_0)}\dist\big((x',X_\lambda(t,x'_0)+\omega\bar R)\,,\,\R^N\setminus F_\theta(t)\big)>\rho.$$
Because condition~\eqref{gamma<0} is invariant by translation of the coordinate system of~$\R^{N-1}$, we can assume without loss of generality that $x'_0=0$. Namely, for any $t\geq\tau$, there exists a point $y_t'\in\R^{N-1}$ with $|y_t'|=\bar R$ such that
$$u(t,x)>\theta\quad \text{for all }x\in B_\rho(y_t',X_\lambda(t,0)+\omega\bar R).$$
This means that, for all $t\geq\tau$, 
$$u\big(t,x+(0,X_\lambda(t,0))\big)\geq \theta\1_{B_\rho(y_t',\omega\bar R)}(x)= v_0(x-(y',\omega\bar R)),$$
hence, by comparison, thanks to \eqref{v>l} one infers
$$u\big(t+T,(0,cT\sqrt{1+\omega^2})+(0,X_\lambda(t,0))\big)\geq\lambda.$$
We have thereby shown that
$$\forall\, t\geq\tau,\quad X_\lambda(t+T,0)\geq X_\lambda(t,0)+cT\sqrt{1+\omega^2},$$
hence, by iteration,
$$\forall\, n\in\N,\quad X_\lambda(\tau+nT,0)\geq X_\lambda(\tau,0)+cnT\sqrt{1+\omega^2}.$$
Therefore,
$$\limsup_{t\to+\infty}\frac{X_\lambda(t,0)}{t}\ge\limsup_{n\to+\infty}\frac{X_\lambda(\tau+nT,0)}{\tau+nT}\geq\sqrt{1+\omega^2}\,c,$$
which is larger than $c^*$ by the choice of $c$. This is in contradiction with Lemma~\ref{lem:X/t}.
\end{proof}

\begin{proof}[Proof of Theorem~$\ref{thm:subgraph}$]
Throughout the proof, Hypothesis~\ref{hyp:minimalspeed} is assumed, thus Hypo\-thesis~\ref{hyp:invasion} holds too, by Section~\ref{sec22}. Let $\theta\in(0,1)$ and $\rho>0$ be given by Hypo\-thesis~\ref{hyp:invasion}, and let $c^*>0$ be the minimal speed of traveling fronts connecting $1$ to $0$. Let~$u$ be a solution to~\eqref{homo}, with an initial condition $u_0$ given by~\eqref{defu0bis}, where $\gamma:\R^{N-1}\to\R$ satisfies~\eqref{gamma<0}. The functions $X_\lambda:(0,+\infty)\times\R^{N-1}\to\R$ are given by~\eqref{defX}, for all $\lambda\in(0,1)$.

We will show~\eqref{liminf2}, which yields~\eqref{liminf}. To show~\eqref{liminf2}, we argue by way of contradiction. Namely, by assuming that~\eqref{liminf2} does not hold for some $\lambda\in[\theta,1)$ and some basis $(\mathrm{e}'_1,\cdots,\mathrm{e}'_{N-1})$ of $\R^{N-1}$, one will show that $u(T_n,x_n)=\lambda$ and $u(T_n+\tau_n,\xi_n)\ge\lambda$ for some sequences of large times $(T_n)$ and $(\tau_n)$ and points $(x_n)$ and $(\xi_n)$ of $\R^N$ having the same projections on $\R^{N-1}$ and such that the difference $(\xi_n-x_n)\cdot\mathrm{e}_N$ is large compared to $c^*\tau_n$. That will eventually lead to a spreading speed larger than $c^*$ in the direction $\mathrm{e}_N$, and then to a contradiction, thanks to Lemma~\ref{lem:X/t}.

Notice that the conclusion~\eqref{liminf2} could also be easily viewed as a consequence of Lemma~\ref{lem:tech} in dimension $N=2$. The arguments used below in the general case $N\ge2$ are actually more involved, and first require some notations.

\medskip
\noindent{\it Step 1: some notations}. In the sequel, we fix a basis $(\mathrm{e}'_1,\cdots,\mathrm{e}'_{N-1})$ of $\R^{N-1}$. The desired property~\eqref{liminf2} is invariant by multiplying any vector $\mathrm{e}'_i$ by any factor $\alpha_i\in\R^*$. Therefore, without loss of generality, one can assume in the sequel that each vector $\mathrm{e}'_i$ has unit norm in $\R^{N-1}$, that is,
$$|\mathrm{e}'_i|=1\quad \text{for each $1\le i\le N-1$}.$$
Observe that, for any $\epsilon=(\epsilon_i)_{1\le i\le N-1}\in\{-1,1\}^{N-1}$, one can choose a point $y'_{\epsilon}\in\R^{N-1}$ such that
$$\overline{B'_\rho(y'_\epsilon)}\subset\Big\{x'=\sum_{i=1}^{N-1}t_{i,\epsilon,x'}\epsilon_i\mathrm{e}'_i:t_{i,\epsilon,x'}\in\R^+\Big\},$$
where one recalls that the notation $B'_r(y')$ stands for the open Euclidean ball in $\R^{N-1}$ of center $y'\in\R^{N-1}$ and radius $r>0$. In the above formula, for any $x'\in\R^{N-1}$ and any $\epsilon=(\epsilon_1,\cdots,\epsilon_{N-1})\in\{-1,1\}^{N-1}$, the real numbers $t_{1,\epsilon,x'},\ldots,t_{N-1,\epsilon,x'}$ denote the (unique) coordinates of $x'$ in the basis $(\epsilon_1\mathrm{e}'_1,\cdots,\epsilon_{N-1}\mathrm{e}'_{N-1})$. One then defines a positive real number~$\rho'$ by
\be\label{defrho'}
\rho'=\max_{\epsilon\in\{-1,1\}^{N-1},\,x'\in\overline{B'_\rho(y'_\epsilon)},\,1\le i\le N-1}t_{i,\epsilon,x'}.
\ee

\medskip
\noindent{\it Step 2: the proof of~\eqref{liminf2}}. In addition to the basis $(\mathrm{e}'_1,\cdots,\mathrm{e}'_{N-1})$ of $\R^{N-1}$, we now fix any $\lambda\in[\theta,1)$. Assume by way of contradiction that~\eqref{liminf2} does not hold. Since the quantities involved in~\eqref{liminf2} are nonnegative and nonincreasing with respect to $R>0$, there exists then $\omega>0$ such that
\be\label{liminf3}
\forall\,R>0,\quad\sup_{x'_0\in\R^{N-1}}\Big[\liminf_{t\to+\infty}\Big(\min_{x'\in\overline{B'_R(x'_0)},\,1\le i\le N-1}|\nabla_{\!x'}X_\lambda(t,x')\cdot\mathrm{e}'_i|\Big)\Big]\ge3\,\omega.
\ee
We now fix a real number $c$ such that
\be\label{defc}
\frac{c^*}{\sqrt{1+\omega^2}}<c<c^*.
\ee
Let $v$ be the solution of~\eqref{homo} with initial condition $v_0:=\theta\,\1_{B_\rho}$. By~\eqref{c<c*}-\eqref{c>c*}, the function~$v$ spreads with the speed $c^*$. In particular, there is $T>0$ such that
\be\label{vlambda}
\min_{|x|\le ct}v(t,x)\ge\lambda\quad \text{for all }t\ge T.
\ee

Let us now consider any $n\in\N$ and apply~\eqref{liminf3} with $R=n+(N-1)\rho'>0$, with $\rho'>0$ given in~\eqref{defrho'}. There is then a point $x'_n\in\R^{N-1}$ such that
$$\liminf_{t\to+\infty}\Big(\min_{x'\in\overline{B'_{n+(N-1)\rho'}(x'_n)},\,1\le i\le N-1}|\nabla_{\!x'}X_\lambda(t,x')\cdot\mathrm{e}'_i|\Big)\ge2\omega.$$
Since the function $X_\lambda$ is at least of class $C^1$ in $(0,+\infty)\times\R^{N-1}$ from the implicit function theorem and the negativity of $\partial_{x_N}u$ in $(0,+\infty)\times\R^N$, it follows by continuity that there exist $T_n>0$ and $\epsilon_n=(\epsilon_{n,i})_{1\le i\le N-1}\in\{-1,1\}^{N-1}$ such that
\be\label{defTn}
\nabla_{\!x'}X_\lambda(t,x')\cdot(\epsilon_{n,i}\mathrm{e}'_i)\ge\omega\quad \text{for all $t\ge T_n$, $x'\in\overline{B'_{n+(N-1)\rho'}(x'_n)}$ and $1\le i\le N-1$}.
\ee
One then infers from the fundamental theorem of calculus and from the definitions of $y'_{\epsilon_n}$ and $\rho'$ in Step~2, that
$$X_\lambda(T_n,x'_n+n\,\epsilon_{n,1}\,\mathrm{e}'_1)\ge X_\lambda(T_n,x'_n)+\omega\,n$$
and then, for any $x'\in\overline{B'_\rho(y'_{\epsilon_n})}$,
\be\label{ineqs}\baa{rcl}
X_\lambda(T_n,x'_n+n\,\epsilon_{n,1}\,\mathrm{e}'_1+x') & \ge & \displaystyle X_\lambda(T_n,x'_n+n\,\epsilon_{n,1}\,\mathrm{e}'_1)+\sum_{i=1}^{N-1}\omega\,\underbrace{t_{i,\epsilon_n,x'}}_{\ge0}\vspace{3pt}\\
& \ge & X_\lambda(T_n,x'_n)+\omega\,n.\eaa
\ee
Call
\be\label{defzn}
z'_n=x'_n+n\,\epsilon_{n,1}\,\mathrm{e}'_1+y'_{\epsilon_n}\in\R^{N-1}\ \hbox{ and }z_n=\big(z'_n,X_\lambda(T_n,x'_n)+\omega\,n-\rho\big)\in\R^N.
\ee
For any $x=(x',x_N)\in\overline{B_\rho(z_n)}$, there holds $x'\in\overline{B'_\rho(z'_n)}=x'_n+n\,\epsilon_{n,1}\,\mathrm{e}'_1+\overline{B'_\rho(y'_{\epsilon_n})}$ and $x_N\le X_\lambda(T_n,x'_n)+\omega\,n$, hence
$$X_\lambda(T_n,x')\ge X_\lambda(T_n,x'_n)+\omega\,n\ge x_N$$
by~\eqref{ineqs}. From the definition~\eqref{defX} of $X_\lambda$ and the fact that $u$ is decreasing with respect to $x_N$ in~$(0,+\infty)\times\R^N$, one then infers that
$$u(T_n,\cdot)\ge\lambda\ge\theta\ \hbox{ in $\overline{B_\rho(z_n)}$}.$$
Hence, $u(T_n,\cdot)\ge v_0(\cdot-z_n)$ in $\R^N$, and
\be\label{uvTn}
u(T_n+t,\cdot)\ge v(t,\cdot-z_n)\ \hbox{ in $\R^N$ for all $t>0$}
\ee
from the maximum principle.

In addition to~\eqref{defzn}, let us now introduce a few other notations, for each $n\in\N$. Call
\be\label{defxnxin}
x_n=(x'_n,X_\lambda(T_n,x'_n))\in\R^N,\ \xi_n=\Big(x'_n,X_\lambda(T_n,x'_n)+|x_n-z_n|\,\frac{\sqrt{1+\omega^2}}{\omega}\Big)\in\R^N,
\ee
and
$$\tau_n=\frac{|\xi_n-z_n|}{c}.$$
Remember that the sequence $(|y'_{\epsilon_n}|)_{n\in\N}$ takes only a finite number of values, and is therefore bounded. It is then easy to check from~\eqref{defzn} and~\eqref{defxnxin} that
$$|x_n-z_n|\sim n\,\sqrt{1+\omega^2},\ \ |x_n-\xi_n|\sim n\,\frac{1+\omega^2}{\omega},\ \ |\xi_n-z_n|\sim n\,\frac{\sqrt{1+\omega^2}}{\omega},\ \ \tau_n\sim n\,\frac{\sqrt{1+\omega^2}}{c\,\omega},$$
as $n\to+\infty$. In other words, the angle between the segments $[z_n,x_n]$ and $[z_n,\xi_n]$ is almost right, and then the angle between the segments $[x_n,z_n]$ and $[x_n,\xi_n]$ is almost $\arccos(\omega/\sqrt{1+\omega^2})=\pi/2-\arctan\omega$. As a consequence, $\tau_n\to+\infty$ as $n\to+\infty$, and
$$\frac{|x_n-z_n|}{\tau_n}\,\frac{\sqrt{1+\omega^2}}{\omega}\to c\,\sqrt{1+\omega^2}>c^*\ \hbox{ as $n\to+\infty$},$$
by~\eqref{defc}. We can then fix $n_0\in\N$ such that
\be\label{defn0}
\tau_{n_0}\ge T\ \hbox{ and }\ \frac{|x_{n_0}-z_{n_0}|}{\tau_{n_0}}\,\frac{\sqrt{1+\omega^2}}{\omega}>c^*,
\ee
with $T>0$ defined in~\eqref{vlambda}.

Lastly,~\eqref{vlambda} and~\eqref{uvTn} yield
$$u\Big(T_{n_0}+\tau_{n_0},x'_{n_0},X_\lambda(T_{n_0},x'_{n_0})+|x_{n_0}-z_{n_0}|\,\frac{\sqrt{1+\omega^2}}{\omega}\Big)\!=\!u(T_{n_0}+\tau_{n_0},\xi_{n_0})\!\ge\!v(\tau_{n_0},\xi_{n_0}-z_{n_0})\!\ge\!\lambda,$$
hence $X_\lambda(T_{n_0}+\tau_{n_0},x'_{n_0})\ge X_\lambda(T_{n_0},x'_{n_0})+|x_{n_0}-z_{n_0}|\,\sqrt{1+\omega^2}/\omega$. Starting again from~\eqref{defTn} (applied with $n=n_0$) and repeating the above arguments, one infers that
$$u\Big(T_{n_0}+2\tau_{n_0},x'_{n_0},X_\lambda(T_{n_0},x'_{n_0})+2\,|x_{n_0}-z_{n_0}|\,\frac{\sqrt{1+\omega^2}}{\omega}\Big)\ge\lambda$$
and $X_\lambda(T_{n_0}+2\tau_{n_0},x'_{n_0})\ge X_\lambda(T_{n_0},x'_{n_0})+2\,|x_{n_0}-z_{n_0}|\,\sqrt{1+\omega^2}/\omega$. By an immediate induction, there holds
$$X_\lambda(T_{n_0}+k\tau_{n_0},x'_{n_0})\ge X_\lambda(T_{n_0},x'_{n_0})+k\,|x_{n_0}-z_{n_0}|\,\frac{\sqrt{1+\omega^2}}{\omega}$$
for all $k\in\N$. Therefore,
$$\limsup_{t\to+\infty}\frac{X_\lambda(t,x'_{n_0})}{t}\ge\frac{|x_{n_0}-z_{n_0}|}{\tau_{n_0}}\,\frac{\sqrt{1+\omega^2}}{\omega}>c^*$$
by~\eqref{defn0}. One has finally reached a contradiction with Lemma~\ref{lem:X/t}, and the proof of Theorem~\ref{thm:subgraph} is thereby complete.
\end{proof}


\subsection{Weaker versions of Conjecture~\ref{conj1} and counterexamples}\label{sec52}

We here prove three weaker versions of Conjecture~\ref{conj1}. Next, we exhibit some counter\-examples to the conclusions~\eqref{liminf}-\eqref{liminf2} of Theorem~\ref{thm:subgraph} when the non-coercivity assumption~\eqref{gamma<0} is not fulfilled.

The first result provides a refined upper bound for $\nabla_{\!x'}X_\lambda(x',t)$ for every sequence of times $t\to+\infty$, compared to the conclusion~\eqref{liminf} of Theorem~\ref{thm:subgraph}, at the price of taking the minimum on sets of $x'$ growing linearly in time.

\begin{proposition}\label{pro:trapped}
Under the same assumptions and with $u$ as in Theorem~$\ref{thm:subgraph}$, for any~$\lambda\in(0,1)$ and $\alpha>0$, there holds that
\Fi{trapped}
\min_{|x'|\leq\alpha  t}|\nabla_{\!x'}X_\lambda(t,x')|\ \rightarrow0\ \hbox{ as }t\to+\infty.
\Ff
\end{proposition}

\begin{proof}
Take $\lambda\in(0,1)$ and $\alpha>0$. Fix $\e>0$ and, for $t>0$, define the function $Y_t:\R^{N-1}\to\R$ by 
$$Y_t(x'):=X_\lambda(t,x')-\frac{\e}{t}|x'|^2.$$
It follows, on the one hand, that $Y_t(0)=c^*t+o(t)$ as $t\to+\infty$, thanks to  Lemma~\ref{lem:X/t}. On~the other hand,~\eqref{X<ct} yields, for $|x'|=\alpha t$, $Y_t(x')\leq (c^*-\e\alpha^2)t+o(t)$  as $t\to+\infty$. This shows that, for $t$ large enough, depending on~$\alpha$ and~$\e$, $Y_t$ has a local maximum at some $\xi'_t$ with $|\xi'_t|<\alpha t$, and thus there holds that
$$|\nabla_{\!x'}X_\lambda(t,\xi'_t)|=2\frac{\e}{t}|\xi'_t|<2\alpha\e.$$
This concludes the proof by the arbitrariness of $\e$.
\end{proof}

The second weaker version of Conjecture~\ref{conj1}, which nevertheless gives a more precise conclusion than the properties~\eqref{liminf}-\eqref{liminf2} of Theorem~\ref{thm:subgraph}, is concerned with positive functions $f$ of the type~\eqref{HTconditions}.

\begin{proposition}\label{flatteningHT}
Assume that $f$ satisfies~\eqref{HTconditions} and let $u$ be a solution of~\eqref{homo} with an initial datum~$u_0$ given by~\eqref{defu0bis}, where $\gamma$ satisfies~\eqref{gamma<0}. Then, for every $\lambda\in(0,1)$, there holds that
$$\liminf_{t\to+\infty}\Big(\min_{|x'|\leq R}|\nabla_{\!x'}X_\lambda(t,x')|\Big)\ \longrightarrow\ 0\ \hbox{ as }R\to+\infty,$$
and even
\be\label{liminf2e'}
\sup_{x'_0\in\R^{N-1}}\Big[\liminf_{t\to+\infty}\Big(\min_{x'\in\overline{B'_R(x'_0)}}|\nabla_{\!x'}X_\lambda(t,x')|\Big)\Big]\ \longrightarrow\ 0\ \hbox{ as }R\to+\infty.
\ee
\end{proposition}

\begin{proof}
Take $\lambda\in(0,1)$ and $\omega>0$. Recall that condition~\eqref{HTconditions} ensures the validity of Hypothesis~$\ref{hyp:invasion}$ for any $\theta\in(0,1)$ and $\rho>0$, see~\cite{AW}. We take in particular $\theta=\lambda/2$ and $\rho>0$ such~that 
\Fi{oscu}
\forall\, t\geq1,\ \forall\,|x-y|\leq 2\rho,\quad|u(t,x)-u(t,y)|\le\frac\lambda2,
\Ff
which is a possible by interior parabolic estimates. Consider then the positive number~$\bar R$ given by Lemma~\ref{lem:tech}, associated with such $\theta=\lambda/2$ and $\rho>0$, and also $\lambda,\omega$. Take $x'_0\in\R^{N-1}$. Then by Lemma~\ref{lem:tech} there exists a sequence of positive numbers $(t_n)_{n\in\N}$ diverging to $+\infty$ such that, for every $n\in\N$ and every $x'\in \partial B'_{\bar R}(x'_0)$, we can find $y_n\in\R^N$ with the properties
$$|y_n-(x',X_\lambda(t_n,x'_0)+\omega\bar R)|\leq 2\rho\quad\text{and}\quad u(t_n,y_n)\leq\theta=\frac\lambda2.$$
It is not restrictive to assume that the $(t_n)_{n\in\N}$ are larger than $1$, hence we derive from~\eqref{oscu}
$$\forall\,n\in\N,\ \forall\,x'\in\partial B'_{\bar R}(x'_0),\quad u(t_n,x',X_\lambda(t_n,x'_0)+\omega\bar R)\leq\lambda,$$
that is,
\Fi{X<X+o}
\forall\,n\in\N,\ \forall\,x'\in\partial B'_{\bar R}(x'_0),\quad X_\lambda(t_n,x')\leq X_\lambda(t_n,x'_0)+\omega\bar R.
\Ff
We now deduce from this a bound on $\nabla_{\!x'}X_\lambda(t_n,\.)$ at some point. Namely, for $n\in\N$, we consider the function $Y_n:\R^{N-1}\to\R$ defined by 
$$Y_n(x'):=X_\lambda(t_n,x')-\frac{\omega}{\bar R}|x'-x'_0|^2.$$
It follows from~\eqref{X<X+o} that $Y_n(x'_0)=X_\lambda(t_n,x'_0)\geq\max_{\partial B'_{\bar R}(x'_0)} Y_n$, hence the maximum of $Y_n$ in $\ol{B'_{\bar R}(x'_0)}$ is attained at some $\xi'_n\in B'_{\bar R}(x'_0)$. We infer that
$$|\nabla_{\!x'}X_\lambda(t_n,\xi'_n)|=2\frac{\omega}{\bar R}|\xi'_n-x'_0|<2\omega.$$
In the end, we have shown that, for any~$x'_0\in\R^{N-1}$,
$$\liminf_{t\to+\infty}\Big(	\min_{x'\in\overline{B'_{\bar R}(x'_0)}}|\nabla_{\!x'}X_\lambda(t,x')|\Big)\leq2\omega,$$
hence
$$\liminf_{t\to+\infty}\Big(	\min_{x'\in\overline{B'_{R}(x'_0)}}|\nabla_{\!x'}X_\lambda(t,x')|\Big)\leq2\omega$$
for all $R\ge\bar R$. By the arbitrariness of $\omega>0$, and recalling that $\bar R$ depends on $\lambda$ and $\omega$ but not on $x'_0$, we conclude that~\eqref{liminf2e'} holds.
\end{proof}

The third weaker version of Conjecture~\ref{conj1} deals with equations~\eqref{homo} with Fisher-KPP nonlinearities $f$ of the type~\eqref{fkpp}. It asserts that, under conditions~\eqref{defu0bis},~\eqref{gamma<0}, the level curves of $u$ become locally uniformly flat along sequences of times diverging to $+\infty$.

\begin{proposition}\label{cor:flatteningKPP}
Assume that $f$ is of the strong Fisher-KPP type~\eqref{fkpp}. Let $u$ be a solution of~\eqref{homo} with an initial condition of the type~\eqref{defu0bis} with $\gamma$ satisfying~\eqref{gamma<0}. Then
\be\label{abA}
\liminf_{t\to+\infty}\,\Big(\sup_{a\le\lambda\le b,\,|x'|\le A}\big|\nabla_{\!x'}X_\lambda(t,x')\big|\Big)=0
\ee
for every $\lambda\in(0,1)$, $0<a\le b<1$ and $A>0$.
\end{proposition}

\begin{proof}
Fix $A>0$, $0<a\le b<1$ and then any $a'$, $b'$ and $b''$ such that
$$0<a'<a\le b<b'<b''<1.$$
Let $\zeta$ be the solution of the ordinary differential equation $\dot\zeta(t)=f(\zeta(t))$ for $t\in\R$, with $\zeta(0)=a'$. Because of~\eqref{fkpp}, there is $\tau>0$ such that $\zeta(\tau)=b''$. Now, for~$\rho>0$, let $v_\rho$ denote the solution of~\eqref{homo} with initial condition
$$v_\rho(0,\cdot)=a'\1_{B_\rho}.$$
Since $f$ is Lipschitz continuous in $[0,1]$, it follows from parabolic estimates that $v_\rho(\tau,\cdot)\to b''$ as $\rho\to+\infty$, locally uniformly in $\R^N$. In particular, let us fix in the sequel a large enough real number~$\rho$ such that
\be\label{defrho2}
\rho>c^*\tau\ \hbox{ and }\ v_\rho(\tau,0)>b',
\ee
where we recall that $c^*=2\sqrt{f'(0)}$ is the minimal speed of traveling fronts connecting~$1$ to~$0$ in the Fisher-KPP case~\eqref{fkpp}.

We now claim that there exist $\e>0$ and $T>0$ such that
\be\label{claim4}\baa{l}
\forall\,t\ge T,\ \forall\,|x'|\le A,\ \forall\,\lambda\in[a,b],\vspace{3pt}\\
\displaystyle|\partial_{x_N}\!u(t,x',X_\lambda(t,x'))|\le\e\Longrightarrow a'\!<\!\min_{\overline{B_{\rho+A}(x',X_\lambda(t,x'))}}\!u(t,\cdot)\!\le\!\max_{\overline{B_{\rho+A}(x',X_\lambda(t,x'))}}\!u(t,\cdot)\!<\!b'.\eaa
\ee
Indeed, otherwise, there would exist a sequence of positive numbers $(t_n)_{n\in\N}$ diverging to~$+\infty$, a sequence $(x_n)_{n\in\N}$ in $\overline{B'_A}\times\R$ such that $\partial_{x_N}u(t_n,x_n)\to0$ as $n\to+\infty$, together with
\be\label{ab}\left\{\baa{l}
a\le u(t_n,x_n)\le b,\vspace{3pt}\\
\displaystyle\hbox{and either }\min_{\overline{B_{\rho+A}(x_n)}}u(t_n,\cdot)\!\le\!a'\!<\!a\ \hbox{ or }\ \max_{\overline{B_{\rho+A}(x_n)}}u(t_n,\cdot)\!\ge\!b'\!>\!b,\hbox{ for all }n\in\N.\eaa\right.
\ee
Up to extraction of a subsequence, the functions $(t,x)\mapsto u(t_n+t,x_n+x)$ converge in~$C^{1;2}_{loc}(\R\times\R^N)$ to a solution $u_\infty$ of~\eqref{homo} such that $0\le u_\infty\le1$ and $\partial_{x_N}u_\infty\le0$ in~$\R\times\R^N$ (remember that $\partial_{x_N}u<0$ in $(0,+\infty)\times\R^N$), while $\partial_{x_N}u_\infty(0,0)=0$. The strong parabolic maximum principle applied to the function~$\partial_{x_N}u_\infty$ then yields $\partial_{x_N}u_\infty\equiv 0$ in $(-\infty,0]\times\R^N$ and then in $\R\times\R^N$. Moreover, since the sequence $(x'_n)_{n\in\N}$ is bounded (in $\R^{N-1}$), it follows from~\cite[Theorem~7.2]{HR2}\footnote{\ The strong Fisher-KPP condition~\eqref{fkpp} is used in all results of~\cite{HR2}.} that $\nabla_{\!x'}u_\infty\equiv 0$ in $\R\times\R^N$. Finally, $\nabla u_\infty\equiv 0$ in $\R\times\R^N$ and there holds in particular $\max_{\overline{B_{\rho+A}(x_n)}}|u(t_n,\cdot)-u(t_n,x_n)|\to0$ as $n\to+\infty$, a contradiction with~\eqref{ab}. Therefore, the claim~\eqref{claim4} has been proved.

To complete the proof of Proposition~\ref{cor:flatteningKPP}, assume by way of contradiction that the conclusion~\eqref{abA} does not hold. Then, using~\eqref{DchiDu} and~\cite[Theorem~7.2]{HR2}, one gets that
$$\min_{a\le\lambda\le b,\,|x'|\le A}|\partial_{x_N}u(t,x',X_\lambda(t,x'))|\to0\ \hbox{ as }t\to+\infty.$$
Together with~\eqref{claim4}, there is then $T'>0$ such that, for every $t\ge T'$, there are $x'_t\in\overline{B'_A}$ and $\lambda_t\in[a,b]$ such that
$$a'<\min_{\overline{B_{\rho+A}(x'_t,X_{\lambda_t}(t,x'_t))}}u(t,\cdot)\le\max_{\overline{B_{\rho+A}(x'_t,X_{\lambda_t}(t,x'_t))}}u(t,\cdot)<b'.$$
Since $\overline{B_{\rho}(0,X_{\lambda_t}(t,x'_t))}\subset\overline{B_{\rho+A}(x'_t,X_{\lambda_t}(t,x'_t))}$, it then follows that
$$a'<\min_{\overline{B_\rho(0,X_{\lambda_t}(t,x'_t))}}u(t,\cdot)\le\max_{\overline{B_\rho(0,X_{\lambda_t}(t,x'_t))}}u(t,\cdot)<b'.$$
In particular, for every $t\ge T'$, one has on the one hand $X_{b'}(t,0)<X_{\lambda_t}(t,x'_t)-\rho$, and on the other hand $u(t,\cdot+(0,X_{\lambda_t}(t,x'_t)))\ge a'\1_{B_\rho}=v_\rho(0,\cdot)$ in $\R^N$. The maximum principle then yields in particular $u(t+\tau,0,X_{\lambda_t}(t,x'_t))\ge v_\rho(\tau,0)>b'$ from~\eqref{defrho2}, hence~$X_{b'}(t+\tau,0)>X_{\lambda_t}(t,x'_t)$. As a consequence,~$X_{b'}(t+\tau,0)>X_{b'}(t,0)+\rho$ for every $t\ge T'$, and thus
$$\limsup_{s\to+\infty}\frac{X_{b'}(s,0)}{s}\ge\frac{\rho}{\tau}>c^*$$
owing to~\eqref{defrho2}. This last formula is in contradiction with Lemma~\ref{lem:X/t}. As a conclusion,~\eqref{abA} has been proved.
\end{proof}

To complete this section, we present some counterexamples to the flatness properties~\eqref{liminf}-\eqref{liminf2},~\eqref{limit2} and~\eqref{limit3chi} of Theorems~\ref{thm:subgraph},~\ref{thm:conical} and Conjecture~\ref{conj1} when the assumptions~\eqref{gamma<0} or~\eqref{hypconical} are modified, and we show that~\eqref{limit2} and~\eqref{limit3chi} do not hold uniformly in general.

\begin{proposition}\label{pro46}
The following properties hold:
\begin{itemize}
\item[{\rm{(i)}}] if one assumes that $\liminf_{|x'|\to+\infty}\gamma(x')/|x'|\ge0$ instead of~\eqref{gamma<0}, the conclusions~\eqref{liminf}-\eqref{liminf2} of Theorem~$\ref{thm:subgraph}$ do not hold in general;
\item[{\rm{(ii)}}] even for $x'$-symmetric solutions $u$, the conclusion~\eqref{limit2} of Conjecture~$\ref{conj1}$ does not hold in general without the assumption~\eqref{gamma<0};
\item[{\rm{(iii)}}] even with the assumption~\eqref{gamma<0}, the conclusion~\eqref{limit2} of Conjecture~$\ref{conj1}$ does not hold in general uniformly with respect to $x'\in\R^{N-1}$;
\item[{\rm{(iv)}}] if $\ell>0$ in condition~\eqref{hypconical}, then the conclusion~\eqref{limit3chi} of Theorem~$\ref{thm:conical}$ cannot be uniform with respect to $x'\in\R^{N-1}$.
\end{itemize}
\end{proposition}

\begin{proof}
(i) To see it, consider for instance a bistable function $f$ satisfying~\eqref{bistable} with
\be\label{bistable2}
f'(0)<0,\ \ f'(1)<0\ \hbox{ and }\ \int_0^1f(s)ds>0.
\ee
In that case, there is a unique up to shift decreasing function~$\varphi:\R\to(0,1)$ and a unique speed $c^*>0$ such that $\varphi(x-c^*t)$ is a traveling front connecting $1$ to $0$ for~\eqref{homo}. Hence, Hypothesis~\ref{hyp:minimalspeed} is fulfilled. Consider now~\eqref{homo} in dimension $N=2$. For any angle $\beta\in(0,\pi/2)$, it is known that there is a $V$-shaped function~$\phi:\R^2\to(0,1)$ such that
$$\phi\Big(x_1,x_2-\frac{c^*}{\sin\beta}\,t\Big)$$
is a traveling front solving~\eqref{homo}, and in addition $\phi$ is even in~$x_1$ and, for every $\lambda\in(0,1)$, there exists an even function $\gamma_\lambda\in C^1(\R)$ for which there holds
\be\label{defphi}\left\{\baa{l}
\{(x_1,x_2)\in\R^2:\phi(x_1,x_2)=\lambda\big\}=\{(x_1,x_2)\in\R^2:x_2=\gamma_\lambda(x_1)\big\},\\
\displaystyle\gamma_\lambda'(x_1)\to\pm\frac{1}{\tan\beta}\ \hbox{ as $x_1\to\pm\infty$},\vspace{3pt}\\
\phi(x_1,x_2)\to0\ \hbox{(resp. $\to1$) as $x_2\!-\!\gamma_\lambda(x_1)\!\to\!+\infty$ (resp. as $x_2\!-\!\gamma_\lambda(x_1)\!\to\!-\infty$)},\vspace{3pt}\\
\qquad\qquad\qquad\qquad\qquad\qquad\qquad\qquad\qquad\qquad\qquad\qquad\ \hbox{uniformly in $x_1\in\R$},\eaa\right.
\ee
see~\cite{HMR1,HMR2,NT}. Moreover,
$$\sup_{a\le\lambda\le b,\,x_1\in\R}\partial_{x_2}\phi(x_1,\gamma_\lambda(x_1))<0$$
for every $0<a\le b<1$, and the function $\phi$ is decreasing in every direction $(\cos\omega,\sin\omega)$ with $|\omega-\pi/2|\le\beta$. Consider now any angle $\vartheta\in(0,\beta)$, let $\mathcal{R}$ be the rotation of angle $\vartheta$, and let $u$ be the solution of~\eqref{homo} with initial condition~\eqref{defu0bis} and $\gamma$ defined by
$$\{(x_1,x_2)\in\R^2:x_2=\gamma(x_1)\}=\mathcal{R}\big(\{(x_1,x_2)\in\R^2:x_2=\gamma_{1/2}(x_1)\}\big).\footnote{\ The function $\gamma$ is well defined in $\R$ since $\phi$ is decreasing in the direction $(\cos(\pi/2-\vartheta),\sin(\pi/2-\vartheta))$.}$$
Notice in particular that~\eqref{gamma<0} is not fulfilled. Instead, one has
$$\liminf_{|x'|\to+\infty}\frac{\gamma(x')}{|x'|}>0.$$
It follows from applications of some results of~\cite{RR1} that the solution $v$ of~\eqref{homo} with initial condition $\1_{\{x_2\le\gamma_{1/2}(x_1)\}}$ satisfies
$$v(t,x_1,x_2)-\phi\Big(x_1,x_2-\frac{c^*}{\sin\beta}\,t+a\Big)\to0\ \hbox{ as~$t\to+\infty$ in $C^2(\R^2)$},$$
for some $a\in\R$. Since $u(t,x_1,x_2)=v(t,\mathcal{R}^{-1}(x_1,x_2))$ for all~$t\ge0$ and~$(x_1,x_2)\in\R^2$ and since $\gamma_\lambda'(x_1)\to\pm1/\tan\beta>0$ as $x_1\to\pm\infty$, one then infers that, for every $\lambda\in(0,1)$,
$$\partial_{x_1}X_\lambda(t,x_1)\to\frac{1}{\tan(\beta-\vartheta)}>0\ \hbox{ as $t\to+\infty$, locally uniformly in~$x_1\in\R$}.$$
In particular, properties~\eqref{liminf}-\eqref{liminf2} of Theorem~\ref{thm:subgraph} do not hold.

(ii) Consider again a function $f$ of the bistable type~\eqref{bistable} and~\eqref{bistable2} (hence, Hypo\-thesis~\ref{hyp:minimalspeed} is fulfilled), assume that $N=2$, fix $\beta\in(0,\pi/2)$ and let $\phi$ and $\gamma_\lambda$ be as in~\eqref{defphi}. Then the solution $u$ of~\eqref{homo} with initial condition~\eqref{defu0bis} defined with, say, $\gamma=\gamma_{1/2}$ (hence,~\eqref{gamma<0} is not fulfilled) is such that $u(t,x_1,x_2)-\phi(x_1,x_2-(c^*/\sin\beta)t+a)\to0$ as~$t\to+\infty$ in $C^2(\R^2)$, for some $a\in\R$. As a consequence, $\partial_{x_1}X_\lambda(t,x_1)\to\gamma'_\lambda(x_1)$ as~$t\to+\infty$, locally uniformly in $x_1\in\R$, for every $\lambda\in(0,1)$. Since $\gamma'_\lambda(x_1)\to\pm1/\tan\beta\neq0$ as $x_1\to\pm\infty$, property~\eqref{limit2} of Conjecture~\ref{conj1} does not hold for all $x'=x_1\in\R$ (although of course it holds at $x_1=0$, and even $\partial_{x_1}X_\lambda(t,0)=0$ for all $t>0$, by even symmetry in~$x_1$).

(iii)-(iv) Assuming Hypothesis~\ref{hyp:minimalspeed}, consider first equation~\eqref{homo} in dimension $N=2$, and let $\gamma:\R\to\R$ be a~$C^1(\R)$ nonpositive function (hence,~\eqref{gamma<0} is satisfied) such that~$\gamma(x_1)=-ax_1<0$ for all $x_1\ge1$, for some~$a>0$. Let $u$ be the solution of~\eqref{homo} with initial condition $u_0$ given by~\eqref{defu0bis}. From standard parabolic estimates, the functions
$$(t,x)\mapsto u(t,x_1+r,x_2-ar)$$
converge, as $r\to+\infty$, in $C^{1;2}_{loc}((0,+\infty)\times\R^2)$ to the unique solution $u_\infty$ of~\eqref{homo} such that~$u_\infty(0,x_1,x_2)=0$ if $x_2>-ax_1-b$ and $u_\infty(0,x_1,x_2)=1$ otherwise. By uniqueness,~$u_\infty$ is then a function of the variables $t$ and $x_2+ax_1$ only, that is,
$$(1,-a)\cdot\nabla u_\infty(t,x_1,x_2)=0\quad \text{for all }(t,x_1,x_2)\in(0,+\infty)\times\R^2.$$
Furthermore, $u_\infty$ is decreasing with respect to the variable $x_2+ax_1$ in $(0,+\infty)\times\R^2$ (more precisely, $(a,1)\cdot\nabla u_\infty(t,x_1,x_2)<0$ in $(0,+\infty)\times\R^2$), and, for each $t\ge0$, $u_\infty(t,x)\to1$ as~$x_2+ax_1\to-\infty$ and $u_\infty(t,x_1,x_2)\to0$ as~$x_2+ax_1\to+\infty$. Therefore, for every $t>0$ and every $\lambda\in(0,1)$, the function $X_\lambda(t,\cdot)$ defined by~\eqref{defX} is such that $X_\lambda(t,x_1)+ax_1$ has a finite limit as $x_1\to+\infty$, and also
$$\partial_{x_1}X_\lambda(t,x_1)\to-a\ \hbox{ as $x_1\to+\infty$}.$$
Finally, for any~$\lambda\in(0,1)$, $\partial_{x_1}X_\lambda(t,x')$ cannot converge to $0$ as~$t\to+\infty$ uniformly with res\-pect to~$x'\in\R$. The conclusion is the same if one just assumes that $\gamma'(x_1)\to-a<0$ as~$x_1\to+\infty$, and it also holds in higher dimensions $N\ge2$ under similar assumptions on~$\gamma$. In particular, if~$\ell>0$ in condition~\eqref{hypconical}, then the conclusion~\eqref{limit3chi} of Theorem~\ref{thm:conical} is not uniform with respect to $x'\in\R^{N-1}$.
\end{proof}


\subsection{Proof of Theorem~\ref{thm:conical}}\label{sec53}

We start with the proof of~\eqref{limit3chi} firstly under condition~\eqref{hypconical} if $N=2$, secondly under condition~\eqref{hypconical} if $N\ge3$, thirdly under the condition $\gamma(x')/|x'|\to-\infty$ as $|x'|\to+\infty$ in any dimension $N\ge2$, fourthly if $\gamma$ is nonincreasing with respect to $|x'-x'_0|$ for large $|x'|$ and for some $x'_0\in\R^{N-1}$ in any dimension $N\ge2$, and  fifthly if $\gamma$ has small derivatives with respect to $|x'-x'_0|$ as $|x'-x'_0|\to+\infty$. The main idea is to argue by way of contradiction and to compare the solution with its reflection with respect to a suitable hyperplane at time $0$ and then at all positive times from the maximum principle. We will eventually get a contradiction using the Hopf lemma at a suitable point of this hyperplane. We finally derive~\eqref{limit3u} in any dimension $N\ge2$, from~\eqref{limit3chi}. Throughout the proof, one assumes Hypothesis~\ref{hyp:invasion}.

\medskip
\noindent\underline{Step 1: property~\eqref{limit3chi} in dimension $N=2$ under condition~\eqref{hypconical}.} Assume by way of contradiction that~\eqref{limit3chi} does not hold. Then there exist a sequence $(\lambda_n)_{n\in\N}$ in~$(0,1)$, a sequence~$(t_n)_{n\in\N}$ of positive real numbers diverging to $+\infty$, and a bounded sequence~$(x'_n)_{n\in\N}$ in $\R$, such that $\sup_{n\in\N}\lambda_n<1$ and~$\inf_{n\in\N}|\partial_{x'}X_{\lambda_n}(t_n,x'_n)|>0$. Up to extraction of a subsequence and changing the variable $x'$ into $-x'$ if need be, it is not restrictive to assume that
$$\sup_{n\in\N}\,\partial_{x'}X_{\lambda_n}(t_n,x'_n)\le-2\epsilon$$
for some $\epsilon>0$. In the sequel, we denote $y$ the variable $x_2$ and set $y_n:=X_{\lambda_n}(t_n,x'_n)$ and~$\sigma_n:=\partial_{x'}X_{\lambda_n}(t_n,x'_n)<0$. Since $u(t_n,x'_n,y_n)=\lambda_n$ is away from $1$ and since~$t_n\to+\infty$ as $n\to+\infty$, it follows from Hypothesis~\ref{hyp:invasion} and from the boundedness of~$(x'_n)_{n\in\N}$ that~$y_n\to+\infty$ as $n\to+\infty$. Notice that
$$(1,\sigma_n)\cdot\nabla u(t_n,x'_n,y_n)=(1,\partial_{x'}X_{\lambda_n}(t_n,x'_n))\cdot\nabla u(t_n,x'_n,y_n)=0$$
by \eqref{DchiDu}, and denote
$$(\alpha_n,\beta_n):=\frac{\nabla u(t_n,x'_n,y_n)}{|\nabla u(t_n,x'_n,y_n)|}=\frac{(\sigma_n,-1)}{\sqrt{1+\sigma_n^2}}.$$
($\beta_n$ is negative since $\partial_{y}u<0$ in $(0,+\infty)\times\R^2$, and then $\alpha_n$ is negative too since so is $\sigma_n$). One then has $\sigma_n=-\alpha_n/\beta_n$ and
\Fi{e<}
0<\epsilon\le-\frac{1}{2}\,\sup_{n\in\N}\sigma_n=\frac{1}{2}\,\inf_{n\in\N}\frac{\alpha_n}{\beta_n}.
\Ff
		
We use now a reflection argument inspired by Jones~\cite{J}. For $n\in\N$, consider the line~$L_n$ passing through the point $(x'_n,y_n)$ and directed as $\nabla u(t_n,x'_n,y_n)$. It is the graph of the function
$$x'\mapsto\rho_n(x'):=\frac{\beta_n}{\alpha_n}(x'-x_n')+y_n=-\frac{1}{\sigma_n}(x'-x_n')+y_n.$$
Then, consider the half-plane given by its open subgraph: 
$$\O_n:=\big\{(x',y)\in\R^2:y<\rho_n(x')\big\}.$$
The vector $(1,\sigma_n)$ is then an inward normal to $\O_n$. Finally, let~$\mc{R}_n$ denote the affine orthogonal reflection with respect to $L_n$, that is,
$$\mc{R}_n(x',y)=(x',y)-2\Big[(x'-x'_n,y-y_n)\.(-\beta_n,\alpha_n)\Big](-\beta_n,\alpha_n).$$
We then define the function $v_n$ in $[0,+\infty)\times\overline{\O_n}$ by
$$v_n(t,x',y):=u(t,\mc{R}_n(x',y)).$$
We claim that, for $n$ large enough,
$$v_n(0,\.,\.)\leq u_0\ \hbox{ in $\overline{\O_n}$}.$$
To prove this, we need to check that if $(x',y)\in\overline{\O_n}$ is such that $\mc{R}_n(x',y)\in\supp u_0$, then necessarily $(x',y)\in\supp u_0$, which is equivalent to show that 
\be\label{Rn}
\mc{R}_n(\supp u_0\!\setminus\!\O_n)\subset\supp u_0.
\ee

Since $(x'_n)_{n\in\N}$ is bounded and $(y_n)_{n\in\N}$ diverges to $+\infty$, and since $\gamma$ is locally bounded, we can assume without loss of generality that, for all $n\in\N$, $(x_n',y_n)\notin\supp u_0$. We~set
$$\xi_n:=\sup\big\{x'<x'_n:\gamma(x')\geq\rho_n(x')\big\}\ \ \hbox{ and }\ \ \zeta_n:=\inf\big\{x'>x'_n:\gamma(x')\geq\rho_n(x')\big\}.$$
If the above sets are empty we define $\xi_n=-\infty$, and $\zeta_n=+\infty$, respectively. Observe that the sequence of functions $\seq{\rho}$ tends locally uniformly to $+\infty$, because $y_n\to+\infty$ and the sequences $(x'_n)_{n\in\N}$ and $(\beta_n/\alpha_n)_{n\in\N}=(-1/\sigma_n)_{n\in\N}$ are bounded. Furthermore, $\gamma$ is locally bounded, and at least continuous outside a compact interval. It follows that 
\Fi{xizeta}
\xi_n\to-\infty\ \hbox{ and }\zeta_n\to+\infty\as n\to+\infty.
\Ff
We have that $(\supp u_0\!\setminus\!\O_n)\cap\big((\xi_n,\zeta_n)\times\R\big)=\emptyset$ for all $n$ large enough, hence for all $n$ without loss of generality. By hypothesis~\eqref{hypconical}, there exists $k>\sup_{n\in\N}|x'_n|+1$ such that~$\gamma$ is of class~$C^1$ in $(-\infty,-k]\cup[k,+\infty)$, and
\Fi{gamma'}
\gamma'\geq\ell-\e\hbox{ in }(-\infty,-k]\ \hbox{ and }\ \gamma'\geq-\ell-\e\hbox{ in }[k,+\infty).
\Ff
Without loss of generality, we can assume that
$$\xi_n<-k<k<\zeta_n$$
for all $n$. We finally define 
$$\left\{\baa{l}
K_n^1:=\mc{R}_n(\supp u_0\!\setminus\!\O_n)\cap\big((-\infty,-k)\times\R\big),\vspace{3pt}\\
K_n^2:=\mc{R}_n(\supp u_0\!\setminus\!\O_n)\cap\big([-k,k]\times\R\big),\vspace{3pt}\\
K_n^3:=\mc{R}_n(\supp u_0\!\setminus\!\O_n)\cap\big((k,+\infty)\times\R\big),\eaa\right.$$
These sets are depicted in Figure~\ref{fig:triangular}. We show separately that they are contained in $\supp u_0$, for all $n$ large enough. That will provide the desired property~\eqref{Rn} for $n$ large.	
		
\begin{figure}[H]
\begin{center}
\includegraphics[width=\textwidth]{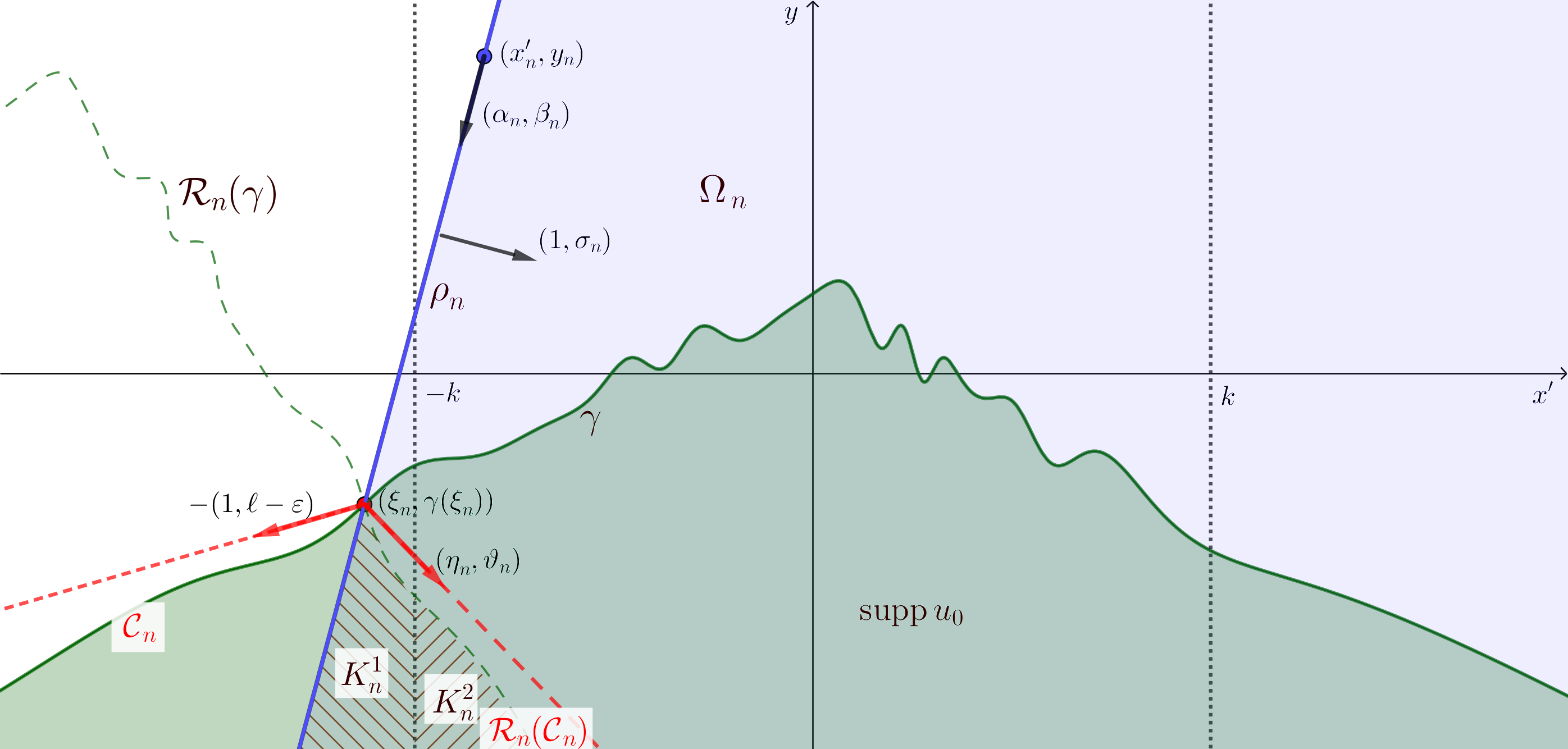}
\caption{The reflection argument, with the sets $K_n^1$ and $K_n^2$.}
\label{fig:triangular}
\end{center}
\vspace{-7pt}
\end{figure}
				
\medskip
\noindent{\em The inclusion $K_n^1\subset\supp u_0$}. Consider a point in $K_n^1$. It can be written as $(x',y)+\tau(1,\sigma_n)$ with $(x',y)\in\supp u_0\!\setminus\!\O_n$ and $0\leq\tau<-x'-k$. Notice that $x'<-\tau-k\le-k$, hence~$y\le\gamma(x')$. We write 
$$\gamma(x'+\tau)=\gamma(x')+\int_0^\tau\gamma'(x'+s)ds.$$
Conditions \eqref{e<} and \eqref{gamma'} yield $\gamma'(x'+s)\geq\ell-\e>\sigma_n$ for $x'+s\leq-k$. We eventually deduce that $\gamma(x'+\tau)\geq\gamma(x')+\sigma_n\tau\geq y+\sigma_n\tau$. Since $x'+\tau<-k$, this implies that $(x',y)+\tau(1,\sigma_n)\in\supp u_0$.
		
\medskip
\noindent{\em The inclusion $K_n^2\subset\supp u_0$ for $n$ sufficiently large}. In this case we consider a point of the type $(x',y)+\tau(1,\sigma_n)$ with $(x',y)\in\supp u_0\!\setminus\!\O_n$ and $\tau\geq0$ such that $-k\leq x'+\tau\leq k$. Since $x'\le k-\tau\le k$ and~$(\supp u_0\setminus\O_n)\cap\big((\xi_n,\zeta_n)\times\R\big)=\emptyset$ with $\xi_n<-k<k<\zeta_n$, we get that $x'\le\xi_n<-k$. Moreover, by hypothesis, there exists $M>0$ (independent of~$n$,~$x'$,~$y$ and~$\tau$) such that $\gamma(s)\leq M+\e|s|$ for all $s\in\R$. As a consequence, using~\eqref{e<}, we infer~that
$$y+\tau\sigma_n\leq M-(\e +\sigma_n)x'-\sigma_n k\leq M-\sigma_n\left(\frac{x'}2+k\right)\leq M-\sigma_n\left(\frac{\xi_n}2+k\right).$$
The latter term tends to $-\infty$ as $n\to+\infty$ by~\eqref{e<} and~\eqref{xizeta}. It follows that for $n$ large enough (independent of $x'$, $y$, $\tau$) there holds that
$$y+\tau\sigma_n<\inf_{[-k,k]}\gamma-1,$$
whence $(x',y)+\tau(1,\sigma_n)\in\supp u_0$. Therefore, $K^2_n\subset\supp u_0$ for all $n$ large enough, and even $(\supp u_0\setminus K^2_n)\cap\big([-k,k]\times\R\big)$ has non-empty interior.

\medskip
\noindent{\em The inclusion $K_n^3\subset\supp u_0$ for $n$ sufficiently large}. We recall that $\xi_n<-k<k<\zeta_n$ and $(\supp u_0\!\setminus\!\O_n)\cap\big((\xi_n,\zeta_n)\times\R\big)=\emptyset$ for all $n$. We can then divide this case in the following two subcases.
		
{\em Subcase 1:} the points in $K^3_n$ of the type $\mc{R}_n(x',y)$ with $(x',y)\in\supp u_0\!\setminus\!\O_n$ and~$x'\geq\zeta_n\ (>k)$. If $\ell>0$ then $\gamma$ is bounded from above and such points do not exist for $n$ sufficiently large, since they would satisfy $\rho_n(x')\leq y\leq\gamma(x')$, whereas the sequence of functions~$\seq{\rho}$ tends to~$+\infty$ uniformly in any half-line $[A,+\infty)$. In the case~$\ell=0$, we write~$\mc{R}_n(x',y)=(x',y)+\tau(1,\sigma_n)$ for some $\tau\geq0$. Then, because $x'\geq\zeta_n>k$, we can argue as in the case of $K_n^1$ and, by virtue of \eqref{e<} and~\eqref{gamma'},~derive
$$\gamma(x'+\tau)\geq\gamma(x')-\e\tau\geq\gamma(x')+\sigma_n\tau\geq y+\sigma_n\tau,$$
that is, $(x',y)+\tau(1,\sigma_n)\subset\supp u_0$.
		
{\em Subcase 2:} the points in $K^3_n$ of the type $\mc{R}_n(x',y)$ with $(x',y)\in\supp u_0\!\setminus\!\O_n$ and~$x'\leq\xi_n\ (<-k)$. Of course, these points exist only if $\xi_n>-\infty$. By definition of~$\xi_n$, we see that $\rho_n'(\xi_n)\geq \gamma'(\xi_n)$. Then, it follows from~\eqref{gamma'}~that
$$\frac{\beta_n}{\alpha_n}=\rho_n'(\xi_n)\geq \gamma'(\xi_n)\geq\inf_{(-\infty,\xi_n]}\gamma'\geq\ell-\e,$$
whence $(x',y)$ is contained in the cone
$$\mc{C}_n:=\big\{(\xi_n,\gamma(\xi_n))+s(-1,-(\ell-\e))+t(\alpha_n,\beta_n):s,t\geq0\big\},$$
see Figure~\ref{fig:triangular}. The point $\mc{R}_n(x',y)$ is contained in the reflected cone
$$\mc{R}_n(\mc{C}_n)=\big\{(\xi_n,\gamma(\xi_n))+s(\eta_n,\vt_n)+t(\alpha_n,\beta_n):s,t\geq0\big\},$$
where 
$$(\eta_n,\vt_n)= \t{\mc{R}}_n(-1,-(\ell-\e))=(-1,-(\ell-\e))-2\Big[(-1,-(\ell-\e))\.(-\beta_n,\alpha_n)\Big](-\beta_n,\alpha_n),$$
and $\t{\mc{R}}_n$ denotes the linear orthogonal reflection with respect to the one-dimensional subspace $\R\,(\alpha_n,\beta_n)$. We see that
$$\eta_n=-1+2\beta_n\Big[(-1,-(\ell-\e))\.(-\beta_n,\alpha_n)\Big]=1-2\alpha_n^2-2(\ell-\e)\alpha_n\beta_n\leq 1-2\alpha_n(\alpha_n-\e\beta_n),$$
which is not larger than $1$ by~\eqref{e<} and the negativity of $\alpha_n$ and $\beta_n$. If $\eta_n\leq0$ then~$\mc{R}_n(\mc{C}_n)\subset(-\infty,\xi_n]\times\R\subset(-\infty,-k)\times\R\subset(-\infty,k]\times\R$, and therefore in this case $K^3_n=\emptyset$ and we are done. Suppose that $\eta_n>0$, i.e., that
$$(-1,-(\ell-\e))\.(-\beta_n,\alpha_n)<\frac1{2\beta_n}.$$
We deduce that
$$\vt_n=-\ell+\e-2\alpha_n\Big[(-1,-(\ell-\e))\.(-\beta_n,\alpha_n)\Big]<-\ell+\e-\frac{\alpha_n}{\beta_n}\le-\ell-\e,$$
always by~\eqref{e<}. This means that $\vt_n/\eta_n\le-\ell-\e$, whence	
$$\mc{R}_n(\mc{C}_n)\subset\big\{(\xi_n,\gamma(\xi_n))+s(1,-\ell-\e)+t(\alpha_n,\beta_n):s,t\geq0\big\}.$$
It eventually follows from~\eqref{gamma'} and from the fact that $\xi_n\to-\infty$ and $\gamma(\xi_n)/\xi_n\to-\ell\le0$ as $n\to+\infty$, that $\mc{R}_n(\mc{C}_n)\cap\big((k,+\infty)\times\R\big)\subset\supp u_0$ for $n$ large enough, that is,~$\mc{R}_n(x',y)\in\supp u_0$ in this last case too.

\medskip
\noindent{\em Conclusion}. We have shown that $\mc{R}_n(\supp u_0\!\setminus\!\O_n)\subset\supp u_0$, hence $v_n(0,\.,\.)\leq u_0$ in $\ol\O_n$ for $n$ sufficiently large, and actually that $u_0-v_n(0,\.,\.)=1$ in a non-trivial ball included in~$\O_n$, because $(\supp u_0\setminus K^2_n)\cap\big([-k,k]\times\R\big)$ has non-empty interior. The function $v_n$ satisfies the same equation \eqref{homo} as $u$, and it coincides with $u$ on $\partial\O_n$. It then follows from the parabolic strong maximum principle that $v_n<u$ in $(0,+\infty)\times\O_n$, and thus, by the Hopf lemma, that
$$\nabla u(t_n,x'_n,y_n)\.(1,\sigma_n)>\nabla v_n(t_n,x'_n,y_n)\.(1,\sigma_n)=\t{\mc{R}}_n(\nabla u(t_n,x'_n,y_n))\.(1,\sigma_n).$$
We have reached a contradiction because $\nabla u(t_n,x'_n,y_n)=\t{\mc{R}}_n(\nabla u(t_n,x'_n,y_n))$ (the vector $\nabla u(t_n,x'_n,y_n)$ is indeed parallel to $(\alpha_n,\beta_n)$). As a consequence,~\eqref{limit3chi} has been proved under condition~\eqref{hypconical} in dimension $N=2$.

\medskip
\noindent{\underline{Step 2: extension to arbitrary dimension $N\ge2$.} Assume  any of the conditions~(i)-(iv) of Theorem~\ref{thm:conical} and assume by way of contradiction that~\eqref{limit3chi} does not hold. Then there exist a sequence~$(\lambda_n)_{n\in\N}$ in $(0,1)$, a sequence $(t_n)_{n\in\N}$ of positive real numbers diverging to~$+\infty$, a bounded sequence $(x'_n)_{n\in\N}$ in~$\R^{N-1}$, and a sequence $(e'_n)_{n\in\N}$ in~$\mathbb{S}^{N-2}$ (if~$N=2$, this means that~$e'_n\in\{-1,1\}$), such that $\sup_{n\in\N}\lambda_n<1$ and
\be\label{supsigman}
\sup_{n\in\N}\,\underbrace{\nabla_{\!x'}X_{\lambda_n}(t_n,x'_n)\cdot e'_n}_{=:\sigma_n}<0.
\ee
Call
$$y_n=X_{\lambda_n}(t_n,x'_n).$$
As in Step~1, one has $y_n\to+\infty$ as $n\to+\infty$. Notice that, for each $n\in\N$,
\be\label{nablaue'n}
(e'_n,\sigma_n)\cdot\nabla u(t_n,x'_n,y_n)=0
\ee
by~\eqref{DchiDu}, and call $H_n$ the affine hyperplane passing through the point $(x'_n,y_n)$ and orthogonal to $(e'_n,\sigma_n)$. This hyperplane is the graph of the function
\be\label{defrhon}
x'\mapsto\rho_n(x'):=-\frac{1}{\sigma_n}(x'-x_n')\cdot e'_n+y_n.
\ee
Then, consider the half-space given by its open subgraph: 
\be\label{defOmegan}
\O_n:=\big\{(x',x_N)\in\R^N:x_N<\rho_n(x')\big\}.
\ee
The vector $(e'_n,\sigma_n)$ is then an inward normal to $\O_n$. Finally, let~$\mc{R}_n$ denote the affine orthogonal reflection with respect to $H_n$, that is,
\be\label{defRn}
\mc{R}_n(x',x_N)=(x',x_N)-2\Big[(x'-x'_n,x_N-y_n)\.(e'_n,\sigma_n)\Big]\,\frac{(e'_n,\sigma_n)}{1+\sigma_n^2}.
\ee
We then define the function $v_n$ in $[0,+\infty)\times\overline{\O_n}$ by
\[
v_n(t,x',x_N):=u(t,\mc{R}_n(x',x_N)),
\]
and we claim that, for $n$ large enough, $v_n(0,\.,\.)\leq u_0$ in $\overline{\O_n}$ and that $u_0-v_n(0,\cdot,\cdot)=1$ in a non-trivial ball. As in Step~1, this then leads to a contradiction and complete the~proof. Namely, using the parabolic strong maximum principle one infers that $v_n<u$ in $(0,+\infty)\times\O_n$, and from the Hopf lemma that, in particular, 
\[
\nabla u(t_n,x'_n,y_n)\.(e'_n,\sigma_n)>\nabla v_n(t_n,x'_n,y_n)\.(e'_n,\sigma_n)=\t{\mc{R}}_n(\nabla u(t_n,x'_n,y_n))\.(e'_n,\sigma_n),
\]
where $\t{\mc{R}}_n$ denotes the linear orthogonal reflection with respect to the linear hyperplane orthogonal to the vector $(e'_n,\sigma_n)$. But this is impossible because the vector $\nabla u(t_n,x'_n,y_n)$ is orthogonal to $(e'_n,\sigma_n)$ by~\eqref{nablaue'n}, hence $\nabla u(t_n,x'_n,y_n)=\t{\mc{R}}_n(\nabla u(t_n,x'_n,y_n))$. 

So, to prove~\eqref{limit3chi} we just need to show that there exists $n\in\N$ such that $v_n(0,\.,\.)\leq u_0$ in $\overline{\O_n}$ and moreover $u_0-v_n(0,\cdot,\cdot)=1$ in a non-trivial ball. These conditions translate into the following ones on $\supp u_0$:
\be\label{Rnbis}
\mc{R}_n(\supp u_0\!\setminus\!\O_n)\subset\supp u_0,
\ee
and moreover $\supp u_0\setminus\mc{R}_n(\supp u_0\!\setminus\!\O_n)$ contains a non-trivial ball. Let us show that the latter property holds for $n$ sufficiently large. Observe firstly that for any non-empty compact set $K\subset\R^N$, one~has
$$\min_{(x',x_N)\in K}\,\Big(x_N-2\big[(x'-x'_n,x_N-y_n)\cdot(e'_n,\sigma_n)\big]\,\frac{\sigma_n}{1+\sigma_n^2}\Big)\to+\infty\ \hbox{as $n\to+\infty$},$$
and
\[\liminf_{n\to+\infty}\ \min_{(x',x_N)\in K}\left(\frac{\displaystyle x_N-2\big[(x'-x'_n,x_N-y_n)\cdot(e'_n,\sigma_n)\big]\,\frac{\sigma_n}{1+\sigma_n^2}}{\displaystyle\Big|x'-2\big[(x'-x'_n,x_N-y_n)\cdot(e'_n,\sigma_n)\big]\,\frac{e'_n}{1+\sigma_n^2}\Big|}\right)\ge\liminf_{n\to+\infty}|\sigma_n|>0,\]
since $y_n\to+\infty$, $\sup_{n\in\N}\sigma_n<0$ and since the sequence $(x'_n)_{n\in\N}$ is bounded and $(e'_n)_{n\in\N}$ is unitary. But $\gamma$ in~\eqref{defu0bis} is always assumed to be locally bounded, and it is easy to see that
\be\label{limsupgamma}
\limsup_{|x'|\to+\infty}\frac{\gamma(x')}{|x'|}\le0
\ee
in all cases~(i)-(iv) of Theorem~\ref{thm:conical}. Therefore, owing to the definition~\eqref{defRn} of $\mc{R}_n$, one gets that $\mc{R}_n(K)\cap\supp u_0=\emptyset$ for all $n$ large enough, that is,
\[
K\cap\mc{R}_n(\supp u_0)=\emptyset,\quad \text{for all $n$ large enough}.
\]
In particular, $\supp u_0\setminus\mc{R}_n(\supp u_0)$ contains a non-trivial ball for any $n$ large enough.

As a consequence, in order to prove~\eqref{limit3chi} we only need to show that~\eqref{Rnbis} holds for $n$ sufficiently large. Assume now by way of contradiction that this is not the case. Then, up to extraction of a subsequence, there is a sequence of points $z_n=(z'_n,\varpi_n)$ in $\R^N$ such that
$$z_n\in\supp u_0\!\setminus\!\O_n\hbox{ and }\mc{R}_n(z_n)\notin\supp u_0,\hbox{ for all~$n\in\N$}.$$
Denote
\be\label{deltan}
\delta_n:=\frac{(z'_n-x'_n,\varpi_n-y_n)\cdot(e'_n,\sigma_n)}{1+\sigma_n^2},
\ee
that is,
\be\label{Rnzn}
\mc{R}_n(z_n)=(z'_n-2\delta_ne'_n,\varpi_n-2\delta_n\sigma_n).
\ee
Since $z_n\not\in\O_n$, one has $\delta_n\le0$, and even
$$\delta_n<0$$
(since otherwise $z_n$ would lie on $H_n$ and~$\mc{R}_n(z_n)$, which does not belong to $\supp u_0$, would be equal to $z_n\in\supp u_0$). { Since $y_n\to+\infty$ as $n\to+\infty$ and~$\sup_{n\in\N}\sigma_n<0$, together with the boundedness of the sequences $(x'_n)_{n\in\N}$ and $(e'_n)_{n\in\N}$, one infers that $\rho_n(x')\to+\infty$ as~$n\to+\infty$ locally uniformly in $x'\in\R^{N-1}$. Since $z_n=(z'_n,\varpi_n)\in\supp u_0\!\setminus\!\O_n$, it then follows from the local boundedness of $\gamma$ and the definition~\eqref{defOmegan} of $\Omega_n$, that
\be\label{|z'n|}
|z'_n|\to+\infty\ \hbox{ as }n\to+\infty,
\ee
and, together with~\eqref{limsupgamma}, that
\be\label{limsupvarpin}
\limsup_{n\to+\infty}\frac{\varpi_n}{|z'_n|}\le0.
\ee

We also claim that
\be\label{claimz'n}
|z'_n-2\delta_ne'_n|\to+\infty\ \hbox{ as }n\to+\infty.
\ee
Indeed, otherwise, up to extraction of a subsequence, the sequence $(|z'_n-2\delta_ne'_n|)_{n\in\N}$ would be bounded, hence $\delta_n\to-\infty$ and $-2\delta_n\sim|z'_n|$ as $n\to+\infty$, since $|z'_n|\to+\infty$, $\delta_n<0$ and~$|e'_n|=1$. Furthermore, since the points $\mc{R}(z_n)$ given in~\eqref{Rnzn} do not belong to $\supp u_0$ and since $\gamma$ is locally bounded, the sequence $(\varpi_n-2\delta_n\sigma_n)_{n\in\N}$ would then be bounded from below, that is, there would exist $A\in\R$ such that $\varpi_n\ge2\delta_n\sigma_n+A$ for all~$n\in\N$. Finally, together with~\eqref{supsigman} and~\eqref{|z'n|}, one would have
$$\liminf_{n\to+\infty}\frac{\varpi_n}{|z'_n|}\ge\liminf_{n\to+\infty}\frac{2\delta_n\sigma_n}{|z'_n|}=-\limsup_{n\to+\infty}\sigma_n>0,$$
a contradiction with~\eqref{limsupvarpin}. As a consequence,~\eqref{claimz'n} has been proved.

Furthermore, since $\mc{R}_n(z_n)=(z'_n-2\delta_ne'_n,\varpi_n-2\delta_n\sigma_n)\not\in\supp u_0$ and $\gamma$ is at least continuous outside a compact set in all cases~(i)-(iv) of Theorem~\ref{thm:conical}, one gets from~\eqref{claimz'n} that
\be\label{varpin0}
\varpi_n-2\delta_n\sigma_n>\gamma(z'_n-2\delta_ne'_n)
\ee
for all $n$ large enough, and then for all $n$ without loss of generality. Moreover, since $z_n=(z'_n,\varpi_n)\in\supp u_0\!\setminus\!\Omega_n$, it follows from~\eqref{defOmegan} and~\eqref{|z'n|} that $\rho_n(z'_n)\le\varpi_n\le\gamma(z'_n)$ for all $n$ large enough, and then for all $n$ without loss of generality. Therefore,
\be\label{gammadeltan}
\gamma(z'_n-2\delta_ne'_n)-\gamma(z'_n)<-2\delta_n\sigma_n.
\ee

On the other hand, since the function~$\gamma$ is always locally bounded, the assumption~\eqref{hypconical} and the nonnegativity of $\ell$ then imply that $\gamma$ is here globally bounded from above. With the above notations, define, for each $n\in\N$,
$$\xi_n:=\sup\big\{x'\cdot e'_n:\gamma(x')\ge\rho_n(x')\big\}$$
(with the value $-\infty$ if the above set is empty). Since $y_n\to+\infty$ as $n\to+\infty$ and~$\sup_{n\in\N}\sigma_n<0$, together with the boundedness of the sequences $(x'_n)_{n\in\N}$ and $(e'_n)_{n\in\N}$, one infers that $\inf_{x'\cdot e'_n\ge A}\rho_n(x')\to+\infty$ for every $A\in\R$. Together with the boundedness from above of $\gamma$ and the fact that it is at least continuous (and even $C^1$) in~$\R^{N-1}\!\setminus\!B'_R$ for some $R>0$, one gets that $\xi_n\to-\infty$ as $n\to+\infty$, and then that
$$\xi_n\le-R\ \hbox{ and }\ \supp u_0\!\setminus\!\O_n\subset\big\{(x',x_N)\in\R^N:x'\cdot e'_n\le\xi_n\big\}$$
for all $n$, without loss of generality. In particular, since $z_n=(z'_n,\varpi_n)\in\supp u_0\!\setminus\!\O_n$, one has $z'_n\cdot e'_n\le\xi_n\le-R$, and
\be\label{z'ne'n}
z'_n\cdot e'_n\to-\infty\ \hbox{ as $n\to+\infty$}.
\ee

Owing to the definition~\eqref{deltan} of $\delta_n$, one also has
$$|z'_n|^2-|z'_n\!-\!2\delta_ne'_n|^2=-4\delta_n(\delta_n\!-\!z'_n\cdot e'_n)=\frac{-4\delta_n}{1\!+\!\sigma_n^2}\,\Big(\!-\sigma_n^2(z'_n\cdot e'_n)\!-\!x'_n\cdot e'_n\!+\!\sigma_n(\varpi_n\!-\!y_n)\Big).$$
Since $\sup_{n\in\N}\sigma_n<0$, since $z'_n\cdot e'_n\to-\infty$, since the sequences $(x'_n)_{n\in\N}$ and $(e'_n)_{n\in\N}$ are bounded, since the sequence $(\varpi_n)_{n\in\N}$ is bounded from above (because $\gamma$ is globally bounded from above and $z_n=(z'_n,\varpi_n)\in\supp u_0$), and since $y_n\to+\infty$, one infers that
$$-\sigma_n^2(z'_n\cdot e'_n)-x'_n\cdot e'_n+\sigma_n(\varpi_n-y_n)\to+\infty\ \hbox{ as $n\to+\infty$}.$$
Together with the negativity of $\delta_n$, one gets that $|z'_n|^2-|z'_n-2\delta_ne'_n|^2>0$ for all $n$ large enough, while $\lim_{n\to+\infty}|z'_n-2\delta_ne'_n|=+\infty$ by~\eqref{claimz'n}, hence
\be\label{z'n}
|z'_n|>|z'_n-2\delta_ne'_n|\ge R
\ee
for all $n$, without loss of generality.

Let us now complete the argument. Since $\gamma$ is here assumed to be of class $C^1$ outside~$B'_R$ and since it satisfies~\eqref{hypconical} (use here the condition on the radial gradients at large $|x'|$ and the positivity of $\eta$), there is $M>0$ such that $\big|\gamma(x')+\ell|x'|\big|\le M$ for all $|x'|\ge R$. Together with~\eqref{gammadeltan}-\eqref{z'n} and the nonnegativity of $\ell$, it follows that
$$-2\delta_n\sigma_n>-\ell|z'_n-2\delta_ne'_n|-M+\ell|z'_n|-M\ge-2M$$
for all $n$. But $\delta_n<0$ and $\sup_{n\in\N}\sigma_n<0$. Thus, the sequence $(\delta_n)_{n\in\N}$ is bounded. Together with~\eqref{|z'n|}, that implies that $|z'_n-2s\delta_ne'_n|\ge R$ for all $s\in[0,1]$ and for all $n\in\N$, without loss of generality. Dividing~\eqref{gammadeltan} by $-2\delta_n>0$ and using the $C^1$ smoothness of $\gamma$ outside~$B'_R$, one then gets the existence of a sequence $(\vartheta_n)_{n\in\N}$ in $(0,1)$ such that
\be\label{varthetan}
\nabla\gamma(z'_n-2\vartheta_n\delta_ne'_n)\cdot e'_n<\sigma_n
\ee
for all $n\in\N$. Since the sequences $(\vartheta_n)_{n\in\N}$, $(\delta_n)_{n\in\N}$ and $(e'_n)_{n\in\N}$ are bounded, one then infers from~\eqref{hypconical} and~\eqref{|z'n|} that
$$\nabla\gamma(z'_n-2\vartheta_n\delta_ne'_n)\cdot e'_n=-\ell\,\frac{z'_n\cdot e'_n}{|z'_n-2\vartheta_n\delta_ne'_n|}+o(1)\ \hbox{ as }n\to+\infty,$$
hence $\liminf_{n\to+\infty}\nabla\gamma(z'_n-2\vartheta_n\delta_ne'_n)\cdot e'_n\ge0$ from~\eqref{z'ne'n} and the nonnegativity of $\ell$. But this last formula contradicts~\eqref{supsigman} and~\eqref{varthetan}.

One has then reached a contradiction, implying that the desired property~\eqref{Rnbis} holds for all $n$ large enough. As explained above, this yields in turn property~\eqref{limit3chi}.

\medskip
\noindent{\underline{Step 3: property~\eqref{limit3chi} for any $N\ge2$ if $\gamma(x')/|x'|\to-\infty$ as $|x'|\to+\infty$.}} In this case, pro\-perty~\eqref{Rnbis} can be directly checked without arguing by contradiction. Indeed, since $\sup_{n\in\N}\sigma_n<0$ and $y_n\to+\infty$, it then easily follows that $\supp u_0\subset\Omega_n$ for all $n$ large enough, hence~\eqref{Rnbis} is automatically satisfied simply because $\mc{R}_n(\supp u_0\setminus\O_n)=\emptyset$.

\medskip
\noindent{\underline{Step 4: property~\eqref{limit3chi} for any $N\!\ge\!2$ if $\!\gamma$ is nonincreasing in $\!|x'\!-\!x'_0|$.}} More precisely, let us assume here that there are $x'_0\in\R^{N-1}$ and a continuous nonincreasing function $\Gamma:\R^+\to\R$ such that $\gamma(x')=\Gamma(|x'-x'_0|)$ for all $x'$ outside a compact set. Since the desired conclusion~\eqref{limit3chi} is invariant by translation with respect to the first $N\!-\!1$ variables of $\R^N$, one can assume without loss of generality that $x'_0=0$ and that $\gamma$ is continuous and nonincreasing with respect to $|x'|$ for $|x'|$ large enough. Since $\gamma$ is locally bounded, it is then globally bounded from above. By using the same notations and repeating the same arguments as in Steps~2 above until~\eqref{z'n} (as far as $\gamma$ is concerned, the arguments until~\eqref{z'n} only use the boundedness of $\gamma$ from above), one gets~\eqref{gammadeltan}-\eqref{z'n}. But both $|z'_n|$ and $|z'_n-2\delta_ne'_n|$ converge to $+\infty$ as $n\to+\infty$ by~\eqref{|z'n|} and~\eqref{claimz'n}, and~$\gamma$ is nonincreasing with respect to~$|x'|$ outside a compact set. Therefore,~\eqref{z'n} implies that~$\gamma(z'_n-2\delta_ne'_n)-\gamma(z'_n)\ge0$ for all $n$ large enough, contradicting~\eqref{gammadeltan} since both $\delta_n$ and $\sigma_n$ are negative. This means that~\eqref{Rnbis} necessarily holds.

\medskip
\noindent{\underline{Step 5: property~\eqref{limit3chi} for any $N\!\ge\!2$ if $\!\gamma=\Gamma(|\cdot-x'_0|)$ with $\Gamma'(+\infty)=0$.}} More precisely, let us assume here that there are $x'_0\in\R^{N-1}$ and a $C^1$ function $\Gamma:\R^+\to\R$ such that~$\Gamma'(r)\to0$ as $r\to+\infty$ and $\gamma(x')=\Gamma(|x'-x'_0|)$ for all $x'$ outside a compact set. As in Step~4, one can assume without loss of generality that $x'_0=0$. With the same notations as in Step~2, both $|z'_n|$ and $|z'_n-2\delta_ne'_n|$ converge to $+\infty$ as $n\to+\infty$ by~\eqref{|z'n|} and~\eqref{claimz'n}. Therefore, for every $\varepsilon>0$, there holds
$$|\gamma(z'_n-2\delta_ne'_n)-\gamma(z'_n)|=\big|\Gamma(|z'_n-2\delta_ne'_n|)-\Gamma(|z'_n|)\big|\le\varepsilon\,\big||z'_n-2\delta_ne'_n|-|z'_n|\big|\le2\varepsilon|\delta_n|$$
for all $n$ large enough (remember that $|e'_n|=1$). Since $\varepsilon>0$ can be arbitrarily small and since $\delta_n<0$ 
for all $n$ and $\sup_{n\in\N}\sigma_n<0$ by~\eqref{supsigman}, the above formula contradicts~\eqref{gammadeltan} and therefore proves~\eqref{Rnbis}. 

\medskip
\noindent{\underline{Step 6: proof of property~\eqref{limit3u}.}} We assume in this last step any of the assumptions~(i)-(iv) of Theorem~\ref{thm:conical}. Consider any bounded sequence $(x'_n)_{n\in\N}$ of $\R^{N-1}$, any sequence $(t_n)_{n\in\N}$ of positive real numbers diverging to $+\infty$, and any sequence $(y_n)_{n\in\N}$ in~$\R$. Two cases may occur, up to extraction of a subsequence.

On the one hand, if $\limsup_{n\to+\infty}u(t_n,x'_n,y_n)<1$, then $|\nabla_{\!x'}X_{u(t_n,x'_n,y_n)}(t_n,x'_n)|\to0$ as~$n\to+\infty$ from~\eqref{limit3chi}, hence
$$|\nabla_{\!x'}u(t_n,x'_n,y_n)|=|\partial_{x_N}u(t_n,x'_n,y_n)|\,|\nabla_{\!x'}X_{u(t_n,x'_n,y_n)}(t_n,x'_n)|\to0\ \hbox{ as $n\to+\infty$}$$
from the boundedness of $\partial_{x_N}u$ in $[1,+\infty)\times\R^N$.

On the other hand, if $u(t_n,x'_n,y_n)\to1$ as $n\to+\infty$, then, up to extraction of a subsequence, the functions $u_n:(t,x',x_N)\mapsto u(t+t_n,x'+x'_n,x_N+y_n)$ converge in $C^{1;2}_{loc}(\R\times\R^N)$ to a classical solution $u_\infty$ of $\partial_tu_\infty=\Delta u_\infty+f(u_\infty)$ in $\R\times\R^N$, with $0\le u_\infty\le1$ in~$\R\times\R^N$, and $u_\infty(0,0,0)=1$. The strong parabolic maximum principle and the uniqueness of the bounded solutions of the Cauchy problem~\eqref{homo} imply that $u_\infty\equiv1$ in $\R\times\R^N$. In particular, $|\nabla_{\!x'}u(t_n,x'_n,y_n)|=|\nabla_{\!x'}u_n(0,0,0)|\to|\nabla_{\!x'}u_\infty(0,0,0)|=0$ as $n\to+\infty$. Since the limit (namely, $0$) does not depend on the subsequence, one concludes that the whole sequence~$(|\nabla_{\!x'}u(t_n,x'_n,y_n)|)_{n\in\N}$ converges to $0$ as $n\to+\infty$.

The previous paragraphs provide property~\eqref{limit3u} under any of the assumptions~(i)-(iv) and the proof of Theorem~\ref{thm:conical} is thereby complete.\hfill$\Box$


\subsection{Proof of Proposition~\ref{pro:asymp}}\label{sec54}

Let $N=2$. Consider a function $f$ such that Hypothesis~\ref{hyp:minimalspeed} is satisfied (hence, Hypothesis~\ref{hyp:invasion} as well), and let $\theta\in(0,1)$ be given by Hypothesis~\ref{hyp:invasion}. Let us call for short~$y$ the variable $x_2$. We consider a function $\gamma$ defined for $|x'|>1$ by
$$\gamma(x')=\sqrt{|x'|}\sin(\sqrt{|x'|}),$$
and extended in a smooth way to the whole $\R$. For $x'>1$, we compute
$$\gamma'(x')=\frac1{2\sqrt{x'}}\sin(\sqrt{x'})+\frac12\cos(\sqrt{x'}).$$
The function $\gamma$ then fulfills condition~\eqref{gamma0} but not~\eqref{gamma'to0}. Moreover, $\gamma'(x'+4\pi^2n^2)\to1/2$ as~$n\to+\infty$, locally uniformly in $x'\in\R$. As a consequence, $u_0(\.+4\pi^2n^2,\.)\to H(2y-x')$ as~$n\to+\infty$ in $L^p_{loc}(\R^2)$, for any $p\geq1$, where~$H$ is the Heaviside function:
$$H(s)=\begin{cases} 1 & \text{if }s\leq0,\\
0 & \text{if }s>0.
\end{cases}$$
Then, by parabolic estimates, $u(t,x'+4\pi^2n^2,y)$ converges as $n\to+\infty$ (up to subsequences) in $C^{1;2}_{loc}((0,+\infty)\times\R^2)$, to the solution $v$ of \eqref{homo} with initial datum $H(2y-x')$. By uniqueness, the function $v$ is of the form $v(t,x',y)=w(t,2y-x')$. Moreover, as for the $x_N$-monotonicity of $u$ with initial conditions satisfying~\eqref{defu0bis}, the comparison principle shows that $w(t,z)$ is nonincreasing with respect to $z$, and the strong maximum principle applied to $\partial_{z} w$ implies that $\partial_{z} w<0$ in $(0,+\infty)\times\R^2$.

Fix now any $\lambda\in(0,1)$, and consider an arbitrary $t>0$. Let $z_t\in\R$ be such that~$w(t,z_t)=\lambda$ (as in~\eqref{defX}, such $z_t$ exists and is unique because the function~$w(t,\cdot)$ is conti\-nuous and decreasing, and $w(t,-\infty)=1$ and~$w(t,+\infty)=0$). We see that~$v(t,0,z_t/2)=\lambda$ and
$$\frac{\partial_{1} v(t,0,z_t/2)}{\partial_{2} v(t,0,z_t/2)}=\frac{-\partial_{z} w(t,z_t)}{2\,\partial_{z} w(t,z_t)}=-\frac12.$$
As a consequence, there holds from one hand that
$$\lim_{n\to+\infty}u(t,4\pi^2n^2,z_t/2)=\lambda,$$
and from the other hand that
$$\lim_{n\to+\infty}\frac{\partial_{1} u(t,4\pi^2n^2,z_t/2)}{\partial_{2} u(t,4\pi^2n^2,z_t/2)}=-\frac12.$$
Hence, owing to~\eqref{DchiDu}, one has $\partial_{x'}X_{u(t,4\pi^2n^2,z_t/2)}(t,4\pi^2n^2)\to1/2$ as~$n\to+\infty$. Therefore, for every $\lambda_0\in(\lambda,1)$ and every $t>0$, one has
$$\sup_{0<\lambda'\le\lambda_0,\,x'\in\R}|\partial_{x'}X_{\lambda'}(t,x')|\ge\frac12.$$
This shows that $u$ violates the conclusion~\eqref{Dchi0} of Conjecture \ref{conj:order1}.

Moreover, since Hypothesis~\ref{hyp:invasion} holds, \cite[Theorem~1.11]{DGM} implies that the functions $w(t,z_t+\cdot)$ converge as $t\to+\infty$ in $C^2_{loc}(\R)$ to the profile of a decreasing or constant solution connecting some values $a$ to $b$ with $1\ge a\ge\lambda\ge b\ge0$, and belonging to the minimal propagating terrace solution to~\eqref{homo} connecting $1$ to $0$. But this minimal propagating terrace reduces here to a single decreasing traveling front owing to Hypothesis~\ref{hyp:minimalspeed}. It follows in particular that  $\lim_{t\to+\infty}-\partial_zw(t,z_t)>0$. Since
$$\lim_{n\to+\infty}\partial_{1} u(t,4\pi^2n^2,z_t/2)=-\partial_z w(t,z_t)$$
for every $t>0$, conclusion~\eqref{Du0} fails too.
\hfill$\Box$

	
\section{Logarithmic lag: proofs of Theorems~\ref{thlag},~\ref{thm:normdelay}, Proposition~\ref{pro:lag}, and Corollary~\ref{cor:DGlag}}\label{sec6}

This section is devoted to the proofs of Theorems~\ref{thlag},~\ref{thm:normdelay}, Proposition~\ref{pro:lag}, and Corollary~\ref{cor:DGlag} on the lag behind the front in the Fisher-KPP case, and on further asymptotic one-dimensional symmetry results in the direction $\mathrm{e}_N$, when the initial conditions $u_0$ are of the type~\eqref{defu0bis}, with $\gamma(x')$ going to $-\infty$ suitably fast as $|x'|\to+\infty$.

	
\subsection{Proof of Theorem~\ref{thlag}}\label{sec61}

We start with deriving property~\eqref{E<U}. Consider the linear equation
\begin{equation}\label{linearised}
\partial_tw=\Delta w+f'(0) w,\quad t>0,\ x\in\R^N,
\end{equation}
with initial datum $w(0,\.)=\1_{\R^N\setminus B_r}$, for given $r>0$. The solution $w$ can be explicitly computed through the heat kernel. One has
$$w(t,0)=\frac{e^{f'(0) t}}{(4\pi t)^{N/2}}\int_{\R^N\setminus B_r}e^{-\frac{|y|^2}{4t}}dy=\frac{e^{f'(0) t}}{(4\pi t)^{N/2}}\times N|B_1|\int_{r}^{+\infty}\rho^{N-1}e^{-\frac{\rho^2}{4t}}d\rho$$
for all $t>0$, where $|B_1|$ denotes the $N$-dimensional Lebesgue measure of the unit ball $B_1$. In order to estimate the latter integral we integrate by parts and write
$$\int_{r}^{+\infty} \rho^{N-1}e^{-\frac{\rho^2}{4t}}d\rho=2te^{-\frac{r^2}{4t}}r^{N-2}+2(N-2)\int_{r}^{+\infty} t\rho^{N-3}e^{-\frac{\rho^2}{4t}}d\rho.$$
In order to estimate the latter term, call
$$R_0:=\frac{4(N-2)}{c^*}=\frac{2(N-2)}{\sqrt{f'(0)}}\ge0,$$
which is a quantity only depending on $N$ and $f'(0)$. Then, for any $t>0$ and~$R\ge R_0$, if $r\ge c^*t+R$, we~get
$$\forall\,\rho\geq r,\quad 2(N-2)t\rho^{N-3}\leq \frac12c^* t\rho^{N-2}\leq \frac12\rho^{N-1}.$$
One then deduces that, for $t>0$, $R\geq R_0$ and $r\geq c^*t+R$, 
$$\int_{r}^{+\infty} \rho^{N-1}e^{-\frac{\rho^2}{4t}}d\rho\leq 4te^{-\frac{r^2}{4t}}r^{N-2},$$
and therefore
$$w(t,0)\leq\frac{4N|B_1|}{(4\pi)^{N/2}}\,t^{(2-N)/2}\,e^{f'(0) t-\frac{r^2}{4t}}\,r^{N-2}.$$
We now take
$$r=c^*t+k\log t+R=2\sqrt{f'(0)}\,t+k\log t+R,$$
with $k\ge0$ and $R\in[R_0,+\infty)$ to be chosen. We then have that, for $t\geq 1$,
\[\begin{split}
w(t,0) &\leq
\frac{4N|B_1|}{(4\pi)^{N/2}}\,t^{(2-N)/2}\,e^{-k\sqrt{f'(0)} \log t-R\sqrt{f'(0)}}\,(2\sqrt{f'(0)}\,t+k\log t+R)^{N-2}\\
&\leq \frac{4N|B_1|(2\sqrt{f'(0)}+k+R)^{N-2}}{(4\pi)^{N/2}}\,t^{(N-2)/2-k\sqrt{f'(0)}}\,e^{-R\sqrt{f'(0)}}.
\end{split}\]
Choosing $k=(N-2)/(2\sqrt{f'(0)})=(N-2)/c^*$, we eventually infer that, for $t\geq 1$,
$$w(t,0)\leq\frac{4N|B_1|}{(4\pi)^{N/2}}\times\Big(2\sqrt{f'(0)}+\frac{N-2}{2\sqrt{f'(0)}}+R\Big)^{N-2}e^{-R\sqrt{f'(0)}}.$$
For any given $\lambda\in(0,1)$, we can then choose a positive real number $R\in[R_0,+\infty)$ large enough (depending only on $R_0$, $f'(0)$, $N$ and $\lambda$, hence on $f'(0)$, $N$ and $\lambda$) such that the above right-hand side (which is independent of $t$) is smaller than $\lambda$. Namely, with this choice of~$R$, the solution of~\eqref{linearised} with initial datum $w(0,\.)=\1_{\R^N\setminus B_r}$, with
$$r=c^*t+\frac{N-2}{c^*}\log t+R$$
and any $t\ge1$ satisfies $w(t,0)\leq\lambda$. 

We now use this function $w$ in order to estimate the solution~$u$ of~\eqref{homo}. We will show~\eqref{E<U} with the above choice of $R$, that is, $\R^N\setminus U+B_{c^*t+\frac{N-2}{c^*}\log t+R}\ \subset\ \R^N\setminus F_\lambda(t)$ for every $t\ge1$. Fix then an arbitrary $t\ge1$. Consider a point $x_t$, if any, in $\R^N\setminus U+B_{c^*t+\frac{N-2}{c^*}\log t+R}$. This means that $u_0=0$ in~$B_{c^*t+\frac{N-2}{c^*}\log t+R}(x_t)$ and thus $u_0(x+x_t)\leq w(0,x)$ for $x\in\R^N$, with the above solution $w(s,x)$ of the linear heat equation $\partial_sw=\Delta w+f'(0)w$ and initial condition $w(0,\cdot)=\1_{\R^N\setminus B_{c^*t+\frac{N-2}{c^*}\log t+R}}$. Since $w$ is a supersolution to~\eqref{homo} due to the Fisher-KPP hypothesis~\eqref{kpp0}, we infer that $u(s,x+x_t)\leq w(s,x)$ for all $s\geq0$ and $x\in\R^N$, whence in particular
$$u(t,x_t)\leq w(t,0)\leq\lambda,$$
that is, $x_t\notin F_\lambda(t)$. This shows~\eqref{E<U}.

Let us turn to property~\eqref{E>U}, under the assumption that $U_\rho\neq\emptyset$ satisfies~\eqref{dUrho}. Let~$v$ be the solution to~\eqref{homo} with initial datum $v(0,\.)=\1_{B_\rho}$. We know from~\cite{D,G,RRR} that, for given $\lambda\in(0,1)$, there exists $R>0$ such that
\begin{equation}\label{v>}
B_{c^*t-\frac{N+2}{c^*}\log t}\subset \big\{x\in\R^N:v(t,x)>\lambda\big\} + B_R \quad\text{for all }t\geq1.
\end{equation} 
By the definition of $U_\rho$ one has that $u_0\geq\1_{B_\rho(x_0)}$ for any $x_0\in U_\rho$. Owing to the parabolic comparison principle we then deduce $u(t,x)\geq v(t,x-x_0)$ for all $t\geq0$, $x\in\R^N$, and thus, by~\eqref{v>},
$$B_{c^*t-\frac{N+2}{c^*}\log t}(x_0)\subset \big\{x\in\R^N:u(t,x)>\lambda\big\} + B_R \quad\text{for all }t\geq1,$$ 
This means that $U_\rho+B_{c^*t-\frac{N+2}{c^*}\log t} \subset F_\lambda(t)+B_R  $ for all $t\geq1$, from which~\eqref{E>U} follows due to~\eqref{dUrho}, even if it means increasing $R$.
\hfill$\Box$

	
\subsection{Proofs of Theorem~\ref{thm:normdelay}, Proposition~\ref{pro:lag}, and Corollary~\ref{cor:DGlag}}\label{sec62}

\begin{proof}[Proof of Theorem~$\ref{thm:normdelay}$] 
We recall that here, being $f$ of the strong Fisher-KPP type~\eqref{fkpp}, Hypotheses~\ref{hyp:invasion} and~\ref{hyp:minimalspeed} are satisfied and $c^*=2\sqrt{f'(0)}$ is the minimal speed of traveling fronts connecting $1$ to $0$. We further have that $u_0$ satisfies~\eqref{defu0bis},~\eqref{gamma<}. By the monotonicity of the functions~$X_\lambda$ with respect to $\lambda$, it is sufficient to show that the limit~\eqref{lag1} holds for any given $\lambda\in(0,1)$, which is fixed throughout the proof.

The fact that the right-hand side of~\eqref{lag1} provides a lower bound for $X_\lambda(t,x')$ is a straightforward consequence of the results about the logarithmic lag for compactly supported initial data, see~\cite{D,G,RRR}. Indeed, we know from these works that the solution~$\ul u$ to~\eqref{homo} emerging from a continuous, compactly supported, radially symmetric and non-trivial initial datum $\ul u_0$ such that $0\leq\ul u_0\leq u_0\le 1$ in $\R^N$, satisfies the following property: there exists $\sigma\in\R$ such that, for any $K>0$, there is $T_K>0$ for which there holds 
$$\forall\,t\geq T_K,\ \forall\,|x'|\leq K, \quad\ul u\Big(t,x',c^*t-\frac{N+2}{c^*}\,\log t+\sigma\Big)>\lambda.$$
Since $1\ge u\geq\ul u\ge 0$ in $[0,+\infty)\times\R^N$ by the parabolic comparison principle, we find that
\Fi{Xl>}
X_\lambda(t,x')\geq c^*t-\frac{N+2}{c^*}\log t+\sigma+o(1)\as t\to+\infty,
\Ff
locally uniformly in $x'\in\R^{N-1}$.
	
In order to show the upper bound, we will construct a supersolution $v$ larger than~$u$ at time $0$, for which we are able to explicitly compute the lag. First of all, owing to~\eqref{gamma<}, we can take $\beta<0$ satisfying
$$\limsup_{|x'|\to+\infty}\,\frac{\gamma(x')}{\log(|x'|)}<\beta<-\frac{2(N-1)}{c^*}.$$
We then take $M>0$ large enough so that
$$\forall\,x'\in\R^{N-1},\quad\gamma(x')\le\beta\log(1+|x'|)+M.$$
Hence, by the parabolic comparison principle, if we show the desired upper bound for $X_\lambda$ when $\gamma(x')$ is replaced by $\beta\log(1\!+\!|x'|)\!+\!M$, we are done. Up to a translation of the coordinate system, we can further assume that $M=0$. We then assume from now on~that
$$\gamma(x')=\beta\log(1+|x'|).$$
In particular, since the above function $\gamma$ is globally Lipschitz continuous, we can find a radius $\delta>0$ large enough, depending on $N$ and the Lipschitz constant of $\gamma$, such that 
\be\label{cupcup}
\big\{(x',x_N)\in\R^{N-1}\times\R : x_N\leq\gamma(x')\big\}\,\subset\,\bigcup_{k\in\Z^{N-1}}\,\bigcup_{h\in\N}B_\delta(k,\gamma(k)-h)
\ee
(we denote $\N$ the set if integers, including $0$). We then consider the solution $0\le w\le 1$ of~\eqref{homo} emerging from a $C^\infty$ compactly supported and radially symmetric initial datum~$w_0$ such that
$$\1_{{B_{\delta}}}\leq w_0\leq 1\ \hbox{ in $\R^N$},$$
and we define a nonnegative function $v$ in $[0,+\infty)\times\R^N$ by
\be\label{eqsum}
v(t,x)=v(t,x',x_N):=\sum_{(k,h)\in\Z^{N-1}\times\N}w(t,x'-k,x_N-\gamma(k)+h).
\ee
From Gaussian estimates, for any $T>0$, there are some positive constants $\alpha_T$ and $C_T$ such that $0\le w(t,x)\le C_Te^{-\alpha_T|x|^2}$ and $|\partial_tw(t,x)|+|\partial_{x_i}w(t,x)|+|\partial_{x_ix_j}w(t,x)|\le C_Te^{-\alpha_T|x|^2}$ for all $(t,x)\in[0,T]\times\R^N$ and $1\le i,j\le N$. Therefore, the function $v$ is well defined in $[0,+\infty)\times\R^N$ and of class~$C^{1;2}_{t;x}([0,+\infty)\times\R^N)$. Furthermore, since the function $f$ is at least Lipschitz continuous in $[0,1]$ and satisfies~\eqref{fkpp}, then, for any series $\sum a_j$ of nonnegative real numbers such that $\sum_{j\in\N}a_j\le1$, the series $\sum f(a_j)$ converges and $f(\sum_{j\in\N}a_j)\le\sum_{j\in\N}f(a_j)$. It then follows that $\min(v,1)$ is a (generalized) supersolution of~\eqref{homo} in $[0,+\infty)\times\R^N$. For any $(x',x_N)\in B_\delta(k,\gamma(k)-h)$, with $k\in\Z^{N-1}$ and $h\in\N\cup\{0\}$, there holds that $v(0,x',x_N)\geq w_0(x'-k,x_N-\gamma(k)+h)=1$, whence $v(0,\.)\geq u_0$ due to~\eqref{cupcup}. The parabolic comparison principle then implies that
\be\label{inequv1}
0\le u\leq\min(v,1)\ \hbox{ in $[0,+\infty)\times\R^N$}.
\ee

Let us estimate the position of the level sets of $v$. Let $\varphi^*$ denote the profile~of~a traveling front connecting $1$ to $0$ with minimal speed~$c^*=2\sqrt{f'(0)}$. We know that $\varphi^*:\R\to\R$ is a decreasing function satisfying 
$$\lim_{r\to+\infty}\,\frac{\varphi^*(r)}{r\,e^{-c^*r/2}}\in(0,+\infty),$$
hence there exists a constant $A>0$ such that 
\be\label{varphi*}
\varphi^*(r)\leq A\,r\,e^{-c^*r/2}\ \hbox{ for all $r\geq1$}.
\ee
It is also known~\cite{D,G,RRR}~that 
$$w(t,x)-\varphi^*\big(|x|-\rho(t)\big)\to0\ \text{ as }t\to+\infty,$$
uniformly with respect to $x\in\R^N$, where $\rho(t)$ satisfies
\be\label{def:rhot}
\rho(t)=c^*t-\frac{N+2}{c^*}\,\log t+O(1)\as t\to+\infty.
\ee
Moreover, by \cite[Lemma 3.5]{D}, there is a positive constant $C>0$ such that
\be\label{x<sqrtt}
w(t,x)\leq C\,\varphi^*(|x|-\rho(t))\quad \text{for all $t\ge1$ and $|x|-\rho(t)\leq\sqrt{t}$},
\ee
and, by the proof of \cite[Lemma 3.2]{D}, there is another constant $b>0$ such that
\be\label{x>sqrtt}
w(t,x)\leq C\,(|x|-\rho(t))^b\,\varphi^*(|x|-\rho(t))\quad \text{for all $t\ge1$ and $|x|-\rho(t)>\sqrt{t}$}.
\ee
	
We now use the above bounds to estimate the value of the function $v$ defined by~\eqref{eqsum} at time~$t$ and position $(0,\rho(t)+y)$. For this, we call for short 
$$D_{t,y}(h,k):=|(-k,\rho(t)+y-\gamma(k)+h)|-\rho(t)=|(0,\rho(t)+y)-(k,\gamma(k)-h)|-\rho(t).$$
We derive from \eqref{x<sqrtt}-\eqref{x>sqrtt} that, for all $t\ge1$ and $y\in\R$,
\be\label{upperv}\baa{rcl}
v(t,0,\rho(t)+y) & \leq & \displaystyle\ C \underbrace{\sum_{(k,h)\in\Z^{N-1}\times \N\,:\, D_{t,y}(h,k)\leq\sqrt{t}}\varphi^*(D_{t,y}(h,k))}_{=:I_1(t,y)}\\
& & \displaystyle+\,C \underbrace{\sum_{(k,h)\in\Z^{N-1}\times \N\,:\, D_{t,y}(h,k)>\sqrt{t}}(D_{t,y}(h,k))^b\,\varphi^*(D_{t,y}(h,k))}_{=:I_2(t,y)}.\eaa
\ee

Let us first evaluate the quantity $I_1(t,y)$. From~\eqref{def:rhot}, let $\t t\ge1$ be such that $\rho(t)\ge0$ for all $t\ge\t t$. Using the fact $\gamma\le 0$ in $\Z^{N-1}$ (since $\beta<-2(N-1)/c^*<0$), we infer that, for every $t\ge\t t$ and~$y\geq 1$, one has $D_{t,y}(h,k)\ge1$ for all $(k,h)\in\Z^{N-1}\times\N$, whence, by~\eqref{varphi*},
$$I_1(t,y)\leq A\sum_{(k,h)\in\Z^{N-1}\times\N}\big(|(-k,\rho(t)+y-\gamma(k)+h)|-\rho(t)\big)\,e^{-c^*(|(-k,\rho(t)+y-\gamma(k)+h)|-\rho(t))/2}.$$
Because $r\mapsto r e^{-c^*r/2}$ is decreasing for $r\geq2/c^*$, we then deduce that, for every $t\ge \t t$  and every $y\geq \max(1,2/c^*)$ (recall that $\gamma\leq0$),
\[\begin{split}
I_1(t,y) &\leq\  A\sum_{(k,h)\in\Z^{N-1}\times\N}(y-\gamma(k)+h)\,e^{-c^*(y-\gamma(k)+h)/2}\\
&=\ A\,\Big(\sum_{k\in\Z^{N-1}}e^{c^*\gamma(k)/2}\Big)\Big(\sum_{h=0}^\infty(y+h)\,e^{-c^*(y+h)/2}\Big)\\
& \ \ \ \ +A\,\Big(\!\sum_{k\in\Z^{N-1}}|\gamma(k)|\,e^{c^*\gamma(k)/2}\Big)\Big(\sum_{h=0}^\infty e^{-c^*(y+h)/2}\Big).
\end{split}\]
As a consequence, calling
\be\label{defC12}
C_1:=\sum_{h=0}^\infty e^{-c^*h/2}=\frac{1}{1-e^{-c^*/2}}\ \hbox{ and }\ C_2:=\sum_{h=0}^\infty h e^{-c^*h/2}=\frac{e^{-c^*/2}}{(1-e^{-c^*/2})^2},
\ee
we find that, for every $t\ge\t t$ and every $y\geq \max(1,2/c^*)$,
\[I_1(t,y) \leq A(C_1y+C_2)e^{-c^*y/2}\sum_{k\in\Z^{N-1}}e^{c^*\gamma(k)/2}+AC_1 e^{-c^*y/2}\sum_{k\in\Z^{N-1}}|\gamma(k)|e^{c^*\gamma(k)/2}.\]
Let us study these series in $k$. Recalling that $\gamma(x')=\beta\log(1+|x'|)$, we compute
$$\sum_{k\in\Z^{N-1}}|\gamma(k)|e^{c^*\gamma(k)/2}=|\beta|\sum_{k\in\Z^{N-1}}(1+|k|)^{c^*\beta/2}\log(1+|k|).$$
We now use the fact that, for any pair of nonnegative functions $p,q:\R\to\R$, with $p$ nonincreasing and~$q$ nondecreasing, there holds that
$$\forall\,k\in\Z^{N-1},\quad p(|k|)\,q(|k|)\leq\int_{k+(0,1)^{N-1}}p(|x'|-\sqrt{N-1})\,q(|x'|+\sqrt{N-1})\,dx',$$
and therefore, for any measurable set $\mathcal{A}\subset\R^{N-1}$, we get
\Fi{sumint}
\sum_{k\in\Z^{N-1}\cap\mathcal{A}}p(|k|)\,q(|k|)\leq\int_{\mathcal{A}+B'_{\sqrt{N-1}}}p(|x'|-\sqrt{N-1})\,q(|x'|+\sqrt{N-1})\,dx'.
\Ff
By using $p(r)=(1+r^+)^{c^*\beta/2}$ and $q(r)=\log(1+r^+)$, this allows us to estimate 
\[\begin{split}
\sum_{k\in\Z^{N-1}\setminus B'_{2\sqrt{N-1}}}\!\!(1+|k|)^{c^*\beta/2}&\log(1+|k|)\\
&\leq\int_{|x'|\geq\sqrt{N-1}}(1+|x'|-\sqrt{N-1})^{c^*\beta/2}\log(1+|x'|+\sqrt{N-1})\,dx',
\end{split}\]
which is finite because $\beta<-2(N-1)/c^*$. This shows that $\sum_{k\in\Z^{N-1}}|\gamma(k)|e^{c^*\gamma(k)/2}$ converges, as well as~$\sum_{k\in\Z^{N-1}}e^{c^*\gamma(k)/2}$ (since $|\gamma(k)|\to+\infty$ as $|k|\to+\infty$). It follows that there exists a constant $C'>0$ such that
\be\label{I1}
\forall\,t\ge\t t,\ \forall\,y\geq \max(1,2/c^*),\ \ I_1(t,y)\leq C'y\,e^{-c^*y/2}.
\ee

We then estimate the term $I_2(t,y)$ in~\eqref{upperv}, for $t\ge1$ and $y\in\R$, using that
\be\label{I21}
\sum_{\substack{(k,h)\in\Z^{N-1}\times\N\,:\\ D_{t,y}(h,k)>\sqrt{t}}}\!\!\!\!\!\!\!(D_{t,y}(h,k))^b\,\varphi^*(D_{t,y}(h,k))\!\leq \!\!\!\sum_{\substack{n\in\N\,:\\ n\geq\sqrt{t}-1}}\,\sum_{\substack{(k,h)\in\Z^{N-1}\times\N\,:\\ n<D_{t,y}(h,k)\leq n+1}}\!\!\!\!\!\!\!\!(D_{t,y}(h,k))^b\,\varphi^*(D_{t,y}(h,k)).
\ee
Observing that the distance between any pair of distinct points of the form $(k,\gamma(k)-h)$ is larger than or equal to $1$, one infers that the ones satisfying $n<D_{t,y}(h,k)\leq n+1$ are at most $\Gamma_N\,n^{N-1}$, for some constant $\Gamma_N>0$ depending only on $N$. Together with the fact that $\varphi^*$ is decreasing, 
this yields
\be\label{I22}
0\le I_2(t,y)=\!\!\sum_{\substack{(k,h)\in\Z^{N-1}\times\N\,:\\ D_{t,y}(h,k)>\sqrt{t}}}\!\!(D_{t,y}(h,k))^b\,\varphi^*(D_{t,y}(h,k))\leq\Gamma_N\!\!\!\!\sum_{\substack{n\in\N\,:\\ n\geq\sqrt{t}-1}}n^{N-1}(n+1)^b\,\varphi^*(n),
\ee
from which, using~\eqref{varphi*}, one derives
\be\label{I2}
I_2(t,y)\to0\hbox{ as $t\to+\infty$ uniformly in $y\in\R$}.
\ee

Because $u\leq v$ by~\eqref{inequv1}, we eventually deduce from~\eqref{def:rhot},~\eqref{upperv},~\eqref{I1} and~\eqref{I2} that there are some positive constants $\t C$ and $\t y$ such that
\be\label{upperu}
\limsup_{t\to+\infty}\Big[\sup_{y\ge\t y}\Big(u\Big(t,0,c^*t-\frac{N+2}{c^*}\log t+y\Big)-\t C y\,e^{-c^*y/2}\Big)\Big]\le0.
\ee
Therefore, by the definition of $X_\lambda(t,0)$, there are some positive constants $C^*$ and $t^*$ such~that
$$X_\lambda(t,0)\leq c^*t-\frac{N+2}{c^*}\log t+C^*\ \text{ for all $t\ge t^*$}.$$
Finally, since $\gamma(x')=\beta\log(1+|x'|)$ is radially symmetric and decreasing (remember that~$\beta$ is here negative), it follows from a standard reflection argument with respect to hyperplanes parallel to $\mathrm{e}_N$ (similarly as in the proof of \thm{conical}) that, for every $t>0$ and~$x_N\in\R$, the function $x'\mapsto u(t,x',x_N)$ is radially symmetric and decreasing with respect to $|x'|$. We then deduce from the above estimate that
\Fi{Xl<}
X_\lambda(t,x')\leq c^*t-\frac{N+2}{c^*}\log t+C^*\ \text{ for all $x'\in\R^{N-1}$ and $t\ge t^*$,}
\Ff
which, together with~\eqref{Xl>}, completes the proof of Theorem~\ref{thm:normdelay}.
\end{proof}

The arguments employed in the above proof of the upper bound for $X_\lambda$ can be adapted to the case of functions $\gamma$ satisfying the logarithmic upper bound~\eqref{hyplag}, as we now show.

\begin{proof}[Proof of Proposition~$\ref{pro:lag}$] 
Throughout the proof, $\lambda$ is any fixed real number in $(0,1)$. Suppose firstly that $\gamma$ is precisely given by
\be\label{defgamma2}
\gamma(x')=\frac{2\sigma}{c^*}\,\log(1+|x'|)\quad \text{for all $x'\in\R^{N-1}$},
\ee
with $\sigma>-(N-1)$. Consider the same function $w$ as in the proof of Theorem~$\ref{thm:normdelay}$, depending on the parameter $\delta>0$. This function fulfills~\eqref{x<sqrtt}-\eqref{x>sqrtt} with $\rho(t)$ satisfying~\eqref{def:rhot}, for some positive constant~$C$. Moreover, since $\gamma$ is globally Lipschitz continuous, it fulfills condition~\eqref{cupcup} for $\delta>0$ sufficiently large, and thus, considering $w$ associated with such a value of $\delta$ and then defining~$v$ as in~\eqref{eqsum}, one has by comparison that~\eqref{inequv1} holds. Take an arbitrary quantity $\beta$ satisfying
\be\label{choicebeta}
\beta>\frac{\sigma+N-1}{c^*}>0.
\ee
Our aim is to show that
\be\label{claimv}
v(t,0,\rho(t)+\beta\log t)\to0\ \hbox{ as $t\to+\infty$}.
\ee

Let us postpone for a moment the proof of~\eqref{claimv} and conclude the argument. Together with~\eqref{inequv1}, this will imply that $u(t,0,\rho(t)+\beta\log t)\to0$ as $t\to+\infty$, hence~$X_\lambda(t,0)\le\rho(t)+\beta\log t$ for all $t$ large enough, and then by~\eqref{def:rhot},
$$\limsup_{t\to+\infty}\frac{X_\lambda(t,0)-c^*t}{\log t}\le\beta-\frac{N+2}{c^*}.$$
Since $\gamma$ is given by~\eqref{defgamma2}, we infer from~\thm{conical} with assumption~(iv) that the above estimate holds true for $X_\lambda(t,x')$, locally uniformly with respect to $x'\in\R^{N-1}$, and then~\eqref{ineqlag} follows from the arbitrariness of $\beta$ in~\eqref{choicebeta}. If we now consider a general~$\gamma$ satisfying~\eqref{hyplag} with $\sigma\ge-(N-1)$, we take an arbitrary $\sigma'>\sigma$ and then, since $\gamma$ satisfies $\gamma(x')<(2\sigma'/c^*)\log(1+|x'|)$ for $x'\in\R^{N-1}$ up  to an additive constant, we deduce from what precedes and the comparison principle, that~\eqref{ineqlag} holds with $\sigma$ replaced by $\sigma'$, locally uniformly with respect to $x'\in\R^{N-1}$. This gives the conclusion of the proposition, owing to the arbitrariness of~$\sigma'\in(\sigma,+\infty)$.

So, we are left to prove that~\eqref{claimv} holds with $\beta$ and $\sigma$ as in~\eqref{choicebeta}, when $\gamma$ is given by~\eqref{defgamma2}. Take $\alpha\in(1/2,1)$ close enough to $1/2$ in such a way that
\be\label{choicealpha}
\frac{2\alpha\sigma}{c^*}\leq\frac{2\alpha(\sigma+N-1)}{c^*}<\beta.
\ee
For every $t>0$, let us compute
\be\label{sum1}
0\le v(t,0,\rho(t)+\beta\log t)=\sum_{(k,h)\in\Z^{N-1}\times\N}w(t,-k,\rho(t)+\beta\log t-\gamma(k)+h).
\ee
We call for short 
$$E_{t}(h,k):=|(-k,\rho(t)+\beta\log t-\gamma(k)+h)|-\rho(t)=|(0,\rho(t)+\beta\log t)-(k,\gamma(k)-h)|-\rho(t).$$
We use~\eqref{x<sqrtt}-\eqref{x>sqrtt} and divide the sum~\eqref{sum1} into three subsums to get, for every $t\ge1$:
\be\label{sums}\baa{rcl}
0\,\le\,v(t,0,\rho(t)+\beta\log t) & \le & C\displaystyle\sum_{(k,h)\in(\Z^{N-1}\cap B'_{t^\alpha})\times\N\,:\,E_t(h,k)\le\sqrt{t}}\varphi^*(E_t(h,k))\\
& & \displaystyle+\,C\sum_{(k,h)\in(\Z^{N-1}\setminus B'_{t^\alpha})\times\N\,:\,E_t(h,k)\le\sqrt{t}}\varphi^*(E_t(h,k))\\
& & \displaystyle+\,C\sum_{(k,h)\in\Z^{N-1}\times\N\,:\,E_t(h,k)>\sqrt{t}}(E_t(h,k))^b\varphi^*(E_t(h,k))\\
& \le & C\displaystyle\underbrace{\sum_{(k,h)\in(\Z^{N-1}\cap B'_{t^\alpha})\times\N}\varphi^*(E_t(h,k))}_{=:J_1(t)}\\
& & \displaystyle+\,C\underbrace{\sum_{(k,h)\in(\Z^{N-1}\setminus B'_{t^\alpha})\times\N}\varphi^*(E_t(h,k))}_{=:J_2(t)}\\
& & \displaystyle+\,C\underbrace{\sum_{(k,h)\in\Z^{N-1}\times\N\,:\,E_t(h,k)>\sqrt{t}}(E_t(h,k))^b\varphi^*(E_t(h,k))}_{=:J_3(t)}.\eaa\ee

Let us first deal with the sum $J_1(t)$. Recalling that $\gamma$ is given in~\eqref{defgamma2}, one has
$$\forall\, t>0,\ \forall\, k\in\Z^{N-1}\cap B'_{t^\alpha},\quad\beta\log t-\gamma(k)\geq
\beta\log t-\frac{2\sigma}{c^*}\log(1+t^\alpha)=\frac{2\sigma}{c^*}\log\frac{t^{\frac{\beta c^*}{2\sigma}}}{1+t^\alpha}.$$
Hence, by~\eqref{choicealpha} and the positivity of $\alpha$, together with~\eqref{def:rhot}, there is~$t_1\geq1$ large enough such that $\rho(t)\ge0$ and $E_t(h,k)\geq\beta\log t-\gamma(k)+h\ge\max(1,2/c^*)$ for all $t\ge t_1$ and $(k,h)\in(\Z^{N-1}\cap B'_{t^\alpha})\times\N$. Therefore, we can use the estimate~\eqref{varphi*} in the expression of $J_1(t)$ for $t\ge t_1$, which, by the monotonicity of the function $r\mapsto r^{-c^*r/2}$ in $[2/c^*,+\infty)$, yields 
\be\label{J1t}\baa{rcl}
0\,\le\,J_1(t) & \!\!\!\le\!\!\! & \displaystyle A\sum_{k\in\Z^{N-1}\cap B'_{t^\alpha}}\sum_{h=0}^\infty(\beta\log t-\gamma(k)+h)\,e^{-c^*(\beta\log t-\gamma(k)+h)/2}\\
& \!\!\!\le\!\!\! & \displaystyle A\,(C_1\beta\log t+C_2)\,t^{-c^*\beta/2}\sum_{k\in\Z^{N-1}\cap B'_{t^\alpha}}e^{c^*\gamma(k)/2}\\
& \!\!\!\!\!\! & \displaystyle+\,A\,C_1\,t^{-c^*\beta/2}\sum_{k\in\Z^{N-1}\cap B'_{t^\alpha}}|\gamma(k)|\,e^{c^*\gamma(k)/2}\\
& \!\!\!\le\!\!\! & \displaystyle A\left(C_1\beta\log t+C_2+C_1\,\frac{2|\sigma|}{c^*}\log(1+t^\alpha)\right)
\,t^{-c^*\beta/2}\!\!\sum_{k\in\Z^{N-1}\cap B'_{t^\alpha}}\!\!\!\!(1+|k|)^\sigma,
\eaa\ee
where we have used the expression~\eqref{defgamma2} of $\gamma$, and $C_1,C_2$ are given in~\eqref{defC12}. We now estimate the above sum using~\eqref{sumint}, which, we recall, holds for any measurable set $\mc{A}$ and any nonnegative, nonincreasing function $p$ and nonnegative, nondecreasing function~$q$. We here use it with $p(s)=1$ and $q(s)=(1+s^+)^\sigma$ if $\sigma\ge0$, and with $p(s)=(1+s^+)^\sigma$ and~$q(s)=1$ if $\sigma<0$. We get, for some $C_N>0$ and with ``$\pm$'' in accordance with the sign of $\sigma$,
\begin{align*}
\sum_{k\in\Z^{N-1}\cap (B'_{t^\alpha}\setminus B'_{2\sqrt{N-1}})}(1+|k|)^\sigma &\le C_N\int_{\sqrt{N-1}}^{t^\alpha+\sqrt{N-1}}r^{N-2}(1+r\pm\sqrt{N-1})^{\sigma}\,dr\\
&\leq C_N\int_{0}^{t^\alpha+2\sqrt{N-1}}(r+\sqrt{N-1})^{N-2}(1+r)^{\sigma}\,dr\\
&\leq C_N(N-1)^{N/2-1}\int_{0}^{t^\alpha+2\sqrt{N-1}}(1+r)^{\sigma+N-2}\,dr.
\end{align*}
As a consequence, calling $C':=\sum_{k\in\Z^{N-1}\cap B'_{2\sqrt{N-1}}}(1+|k|)^\sigma$, since $\sigma+N-1>0$, we find 
$$\sum_{k\in\Z^{N-1}\cap B'_{t^\alpha}}(1+|k|)^\sigma\le\frac{C_N(N-1)^{N/2-1}}{\sigma+N-1}\,
(1+t^\alpha+2\sqrt{N-1})^{\sigma+N-1}+C'.$$
Together with~\eqref{choicealpha} and~\eqref{J1t}, one concludes that $J_1(t)\to0$ as $t\to+\infty$.

Let us then deal with the second sum $J_2(t)$ in~\eqref{sums}. For each $t\ge t_1\ge1$ and each $(k,h)\in(\Z^{N-1}\setminus B'_{t^\alpha})\times\N$, one has $\rho(t)\ge0$ and
\be\label{ineq61}\baa{l}
|(-k,\rho(t)+\beta\log t-\gamma(k)+h)|-\rho(t)\vspace{3pt}\\
\qquad\qquad\qquad\qquad=\displaystyle\frac{|k|^2+(\beta\log t-\gamma(k)+h)^2+2\rho(t)(\beta\log t-\gamma(k)+h)}{|(-k,\rho(t)+\beta\log t-\gamma(k)+h)|+\rho(t)}\vspace{3pt}\\
\qquad\qquad\qquad\qquad\ge\displaystyle\frac{|k|^2+h^2-2h\gamma(k)-2\gamma(k)(\beta\log t+\rho(t))}{|(-k,\rho(t)+\beta\log t-\gamma(k)+h)|+\rho(t)},\eaa
\ee
since $(\beta\log t)^2+\gamma(k)^2+2\beta h\log t+2\beta\rho(t)\log t+2h\rho(t)\ge0$ (remember that $\beta>0$ by~\eqref{choicebeta}). In order to estimate the numerator, we use the facts that $t_1\leq t\le|k|^{1/\alpha}$, with $0<1/\alpha<2$, and that $\rho(t)\sim c^*t$ as $t\to+\infty$ and~$|\gamma(k)|=O(\log|k|)$ as $|k|\to+\infty$. We infer the existence of some $t_2\ge t_1$ such that, for every~$t\ge t_2$ and $(k,h)\in(\Z^{N-1}\setminus B'_{t^\alpha})\times\N$,
\be\label{ineq62}
|k|^2+h^2-2h\gamma(k)-2\gamma(k)(\beta\log t+\rho(t))\ge\frac{|k|^2+h^2}{2}\ge0.
\ee
Using the same estimates and $1<1/\alpha$, one gets for the denominator, 
\be\label{ineq63}\baa{rcl}
0<|(-k,\rho(t)+\beta\log t-\gamma(k)+h)|+\rho(t) & \leq & {|k|+\beta\log t+|\gamma(k)|+h+2\rho(t)}\vspace{3pt}\\
& \le & 3c^*|k|^{1/\alpha}+h\le6c^*|(k,h)|^{1/\alpha},\eaa
\ee
for all $t$ larger than some $t_3\geq t_2$ and for all $(k,h)\in(\Z^{N-1}\setminus B'_{t^\alpha})\times\N$. Gathering the estimates~\eqref{ineq61}-\eqref{ineq63} and recalling that $1<1/\alpha<2$, one has that, for $t\geq t_3$ and $(k,h)\in(\Z^{N-1}\setminus B'_{t^\alpha})\times\N$, 
$$E_t(h,k)=|(-k,\rho(t)+\beta\log t-\gamma(k)+h)|-\rho(t)\ge\frac{|(k,h)|^{2-1/\alpha}}{12 c^*},$$
which is larger than $\max(1,2/c^*)$ for $t$ larger than some $t_4\geq t_3$, since $|k|\ge t^\alpha$. One eventually gets from~\eqref{varphi*} that, for every $t\ge t_4$,
$$\baa{rcl}
J_2(t) & \le & \displaystyle A\sum_{(k,h)\in(\Z^{N-1}\setminus B'_{t^\alpha})\times\N}\Big[\big(|(-k,\rho(t)+\beta\log t-\gamma(k)+h)|-\rho(t)\big)\\
& & \qquad\qquad\qquad\qquad\qquad\qquad\qquad\times\ e^{-c^*(|(-k,\rho(t)+\beta\log t-\gamma(k)+h)|-\rho(t))/2}\Big]\\
& \le & \displaystyle \frac A{12 c^*}\sum_{(k,h)\in(\Z^{N-1}\setminus B'_{t^\alpha})\times\N}|(k,h)|^{2-1/\alpha}\,e^{-(1/24)|(k,h)|^{2-1/\alpha}}\vspace{3pt}\\
& \le & \displaystyle\frac A{12 c^*}\sum_{(k,h)\in(\Z^{N-1}\setminus B'_{t^\alpha})\times\Z}|(k,h)|^{2-1/\alpha}\,e^{-(1/24)|(k,h)|^{2-1/\alpha}}.\eaa$$
We then use an estimate of the type~\eqref{sumint} in dimension $N$, and the inequality $2-1/\alpha>0$, to infer that $J_2(t)\to0$ as~$t\to+\infty$.

Let us finally deal with the last sum $J_3(t)$ in~\eqref{sums}. This sum $J_3(t)$ can actually be estimated similarly as $I_2(t,y)$ in~\eqref{I21}-\eqref{I22}, and one gets that $J_3(t)\to0$ as $t\to+\infty$.

As a conclusion, $v(t,0,\rho(t)+\beta\log t)\to0$ as $t\to+\infty$. The claim~\eqref{claimv} has been shown, and, as already emphasized, this completes the proof of Proposition~\ref{pro:lag}.
\end{proof}

\begin{proof}[Proof of Corollary~$\ref{cor:DGlag}$]
(i) On the one hand, we know from~\eqref{c<c*} that
$$X_\lambda(t,x')\to+\infty\ \hbox{ as }t\to+\infty,$$
locally uniformly in $\lambda\in[0,1)$ and $x'\in \R^{N-1}$. On the other hand, it follows from~\cite[Theorem~7.2]{HR2} that
$$\nabla_{\!x'}u(t,x',x_N)\to0\ \hbox{ as }t\to+\infty,\text{ locally in $x'\in\R^{N-1}$ and uniformly in $x_N\in\R$}.$$
Then, owing to~\eqref{DchiDu}, in order to show that~\eqref{limit2} holds locally uniformly in $\lambda\in(0,1)$, it is sufficient to derive a lower bound on $|\partial_{x_N}u|$ on the level sets. Namely, if we assume that~\eqref{limit2} does not hold locally uniformly in $\lambda\in(0,1)$, then there necessarily exist a sequence $(\lambda_n)_{n\in\N}$ contained in some interval $[\ul\lambda,\ol\lambda]$ with $0<\ul\lambda<\ol\lambda<1$, a sequence $(t_n)_{n\in\N}$ in $(0,+\infty)$ diverging to $+\infty$ and a bounded sequence $(x'_n)_{n\in\N}$ in $\R^{N-1}$ such that
$$\partial_{x_N}u\big(t_n,x'_n,X_{\lambda_n}(t_n,x'_n)\big)\to0\as n\to+\infty.$$
Since the function $\partial_{x_N} u$ is a negative solution of a linear parabolic equation in $(0,+\infty)\times\R^N$, it readily follows from the strong maximum principle and parabolic estimates, as in the proof of Proposition~\ref{cor:flatteningKPP}, that
$$\partial_{x_N}u\big(t_n,x',x_N+X_{\lambda_n}(t_n,x'_n)\big)\to0\ \hbox{ as }n\to+\infty,\text{ locally uniformly in $(x',x_N)\in\R^N$}.$$
Writing 
$$\ul\lambda-\ol\lambda=\int_{X_{\ol\lambda}(t_n,x'_n)}^{X_{\ul\lambda}(t_n,x'_n)}\partial_{x_N}u(t_n,x'_n,x_N)\,dx_N$$
and observing that $X_{\ol\lambda}(t_n,x'_n)\le X_{\lambda_n}(t_n,x'_n)\le X_{\ul\lambda}(t_n,x'_n)$, one deduces from the above convergence that $X_{\ul\lambda}(t_n,x'_n)-X_{\ol\lambda}(t_n,x'_n)\to+\infty$ as $n\to+\infty$. This is impossible because $X_{\ul\lambda}(t,x')-X_{\ol\lambda}(t,x')$ is bounded uniformly in $t$ large enough and locally in $x'$ thanks to Theorem~\ref{thm:normdelay}. We have reached a contradiction, and the desired property follows.
	
(ii) By standard parabolic estimates, for given $\lambda\in(0,1)$ and $x_0'\in\R^{N-1}$ and any sequence $(s_n)_{n\in\N}$ diverging to $+\infty$, the limit
$$\t u(t,x',x_N):=\lim_{n\to+\infty}u(s_n+t,x',X_\lambda(s_n,x'_0)+x_N),$$
exists (up to subsequences) in $C^{1;2}_{loc}(\R\times\R^N)$. We know from the statement (i) above that~$\t u(t,x',x_N)$ is independent of $x'$, i.e., $\t u=\t u(t,x_N)$. We apply the estimates derived in the proof of Theorem~\ref{thm:normdelay}. Namely, by~\eqref{Xl>},~\eqref{Xl<}, for any $\eta\in(0,1)$, there exist $C_\eta>0$ such that, for any $x'\in\R^{N-1}$,
$$\Big|X_\eta(t,x')-\Big(c^*t-\frac{N+2}{c^*}\log t\Big)\Big|\leq C_\eta+o(1) \as t\to+\infty.$$ 
We deduce that, for any $\eta_1,\eta_2\in(0,1)$, any $t\in\R$, and any $x'_1,x'_2\in\R^{N-1}$,
\begin{align*}
|X_{\eta_1}(s,x'_1)+c^*t-X_{\eta_2}(s+t,x'_2)| &\leq\frac{N+2}{c^*}\,|\log s-\log(s+t)|+C_{\eta_1}+C_{\eta_2}+o(1)\\
& \leq C_{\eta_1}+C_{\eta_2}+o(1)\as s\to+\infty.
\end{align*}
It follows, for any $\eta\in(0,1)$ and any $t\in\R$, $x'\in\R^{N-1}$, from the one hand that
\begin{align*}
\t u(t,c^*t+C_{\lambda}+C_{\eta}+1) &=\lim_{n\to+\infty}u(s_n+t,x',X_\lambda(s_n,x'_0)+c^*t+C_{\lambda}+C_{\eta}+1)\\
&\leq \lim_{n\to+\infty}u(s_n+t,x',X_\eta(s_n+t,x'))=\eta,
\end{align*} 
and from the other hand that
$$\t u(t,c^*t-C_{\lambda}-C_{\eta}-1)\geq \lim_{n\to+\infty}u(s_n+t,x',X_\eta(s_n+t,x'))=\eta.$$
Owing to the arbitrariness of $\eta\in(0,1)$, and the fact that $\t u$ is nonincreasing with respect to~$x_N$ (as so is~$u$), the proof of Corollary~\ref{cor:DGlag} is complete.
\end{proof}


\end{document}